\documentclass[10pt]{amsart}

\usepackage{amsmath,amssymb,amsthm,amsfonts,mathrsfs}
\usepackage{amscd,stmaryrd}
\usepackage[usenames,dvipsnames]{color}
\usepackage{hyperref}
\usepackage{graphicx,subfigure}
\definecolor{nicegreen}{RGB}{0,180,0}
\hypersetup{colorlinks=true,citecolor=nicegreen,linkcolor=Red,urlcolor=blue,pdfstartview=FitH,
pdfauthor=Koziol}
\usepackage[divide={2.45cm,*,2.45cm}]{geometry}
\usepackage{enumerate}
\usepackage{color}
\usepackage[all]{xy}
\usepackage{tikz-cd}
\usepackage{wasysym}

\allowdisplaybreaks

\newtheorem{thm}{Theorem}[section]
\newtheorem*{thm*}{Theorem}
\newtheorem{cor}[thm]{Corollary}
\newtheorem{lemma}[thm]{Lemma}

\newtheorem{propn}[thm]{Proposition}
\newtheorem*{propn*}{Proposition}

\theoremstyle{definition}
\newtheorem{defn}[thm]{Definition}
\newtheorem{rmk}[thm]{Remark}

\newcommand{\qp}{\mathbb{Q}_p}

\newcommand{\qpb}{\overline{\mathbb{Q}}_p}
\newcommand{\fpb}{\overline{\mathbb{F}}_p}

\newcommand{\zp}{\mathbb{Z}_p}
\newcommand{\T}{\textnormal{T}}
\newcommand{\salpha}{\widehat{s_{\fat{\alpha}}}}
\newcommand{\bo}{\bar{\omega}}
\newcommand{\boldl}{\boldsymbol{\ell}}
\newcommand{\hgt}{\textnormal{ht}}

\newcommand{\soc}{\textnormal{soc}}
\newcommand{\Hom}{\textnormal{Hom}}
\newcommand{\End}{\textnormal{End}}
\newcommand{\ind}{\textnormal{Ind}}
\newcommand{\cind}{\textnormal{c-ind}}
\newcommand{\Ext}{\textnormal{Ext}}

\newcommand{\aff}{\textnormal{aff}}
\newcommand{\fat}[1]{\pmb{\boldsymbol{#1}}}
\newcommand{\fil}{\textnormal{Fil}}
\newcommand{\gr}{\textnormal{Gr}}

\newcommand*{\longhookrightarrow}{\ensuremath{\lhook\joinrel\relbar\joinrel\rightarrow}}
\newcommand*{\longtwoheadrightarrow}{\ensuremath{\relbar\joinrel\twoheadrightarrow}}
\newcommand{\sub}[2]{\genfrac{}{}{0pt}{}{#1}{#2}}

\newcommand{\cA}{{\mathcal{A}}}

\newcommand{\cC}{{\mathcal{C}}}

\newcommand{\cH}{{\mathcal{H}}}

\newcommand{\cP}{{\mathcal{P}}}

\newcommand{\cW}{{\mathcal{W}}}
\newcommand{\cX}{{\mathcal{X}}}

\newcommand{\bB}{{\mathbf{B}}}

\newcommand{\bG}{{\mathbf{G}}}
\newcommand{\bH}{{\mathbf{H}}}

\newcommand{\bL}{{\mathbf{L}}}
\newcommand{\bM}{{\mathbf{M}}}
\newcommand{\bN}{{\mathbf{N}}}

\newcommand{\bP}{{\mathbf{P}}}

\newcommand{\bS}{{\mathbf{S}}}
\newcommand{\bT}{{\mathbf{T}}}
\newcommand{\bU}{{\mathbf{U}}}

\newcommand{\bX}{{\mathbf{X}}}

\newcommand{\bZ}{{\mathbf{Z}}}

\newcommand{\bbF}{{\mathbb{F}}}

\newcommand{\bbR}{{\mathbb{R}}}

\newcommand{\bbZ}{{\mathbb{Z}}}

\newcommand{\tW}{\widetilde{W}}

\newcommand{\fS}{{\mathfrak{S}}}

\newcommand{\ff}{{\mathfrak{f}}}

\newcommand{\fm}{{\mathfrak{m}}}
\newcommand{\fn}{{\mathfrak{n}}}
\newcommand{\fo}{{\mathfrak{o}}}
\newcommand{\fp}{{\mathfrak{p}}}

\newcommand{\fs}{{\mathfrak{s}}}

\newcommand{\fv}{{\mathfrak{v}}}

\begin{document}
\nocite{}

\title[First pro-$p$-Iwahori cohomology of principal series]{The first pro-$p$-Iwahori cohomology of mod-$p$ principal series for $p$-adic $\textnormal{GL}_n$}
\date{}
\author{Karol Kozio\l}
\address{Department of Mathematics, University of Toronto, 40 St. George Street, Toronto, ON M5S 2E4, Canada} \email{karol@math.toronto.edu}


\begin{abstract}
Let $p\geq 3$ be a prime number and $F$ a $p$-adic field.  Let $I_1$ denote the pro-$p$-Iwahori subgroup of $\textnormal{GL}_n(F)$, and $\mathcal{H}$ the pro-$p$-Iwahori--Hecke algebra of $\textnormal{GL}_n(F)$ with respect to $I_1$ (over a coefficient field of characteristic $p$).  We compute the structure of $\textnormal{H}^1(I_1,\pi)$ as an $\mathcal{H}$-module, where $\pi$ is a mod-$p$ principal series representation of $\textnormal{GL}_n(F)$.  We also give some partial results about the structure of $\textnormal{H}^1(I_1,\pi)$ for a general split reductive group with irreducible root system.  
\end{abstract}

\maketitle

\section{Introduction}

There has been a great deal of recent activity surrounding the mod-$p$ representation theory of $p$-adic reductive groups, mainly due to its applications in the mod-$p$ and $p$-adic Local Langlands programs (see \cite{breuil:icm} and the references therein).  For a general connected reductive group $G$, work of Abe--Henniart--Herzig--Vign\'eras \cite{ahhv} gives a classification of the smooth irreducible admissible mod-$p$ representations of $G$ in terms of the so-called supersingular representations.  However, these supersingular representations have only been classified for the group $\textnormal{GL}_2(\qp)$ (and other small examples), and therefore our understanding of the category of smooth $G$-representations remains woefully incomplete.

One way of addressing this shortcoming is by passing to the derived category.  Let $I_1$ denote a fixed choice of pro-$p$-Iwahori subgroup of $G$, which we assume to be torsion-free.  In this setting, Schneider has shown that the (unbounded) derived category of smooth mod-$p$ $G$-representations is equivalent to the (unbounded) derived category of differential graded $\cH^\bullet$-modules, where $\cH^\bullet$ is the derived Hecke algebra of $G$ with respect to $I_1$ (see \cite{schneider:dga}).  The equivalence is given by taking a complex $\pi^\bullet$ of smooth $G$-representations to the complex $\textnormal{R}\Gamma(I_1,\pi^\bullet)$, which has a natural action of $\cH^\bullet$.  For some speculations on this derived equivalence and its relation to the mod-$p$ Local Langlands program, see \cite{harris:specs}.

In order to shed some light on the derived equivalence above, we specialize the situation.  Assume $p\geq 3$, let $F$ denote a finite extension of $\qp$, and set $G := \textnormal{GL}_n(F)$.  We let $C$ denote an algebraically closed field of characteristic $p$ (which will serve as our coefficient field), and let $\cH := \End_G(\cind_{I_1}^G(C))$ denote the \emph{pro-$p$-Iwahori--Hecke algebra} over $C$.  Furthermore, we take the complex $\pi^\bullet$ to be a single principal series representation $\ind_B^G(\chi)$ concentrated in degree 0, where $B$ is the upper-triangular Borel subgroup of $G$ and $\chi$ is a smooth $C^\times$-valued character of $B$.

Computing the derived invariants of $\ind_B^G(\chi)$ is still quite difficult, so we consider instead the associated cohomology module.  The space $\bigoplus_{i\geq 0} \textnormal{H}^i(I_1,\ind_B^G(\chi))$ then becomes a graded module over the cohomology algebra $\bigoplus_{i\geq 0}\textnormal{H}^i(\cH^\bullet)$.  In particular, the degree 0 piece $\textnormal{H}^0(\cH^\bullet)$, which naturally identifies with $\cH$, acts on $\bigoplus_{i\geq 0} \textnormal{H}^i(I_1,\ind_B^G(\chi))$.  Our goal in this article will be to focus on the degree 1 piece of the cohomology module, and compute the structure of $\textnormal{H}^1(I_1,\ind_B^G(\chi))$ as an $\cH$-module.  This generalizes work of Breuil--Pa\v{s}k\={u}nas in the $n = 2$ case (\cite[Chapter 7]{breuilpaskunas}).  We hope that this calculation will be useful in understanding the nature of Schneider's derived equivalence.

We now describe the contents of this article in more detail.  After recalling the necessary notation in Section \ref{notation}, we discuss Hecke algebras in Section \ref{heckesect}.  Throughout the course of the article we will make use of the functor $\ind_{\cH_M}^{\cH}$ of parabolic induction for Hecke modules, where $M$ is a standard Levi subgroup of $G$, so we collect the relevant results.  In particular, we recall recent work of Abe (\cite{abe:inductions}) on the left and right adjoint functors of $\ind_{\cH_M}^{\cH}$, as well as results of Ollivier--Vign\'eras (\cite{olliviervigneras}) relating the functors $\ind_{\cH_M}^{\cH}$ and $\ind_P^G$ (where $P$ is the standard parabolic subgroup with Levi component $M$).  Specifically, if $\sigma$ is a smooth $M$-representation over $C$, we have
$$\ind_P^G(\sigma)^{I_1} \cong \ind_{\cH_M}^{\cH}(\sigma^{I_1\cap M})$$
as $\cH$-modules.  In the final subsection, we recall the explicit action of $\cH$ on $I_1$-cohomology in terms of cocycles; this will be our main computational tool.  In the following Section \ref{calcofxi}, we calculate the functions $\xi_x$ which are the input for the Hecke action on cohomology.

In Section \ref{prinseries} we begin to investigate the $\cH$-module $\textnormal{H}^1(I_1,\ind_B^G(\chi))$.  By applying the Mackey formula and Shapiro's lemma, we obtain an isomorphism of vector spaces
$$\textnormal{H}^1\big(I_1,\ind_B^G(\chi)\big) \cong \bigoplus_{w\in W_0}\textnormal{H}^1(I_1\cap w^{-1}Bw,C) = \bigoplus_{w\in W_0}\Hom^{\textnormal{cts}}(I_1\cap w^{-1}Bw,C),$$
where $W_0$ denotes the Weyl group of $G$.  Since the right-hand side is easier to handle than the cohomology classes on the left-hand side, we use the isomorphism above to transfer the Hecke action.  The description of the action of $\cH$ on these tuples of homomorphisms is the content of Lemmas \ref{mainlemma} and \ref{Tomega}.  Unfortunately, this action is still cumbersome to deal with, and we introduce in Subsection \ref{hgtfilt} an $\cH$-stable filtration $\fil_\bullet$ on $\textnormal{H}^1(I_1,\ind_B^G(\chi))$, indexed by the height of a positive root, in order to simplify the formulas.  Given this filtration, we are able to calculate the $\cH$-module structure of the first two graded pieces $\gr_0$ and $\gr_1$ ``by hand.''  Namely, Proposition \ref{gr0} gives
$$\gr_0 \cong \ind_{\cH_T}^{\cH}(\chi)^{\oplus n[F:\qp]},$$
where $T$ is the diagonal maximal torus of $G$, and Proposition \ref{gr1} gives
$$\gr_1 \cong \bigoplus_{\beta\in \Pi}\bigoplus_{r = 0}^{f - 1}\ind_{\cH_{M_{\beta}}}^{\cH}(\fn_{F,\beta,r}),$$
where $\Pi$ is the set of simple roots determined by $B$, where $f$ denotes the degree of the residue field of $F$ over $\bbF_p$, and where $\fn_{F,\beta,r}$ is a certain $\cH_{M_\beta}$-module described in Subsection \ref{gr1sect}.  Furthermore, Lemma \ref{splitfilt} gives a necessary and sufficient condition for the splitting of the first piece of the filtration: the short exact sequence of $\cH$-modules
$$0\longrightarrow \gr_0 \longrightarrow \fil_1 \longrightarrow \gr_1 \longrightarrow 0$$
is nonsplit if and only if $F = \qp$ and $\chi = \chi^{s_\alpha}\overline{\alpha}$ for some simple root $\alpha$ (here $\chi^{s_\alpha}$ denotes the twist of $\chi$ by the simple reflection $s_\alpha$, and $\overline{\alpha}$ is the composition of $\alpha$ with the mod-$p$ cyclotomic character).

A key feature in the calculation of the $\cH$-modules $\gr_0$ and $\gr_1$ is the fact that we can easily write down a basis for each space.  This becomes more complicated for the remaining graded pieces, and we therefore employ the adjoint functors of parabolic induction in order to cimcumvent this difficulty.  This takes up the majority of Section \ref{prinseriesgln}.  The main step is the computation of the $\cH_M$-module structure of the right adjoint $R_{\cH_M}^{\cH}(\gr_{a})$, which is carried out in Proposition \ref{rightadjbeta}.  By adjunction, this is enough to deduce the structure of $\gr_a$ as an $\cH$-module (Corollary \ref{mbetaindcor}). With this in hand, we fully determine when the filtration $\fil_\bullet$ splits, and arrive at our main result.  In the statement of the theorem below, $\Phi^+$ denotes the set of positive roots determined by $B$, and $\Phi_M$ denotes the set of roots attached to a standard Levi subgroup $M$.

\begin{thm*}[Theorem \ref{mainthm}]
Let $\beta\in \Phi^+$, and let $M_\beta$ denote the smallest standard Levi subgroup for which $\beta\in \Phi_{M_\beta}$.  We then have an isomorphism of $\cH$-modules
$$\textnormal{H}^1\big(I_1,\ind_B^G(\chi)\big)~  \cong~  \fil_1 \oplus \bigoplus_{a = 2}^{n - 1}\gr_a~ \cong~  \fil_1 \oplus \bigoplus_{\sub{\beta\in \Phi^+}{\hgt(\beta)\geq 2}}\bigoplus_{r = 0}^{f - 1}\ind_{\cH_{M_\beta}}^{\cH}(\fn_{\beta,r}),$$
where $\fn_{\beta,r}$ is a supersingular $\cH_{M_\beta}$-module, and where the splitting of $\fil_1$ is governed by Lemma \ref{splitfilt}.  In particular, if $F\neq \qp$, we have
$$\textnormal{H}^1\big(I_1,\ind_B^G(\chi)\big)~ \cong~ \ind_{\cH_T}^{\cH}(\chi)^{\oplus n[F:\qp]} \oplus \bigoplus_{\beta\in \Phi^+}\bigoplus_{r = 0}^{f - 1}\ind_{\cH_{M_\beta}}^{\cH}(\fn_{\beta,r}),$$
with $\fn_{\beta,r}$ supersingular.  
\end{thm*}

We observe that $\textnormal{H}^1(I_1,\ind_B^G(\chi))$ always contains a supersingular $\cH$-module if $n\geq 3$.  We can also be more explicit about the structure of the supersingular modules $\fn_{\beta,r}$ appearing in the above theorem.

\begin{propn*}[Proposition \ref{rightadjbeta}, Remarks \ref{fil'notsplit} and \ref{gr'irred}]
Let $\beta \in \Phi^+$ and suppose $\hgt(\beta)\geq 2$.  The $\cH_{M_\beta}$-module $\fn_{\beta,r}$ is supersingular of dimension $(\hgt(\beta) + 1)!/\hgt(\beta)$.  More precisely, $\fn_{\beta,r}$ has an $\cH_{M_\beta}$-stable filtration indexed by $\textnormal{stab}_{W_{M_\beta,0}}(\beta)$, where $W_{M_\beta,0}$ denotes the Weyl group of $M_\beta$.  The graded piece indexed by $w\in \textnormal{stab}_{W_{M_\beta,0}}(\beta)$ is an $(\hgt(\beta) + 1)$-dimensional cyclic supersingular module, which is simple for generic $\chi$.  The filtration on $\fn_{\beta,r}$ does not split in general.
\end{propn*}

In fact, we can be even more precise about the action of $\cH_{M_\beta}$ on $\fn_{\beta,r}$; see the proofs of Proposition \ref{rightadjbeta} and Lemma \ref{Wfilt}.

As an application of the above results, we study extensions of $G$-representations.  Suppose we have a short exact sequence of the form 
$$0\longrightarrow \ind_B^G(\chi) \longrightarrow \pi \longrightarrow \tau \longrightarrow 0,$$
where $\tau$ is an irreducible, admissible, supersingular representation of $G$.  Hu has shown that such an extension does \emph{not} necessarily split (contrary to the case of complex coefficients; see \cite{hu:exts}).  Using the calculation of $\textnormal{H}^1(I_1,\ind_B^G(\chi))$, we give in Proposition \ref{neccond} a necessary condition for the short exact sequence above to be nonsplit.  We conclude Subsection \ref{subsectexts} by giving an example of how this result relates to Serre weight conjectures.

We have assumed in this introduction that $G = \textnormal{GL}_n(F)$ and $I_1$ torsion-free for the sake of simplicity.  In fact, all of the results of Section \ref{prinseries} hold for an arbitrary split connected reductive group $G$ with irreducible root system, and all results throughout hold without the hypothesis that $I_1$ is torsion-free.  In particular, we know the $\cH$-module structure of $\gr_0$ and $\gr_1$ in general.  One of the difficulties in extending the arguments of Section \ref{prinseriesgln} from $\textnormal{GL}_n(F)$ to an arbitrary group $G$ comes from understanding the spaces $\Hom^{\textnormal{cts}}(I_1\cap w^{-1}Bw,C)$, or equivalently, the abelianizations $(I_1\cap w^{-1}Bw)^{\textnormal{ab}}$, as $w$ varies over the Weyl group of $G$.  This seems to be a tricky question in general.  However, we have the following result regarding the higher graded pieces.

\begin{propn*}[Proposition \ref{ssingprop}]
Suppose $G$ is an arbitrary split connected reductive group with irreducible root system.  If the root system of $G$ is not of type $A_1$, then the highest graded piece $\gr_{\textnormal{highest}}$ contains a supersingular $\cH$-module.
\end{propn*}

This proposition, along with the determination of $\gr_0$ and $\gr_1$ for general $G$, shows that $\textnormal{H}^1(I_1,\ind_B^G(\chi))$ contains no supersingular subquotients if and only if $F = \qp$ and the root system of $G$ is of type $A_1$.  This parallels the following expectation (for groups with irreducible root systems): there exists an (underived) equivalence between the category of smooth $G$-representations generated by their $I_1$-invariant vectors and the category of $\cH$-modules if and only if $F = \qp$ and the root system of $G$ is of type $A_1$.

\bigskip

\noindent \textbf{Acknowledgements}.  This article was written in response to a question of Christophe Breuil.  I would like to thank him and Florian Herzig for several useful conversations.  I would also like to thank the anonymous referee(s) for a thorough reading and helpful comments.  During the preparation of this article, funding was provided by NSF grant DMS-1400779 and an EPDI fellowship.

\bigskip

\section{Notation and Preliminaries}\label{notation}

\subsection{Basic Notation}

Let $p\geq 5$ denote a fixed prime number.  Let $F$ be a finite extension of $\qp$, with ring of integers $\fo$, maximal ideal $\fp$, uniformizer $\varpi$ and residue field $k_F$.  We denote by $q = p^f$ the order of the residue field $k_F$.  Given an element $x\in k_F$, we denote by $[x]\in \fo$ its Teichm\"uller lift.  Conversely, if $y\in \fo$, we denote by $\overline{y}\in k_F$ its image in the residue field.

If $\bH$ is an algebraic group over $F$, we denote by $H$ the group $\bH(F)$ of $F$-points.

Let $\bG$ denote a split, connected, reductive group over $F$.  When $\bG = \bG\bL_n$, we relax the condition $p\geq 5$ to $p\geq 3$.  We let $\bZ$ denote the connected center of $\bG$.  Fix a split maximal torus $\bT$, and let 
$$\langle-,-\rangle: X^*(\bT)\times X_*(\bT)\longrightarrow \bbZ$$ 
denote the natural perfect pairing between characters and cocharacters.  We define a homomorphism $\nu:T\longrightarrow X_*(\bT)$ by the condition
$$\langle\chi,\nu(t)\rangle = -\textnormal{val}(\chi(t))$$
for all $\chi\in X^*(\bT)$ and $t\in T$, and where $\textnormal{val}:F^\times\longrightarrow \bbZ$ is the normalized valuation.  Given $\lambda\in X_*(\bT)$, we have $\nu(\lambda(\varpi^{-1})) = \lambda$.

In the standard apartment corresponding to $\bT$ of the semisimple Bruhat--Tits building of $\bG$, we fix a chamber $\cC$ and a hyperspecial vertex $x_0$ such that $x_0\in \overline{\cC}$.  We let $\cP_{x_0}$ (resp. $I$) denote the parahoric subgroup corresponding to $x_0$ (resp. $\cC$).  We also define $I_1$ to be the pro-$p$-Sylow subgroup of $I$.  The group $I$ (resp. $I_1$) is the \emph{Iwahori subgroup} (resp. \emph{pro-$p$-Iwahori subgroup}) of $G$.  (We do \emph{not} assume in the body of the paper that $I_1$ is torsion-free.)  We have $\ker(\nu) = T\cap \cP_{x_0} = T \cap I =: T_0$, which is equal to the maximal compact subgroup of $T$.  Furthermore, we set $T_1:= T\cap I_1$; this is the maximal pro-$p$ subgroup of $T_0$ and $T(k_F) := T_0/T_1$ identifies with the group of $k_F$-points of $\bT$.

Denote the root system of $\bG$ with respect to $\bT$ by $\Phi \subset X^*(\bT)$.  We will eventually assume that $\Phi$ is irreducible.  For a root $\alpha\in \Phi$, we let $\bU_\alpha$ denote the associated root subgroup, and $\alpha^\vee\in X_*(\bT)$ the associated coroot.  We fix a set of simple roots $\Pi$, and let $\Phi = \Phi^+ \sqcup \Phi^-$ be the decomposition defined by $\Pi$.  We let $\textnormal{ht}:\Phi \longrightarrow \bbZ$ denote the height function on $\Phi$ with respect to $\Pi$.

Let $\bB$ denote the Borel subgroup corresponding to $\Pi$ and $\bU$ its unipotent radical, so that $\bB = \bT\ltimes\bU$.  We have $\bU = \langle \bU_\alpha : \alpha\in \Phi^+\rangle$.   A standard parabolic subgroup $\bP = \bM\ltimes \bN$ is any parabolic subgroup containing $\bB$.  It will be tacitly assumed that all Levi subgroups $\bM$ appearing are standard; that is, they are Levi factors of standard parabolic subgroups and contain $\bT$.  Given a standard Levi subgroup $\bM$, we let $\Pi_{\bM}$ (resp. $\Phi_{\bM}$, resp. $\Phi_{\bM}^+$) denote the simple roots (resp. root system, resp. positive roots) defined by $\bM$.  Reciprocally, given a subset $J\subset \Pi$, we let $\bM_J$ denote the standard Levi subgroup it defines.  This sets up a bijection between subsets of $\Pi$ and standard Levi subgroups.  Given a positive root $\beta\in \Phi^+$, we will denote by $\bM_\beta$ the smallest standard Levi subgroup containing $\bU_{\beta}$.  In particular, if $\beta\in \Pi$, we have $\bM_\beta = \bM_{\{\beta\}}$.

\subsection{Weyl groups}

Let $W_0 := N_{\bG}(\bT)/\bT = N_G(T)/T$ denote the Weyl group of $\bG$.  For $\alpha\in \Phi$, we let $s_\alpha\in W_0$ denote the corresponding reflection (and note that $s_{-\alpha} = s_{\alpha}$).  Then the Coxeter group $W_0$ is generated by $\{s_\alpha\}_{\alpha \in \Pi}$.  We let $\boldl:W_0\longrightarrow \bbZ_{\geq 0}$ denote the length function with respect to the set of simple reflections $\{s_\alpha\}_{\alpha\in \Pi}$, and let $\leq$ denote the Bruhat order.  For $\alpha\in \Phi$, the reflection $s_\alpha$ acts on $\chi\in X^*(\bT)$ (resp. $\lambda\in X_*(\bT)$) by the formula
$$s_\alpha(\chi)  =  \chi - \langle\chi,\alpha^\vee\rangle\alpha\qquad(\textnormal{resp.}~s_{\alpha}(\lambda)  =  \lambda - \langle\alpha,\lambda\rangle\alpha^\vee).$$

For a standard Levi subgroup $\bM$, we let $W_{\bM,0}\subset W_0$ denote the corresponding Weyl group, generated by $\{s_\alpha\}_{\alpha\in \Pi_{\bM}}$.  We set 
$$W_0^{\bM} := \left\{w\in W_0: \boldl(ws_\alpha) > \boldl(w)~\textnormal{for all}~\alpha\in \Pi_{\bM}\right\},\qquad {}^{\bM}W_0 := (W_0^{\bM})^{-1}.$$  
Every element $w$ of $W_0$ can be written as $w = vu$ for unique $v\in W_0^{\bM}, u\in W_{\bM,0}$, satisfying $\boldl(w) = \boldl(v) + \boldl(u)$ (see \cite[Section 1.10]{humphreys:book}).  We will denote the longest element of $W_{\bM,0}$ by $w_{\bM,\circ}$, and write $w_\circ$ for $w_{\bG,\circ}$.

Next, we set 
$$\Lambda := T/T_0,\qquad  W := N_G(T)/T_0.$$
The homomorphism $\nu:T\longrightarrow X_*(\bT)$ factors through $\Lambda$, and identifies $\Lambda$ with $X_*(\bT)$.  These groups fit into an exact sequence
$$1\longrightarrow \Lambda \longrightarrow W \longrightarrow W_0 \longrightarrow 1,$$
and the group $W$ acts on the standard apartment $X_*(\bT)/X_*(\bZ)\otimes_{\bbZ}\bbR$ (see \cite[Section I.1]{ss:reptheorysheaves}).  Since $x_0$ is hyperspecial we have $W_0 \cong (N_G(T)\cap \cP_{x_0})/(T\cap \cP_{x_0})$, which gives a section to the surjection $W\longtwoheadrightarrow W_0$ and thus a decomposition
$$W\cong W_0\ltimes \Lambda.$$
 We will always view $W_0$ as the subgroup of $W$ fixing $x_0$ via this section.  The length function $\boldl$ on $W_0$ extends to $W$ (see \cite[Corollary 5.10]{vigneras:hecke1}).

For any standard Levi subgroup $\bM$, we let $W_{\bM}$ denote the subgroup of $W$ given by
$$W_{\bM} := W_{\bM,0}\ltimes \Lambda.$$
Note that, while the restriction of $\boldl$ from $W_0$ to $W_{\bM,0}$ agrees with the length function on $W_{\bM,0}$, the restriction of (the extension of) $\boldl$ from $W$ to $W_{\bM}$ \emph{is not} equal to the extension of $\boldl|_{W_{\bM,0}}$ to $W_{\bM}$ in general.

The set of affine roots is defined as $\Phi\times \bbZ$, with the element $(\alpha,\ell)$ taking the value $\alpha(\lambda) + \ell$ on $\lambda\in X_*(\bT)/X_*(\bZ)\otimes_{\bbZ}\bbR$.  We view $\Phi$ as a subset of $\Phi\times \bbZ$ via the identification $\Phi\cong \Phi\times \{0\}$.  \textbf{We assume that $x_0$ and $\cC$ are chosen so that every element of $\Phi$ takes the value $0$ on $x_0$ and every element of $\Phi^+$ is positive on $\cC$.}  We let 
$$\Pi_{\aff}:= \Pi \sqcup \bigsqcup_{i = 1}^d\{(-\alpha_{0}^{(i)}, 1)\}$$ 
denote the set of simple affine roots, where the elements $\{\alpha_{0}^{(i)}\}_{i = 1}^d$ are the highest roots of the irreducible components of $\Phi$.

We will use boldface Greek letters to denote affine roots.  Given an affine root $\fat{\alpha} = (\alpha,\ell) \in \Phi\times\bbZ$, we let $s_{\fat{\alpha}}\in W$ denote the reflection in the affine hyperplane $\{\lambda\in X_*(\bT)/X_*(\bZ)\otimes_{\bbZ}\bbR: \alpha(\lambda) + \ell = 0\}$.   We define the affine Weyl group $W_\aff\subset W$ to be the subgroup generated by the set $\{s_{\fat{\alpha}}\}_{\fat{\alpha}\in \Pi_\aff}$.  The group $W_\aff$ is a Coxeter group (with respect to the generators $\{s_{\fat{\alpha}}\}_{\fat{\alpha}\in \Pi_\aff}$), and the restriction of $\boldl$ from $W$ to $W_\aff$ agrees with the length function of $W_\aff$ as a Coxeter group.  We also define $\Omega$ as the subgroup of elements of $W$ stabilizing $\cC$; equivalently, $\Omega$ is the subgroup of length 0 elements of $W$.  It is a finitely generated abelian group.  This gives the decomposition
$$W\cong W_\aff\rtimes\Omega.$$ 
The group $W$ acts on the affine roots, and the subgroup $\Omega$ stabilizes the subset $\Pi_\aff$.  We have a similar construction for the group $W_{\bM}$.

We now set 
$$\tW := N_G(T)/T_1.$$
Given any subset $X$ of $W$, we let $\widetilde{X}$ denote its preimage in $\tW$ under the natural projection $\tW\longtwoheadrightarrow W$, so that $\widetilde{X}$ is an extension of $X$ by $T(k_F)$.  The length function $\boldl$ on $W$ inflates to $\tW$ via the projection, and similarly the homomorphism $\nu$ on $\Lambda$ inflates to $\widetilde{\Lambda}$.  For typographical reasons we write $\widetilde{X}_\square$ as opposed to $\widetilde{X_\square}$ if the symbol $X$ has some decoration $\square$.  Given some element $w\in W$ we often let $\widehat{w}\in \tW$ denote a specified choice of lift.  Furthermore, given $w\in \tW$, we let $\bar{w}$ denote the image of $w$ in $W_0$ via the projections $\tW\longtwoheadrightarrow W \longtwoheadrightarrow W_0$.  We will denote by $\overline{\Omega}$ the image in $W_0$ of the group $\Omega$.

\subsection{Lifts of Weyl group elements}
Let $\bG_{x_0}$ denote the Bruhat--Tits $\fo$-group scheme with generic fiber $\bG$ associated to the point $x_0$.  Since the group $\bG$ is split, we have a Chevalley system for $\bG$.  In particular, this means that for every $\alpha\in \Phi$, we have a homomorphism of $\fo$-group schemes (cf. \cite[Section 3.2]{bruhattits2})
$$\varphi_{\alpha}:\bS\bL_{2/\fo}\longrightarrow \bG_{x_0}.$$
We normalize $\varphi_\alpha$ as in \cite[Section II.1.3]{jantzen}.  We define $u_\alpha$ to be the homomorphism
\begin{eqnarray*}
u_{\alpha}:\bG_{\textnormal{a}/\fo} & \longrightarrow & \bG_{x_0}\\
x & \longmapsto & \varphi_{\alpha}\left(\begin{pmatrix}1 & x \\ 0 & 1\end{pmatrix}\right).
\end{eqnarray*}

We now define two sets of lifts of affine reflections.  Given a root $\alpha\in \Phi$, define 
$$\fs_{\alpha} := \varphi_\alpha\left(\begin{pmatrix}  0 & 1 \\ -1 & 0\end{pmatrix}\right)\in N_G(T).$$
(These elements are denoted $n_\alpha$ in \cite{springer}.)  We have 
$$\fs_\alpha^2 = \alpha^\vee(-1)\qquad\textnormal{and}\qquad\fs_\alpha^{-1} = \fs_{-\alpha} = \alpha^\vee(-1)\fs_\alpha.$$

On the other hand, given an affine root $\fat{\alpha} = (\alpha,\ell)\in \Phi\times\bbZ$, define
$$\varphi_{\fat{\alpha}}\left(\begin{pmatrix}a & b \\ c & d\end{pmatrix}\right) := \varphi_{\alpha}\left(\begin{pmatrix}a & \varpi^{\ell}b \\ \varpi^{-\ell}c & d\end{pmatrix}\right),$$
and 
$$u_{\fat{\alpha}}(x) := u_\alpha(\varpi^{\ell}x).$$  
Then, for every $\fat{\alpha} = (\alpha,\ell)\in \Pi_\aff$, we define
$$\salpha := \varphi_{\fat{\alpha}}\left(\begin{pmatrix}0 & 1 \\ -1 & 0\end{pmatrix}\right) \in N_G(T).$$
When $\fat{\alpha} = (\alpha,0)\in \Pi_\aff$, we simply write $\widehat{s_\alpha}$.  Note that
$$\salpha =  \varphi_{\alpha}\left(\begin{pmatrix}0 & \varpi^{\ell} \\ -\varpi^{-\ell} & 0\end{pmatrix}\right) = \alpha^\vee(\varpi^\ell)\fs_{\alpha} = \fs_\alpha\alpha^\vee(\varpi^{-\ell}).$$
In particular, if $\alpha\in \Pi$, then $\widehat{s_\alpha} = \fs_\alpha$.

Given $w\in W_0$ with reduced expression $w = s_{\alpha_1}\cdots s_{\alpha_k}$, $\alpha_i\in \Pi$, we define 
\begin{equation}\label{liftredexp}
\widehat{w} := \widehat{s_{\alpha_1}}\cdots\widehat{s_{\alpha_k}} = \fs_{\alpha_1}\cdots \fs_{\alpha_k} \in N_G(T),
\end{equation}
which is a lift of $w$.  By \cite[Propositions 8.3.3 and 9.3.2]{springer}, this expression is well-defined.  Note that if $\alpha\in \Phi\smallsetminus\Pi$, then $\widehat{s_\alpha}$ (computed using \eqref{liftredexp}) is not necessarily equal to $\fs_{\alpha}$.

\cite[Propositions 8.3.3 and 9.3.2 and Lemma 8.3.2(iii)]{springer} imply that, if $w\in W_0$ and $\alpha\in \Pi$, then 
\begin{equation}\label{bruhateqs}
\begin{aligned}
w(\alpha) \in \Phi^+   & \Longleftrightarrow~  \boldl(ws_\alpha) > \boldl(w)~  \Longleftrightarrow~  \widehat{ws_\alpha} = \widehat{w}\widehat{s_\alpha},\\
w(\alpha) \in \Phi^-  & \Longleftrightarrow~  \boldl(ws_\alpha) < \boldl(w)~  \Longleftrightarrow~  \widehat{ws_\alpha} = \widehat{w}\widehat{s_\alpha}^{-1}.
\end{aligned}
\end{equation}

\subsection{Structure constants}

Let us fix once and for all a total order on $\Phi$.  For two roots $\alpha,\beta\in \Phi$ with $\alpha\neq \pm \beta$, we have
\begin{equation}
\label{comm}
[u_\alpha(x),u_{\beta}(y)] := u_\alpha(x)u_\beta(y)u_{\alpha}(x)^{-1}u_{\beta}(y)^{-1} = \prod_{\sub{i,j > 0}{i\alpha + j\beta\in \Phi}}u_{i\alpha + j\beta}(c_{\alpha,\beta;i,j}x^iy^j),
\end{equation}
where $c_{\alpha,\beta;i,j}$ are some constants, and where the product is taken with respect to the fixed total order on $\Phi$.  We also define constants $d_{\alpha,\beta}$ by
$$\fs_\alpha u_{\beta}(x)\fs_\alpha^{-1} = u_{s_\alpha(\beta)}(d_{\alpha,\beta}x).$$
Since the root morphisms $u_\alpha$ come from a Chevalley system and we assume $p\geq 5$ (resp. $p\geq 3$ when $\bG = \bG\bL_n$), we have $c_{\alpha,\beta;i,j}, d_{\alpha,\beta} \in \fo^\times$ for all choices of $\alpha,\beta\in \Phi$, $i,j > 0$ for which the constants are nonzero.  More precisely, the constants satisfy $c_{\alpha,\beta; i,j}\in \bbZ\cap \fo^\times$ and $d_{\alpha,\beta} = \pm 1$ (\cite[Section 3.2]{bruhattits2}).

If $w= s_{\alpha_1}\cdots s_{\alpha_k}\in W_0$ is a reduced expression and $\beta\in \Phi$, we define
\begin{equation}\label{defofstrcst}
d_{w,\beta}:= d_{\alpha_1,s_{\alpha_2}\cdots s_{\alpha_k}(\beta)}\cdots d_{\alpha_{k - 1},s_{\alpha_k}(\beta)}d_{\alpha_{k},\beta}
\end{equation}
(note that this expression is well-defined).  In particular $d_{w,\beta}= \pm 1$ and $d_{w,\beta} = d_{w,-\beta}$ (\cite[Lemma 9.2.2(ii)]{springer}).  By definition, the elements $d_{w,\beta}$ satisfy
$$\widehat{w}u_\beta(x)\widehat{w}^{-1} = u_{w(\beta)}(d_{w,\beta}x).$$
Further, \cite[Lemma 9.2.2(iii)]{springer} implies
\begin{equation}\label{conjlift}
\widehat{w}\fs_{\beta}\widehat{w}^{-1} = w(\beta^\vee)(d_{w,\beta})\fs_{w(\beta)}.
\end{equation}

\subsection{Miscallany}

\subsubsection{} We record a decomposition for future use.  Given $x\in F^\times$, we have the following equality in $\textnormal{SL}_2(F)$:
$$\begin{pmatrix}1 & 0 \\ x & 1\end{pmatrix} = \begin{pmatrix}1 & x^{-1} \\ 0 & 1\end{pmatrix}\begin{pmatrix}x^{-1} & 0 \\ 0 & x\end{pmatrix}\begin{pmatrix}0 & -1 \\ 1 & 0\end{pmatrix}\begin{pmatrix}1 & x^{-1}\\ 0 & 1\end{pmatrix}.$$
Applying the homomorphism $\varphi_{\fat{\alpha}}$ for $\fat{\alpha} = (\alpha,\ell)\in \Pi_\aff$ gives
\begin{equation}\label{unip}
u_{-\fat{\alpha}}(x) = u_{\fat{\alpha}}(x^{-1})\alpha^\vee(x^{-1})\salpha^{-1}u_{\fat{\alpha}}(x^{-1}).
\end{equation}

\subsubsection{} We fix for the remainder of the article an algebraically closed field $C$ of characteristic $p$.  This will serve as the field of coefficients for all representations and modules appearing.  We fix once and for all an embedding $k_F\longhookrightarrow C$, and view $k_F$ as a subfield of $C$.  To simplify notation, we will identify the groups of square roots of unity $\mu_2(F), \mu_2(k_F),$ and $\mu_2(C)$.

\subsubsection{}\label{miscchar} Given $\alpha\in \Phi$, we define a smooth character $\overline{\alpha}:T \longrightarrow k_F^\times \longhookrightarrow C^\times$ by
$$t\longmapsto \overline{\varpi^{-\textnormal{val}(\alpha(t))}\alpha(t)}.$$
In particular, the character $\overline{\alpha}$ restricted to $T_0$ is given by
$$t_0\longmapsto\overline{\alpha(t_0)},$$
and therefore does not depend on the choice of uniformizer.

\subsubsection{} If $A$ is some set, $a,b\in A$, and $A'\subset A$, we define
\begin{align*}
\mathbf{1}_{A'}(a)  :=  & \begin{cases}1 & \textnormal{if}~a\in A',\\ 0 & \textnormal{if}~a\not\in A',\end{cases}\\
\delta_{a,b} := \mathbf{1}_{\{b\}}(a)  = & \begin{cases}1 & \textnormal{if}~a = b,\\ 0 & \textnormal{if}~a\neq b.\end{cases}
\end{align*}

\subsubsection{} We deviate from the above notation in only one instance:  if $\chi,\chi':F^\times\longrightarrow C^\times$ are two smooth characters, we use the notation $\delta_{\chi,\chi'}$ to denote the function
$$\delta_{\chi|_{\fo^\times},\chi'|_{\fo^\times}} = \begin{cases} 1 & \textnormal{if}~\chi|_{\fo^\times} = \chi'|_{\fo^\times}, \\ 0 & \textnormal{if}~\chi|_{\fo^\times} \neq \chi'|_{\fo^\times}. \end{cases}$$

\bigskip

\section{Hecke algebras}\label{heckesect}

We review some facts related to pro-$p$-Iwahori--Hecke algebras.

\subsection{Definitions}  We let $\cH$ denote the pro-$p$-Iwahori--Hecke algebra of $G$ with respect to $I_1$ over $C$:
$$\cH:= \End_{G}\left(\cind_{I_1}^G(C)\right),$$
where $C$ denotes the trivial $I_1$-module over $C$ (see \cite[Sections 4.1 and 4.2]{vigneras:hecke1} for details).  Similarly, for any standard Levi subgroup $\bM$, let $\cH_M$ denote the analogously defined pro-$p$-Iwahori--Hecke algebra of $M$ with respect to $I_1\cap M$ (which is \emph{not} a subalgebra of $\cH$ in general).

Using the adjunction isomorphism
$$\cH\cong \Hom_{I_1}\left(C,\cind_{I_1}^G(C)\right) = \cind_{I_1}^G(C)^{I_1},$$
we view $\cH$ as the convolution algebra of compactly supported, $C$-valued, $I_1$-biinvariant functions on $G$.  Given $g\in G$, we let $\T_g$ denote the characteristic function of $I_1gI_1$.  If $w\in \tW$ and $\dot{w}\in N_G(T)$ is a lift of $w$, then the double coset $I_1\dot{w}I_1$ does not depend on the choice of lift, and we will therefore write $\T_w$ for $\T_{\dot{w}}$.  By the Bruhat decomposition (cf. \cite[Proposition 3.35]{vigneras:hecke1}), the set $\{\T_w\}_{w\in \tW}$ then gives a basis of $\cH$.

The elements of $\cH$ satisfy braid relations and quadratic relations:
\begin{enumerate}[$\bullet$]
\item if $w,w'\in \tW$ satisfy $\boldl(w) + \boldl(w') = \boldl(ww')$, then
$$\T_{w}\T_{w'} = \T_{ww'};$$
\item if $\fat{\alpha} = (\alpha,\ell)\in \Pi_\aff$, then
$$\T_{\salpha}^2 = c_{\fat{\alpha}}\T_{\salpha} = \T_{\salpha}c_{\fat{\alpha}},$$
where
$$c_{\fat{\alpha}}:= \sum_{x\in k_F^\times}\T_{\alpha^\vee([x])}.$$
\end{enumerate}
Since $\tW$ is generated by $\{\salpha\}_{\fat{\alpha}\in \Pi_\aff}$ and $\widetilde{\Omega}$, the braid relations show that $\cH$ is generated by $\{\T_{\salpha}\}_{\fat{\alpha}\in \Pi_\aff}$ and $\{\T_{\omega}\}_{\omega\in \widetilde{\Omega}}$.

As in \cite[Lemma 4.12]{vigneras:hecke1}, we define 
$$\T_{\salpha}^* := \T_{\salpha} - c_{\fat{\alpha}}.$$
If $w\in \tW$ has reduced expression $w = \widehat{s_{\fat{\alpha}_1}}\cdots \widehat{s_{\fat{\alpha}_k}}\omega$, where $\fat{\alpha}_i\in \Pi_\aff$ and $\omega\in \widetilde{\Omega}$, we define
$$\T_w^* := \T_{\widehat{s_{\fat{\alpha}_1}}}^*\cdots \T_{\widehat{s_{\fat{\alpha}_k}}}^*\T_{\omega}.$$
(This expression is well-defined, cf. \cite[Proposition 4.13(2)]{vigneras:hecke1}.)

Finally, we define $\cH_\aff$ to be the $C$-vector subspace of $\cH$ which is spanned by $\{\T_w\}_{w\in \tW_\aff}$.  By the braid and quadratic relations, this forms a subalgebra of $\cH$, called the \emph{affine pro-$p$-Iwahori--Hecke algebra}.  The algebra $\cH_\aff$ is generated by $\{\T_{\salpha}\}_{\fat{\alpha}\in \Pi_\aff}$ and $\{\T_{t_0}\}_{t_0\in T(k_F)}$.  If $\Phi = \Phi^{(1)}\sqcup \ldots \sqcup \Phi^{(d)}$ is the decomposition of $\Phi$ into irreducible components, we get a corresponding decomposition of simple affine roots $\Pi_\aff = \Pi_\aff^{(1)}\sqcup \ldots \sqcup \Pi_\aff^{(d)}$.  For $1\leq i \leq d$, the subalgebras of $\cH_\aff$ generated by $\{\T_{\salpha}\}_{\fat{\alpha}\in \Pi_\aff^{(i)}}$ and $\{\T_{t_0}\}_{t\in T(k_F)}$ are called the \emph{irreducible components of $\cH_\aff$}.

\subsection{Induction functors and their adjoints}

We now recall the functor of parabolic induction for Hecke modules, along with its left and right adjoints.  All $\cH$-modules we consider will be right modules, unless otherwise indicated.  For more information, see \cite{abe:inductions} and \cite{olliviervigneras}.

Let $\bM$ denote a standard Levi subgroup, and $\cH_M$ the associated Hecke algebra with respect to the pro-$p$-Iwahori subgroup $I_1\cap M$.  This space has a basis given by $\{\T_{w}^M\}_{w\in \tW_{\bM}}$, where $\T_w^M$ denotes the characteristic function of the double coset $(I_1\cap M)w(I_1\cap M)$.  This algebra again satisfies braid relations and quadratic relations, \emph{relative to the length function on $\tW_{\bM}$} (which is different than the restriction to $\tW_{\bM}$ of the length function $\boldl$ on $\tW$).

Let $w\in \tW_{\bM}$, and write $w = \lambda w'$ with $\lambda\in \widetilde{\Lambda}, w'\in \tW_{\bM,0}$.  Recall than the element $w$ is \emph{$M$-positive} (resp. $M$-negative) if 
$$\langle\alpha,\nu(\lambda)\rangle\leq 0~\textnormal{for all}~\alpha\in \Phi^+\smallsetminus\Phi_{\bM}^+\qquad(\textnormal{resp.}~ \langle\alpha,\nu(\lambda)\rangle \geq 0~\textnormal{for all}~\alpha\in \Phi^+\smallsetminus\Phi_{\bM}^+).$$
The $C$-vector space spanned by all $\T_{w}^M$ for $M$-positive $w$ (resp. $M$-negative $w$) is actually a subalgebra of $\cH_M$, which we denote $\cH_{M}^+$ and call the \emph{positive subalgebra} (resp. $\cH_M^-$, called the \emph{negative subalgebra}).  Furthermore, we have injective algebra homomorphisms
\begin{equation}\label{possubalg}
\begin{aligned}
\theta: \cH_M^+ & \longhookrightarrow  \cH \\
\T_{w}^M & \longmapsto  \T_w,
\end{aligned}
\end{equation}
\begin{equation}\label{negsubalg}
\begin{aligned}
\theta^*: \cH_M^- & \longhookrightarrow  \cH\\
\T_{w}^{M,*} & \longmapsto  \T_w^*.
\end{aligned}
\end{equation}

\begin{defn}
Let $\fn$ denote a right $\cH_M$-module.  We define the \emph{parabolic induction of $\fn$ from $\cH_M$ to $\cH$} to be the right $\cH$-module
$$\ind_{\cH_M}^{\cH}(\fn) := \fn\otimes_{\cH_M^+}\cH,$$
where we view $\cH_M^+$ as a subalgebra of $\cH$ via the morphism \eqref{possubalg}.  
\end{defn}

The functor $\ind_{\cH_M}^\cH$ is a faithful and exact functor from the category of right $\cH_M$-modules to the category of right $\cH$-modules.  Furthermore, it has both a left and right adjoint, which we denote by $L_{\cH_M}^{\cH}$ and $R_{\cH_M}^{\cH}$, respectively.  They are defined as follows.

Let $\bM'$ denote the standard Levi subgroup $w_\circ\bM w_\circ^{-1}$, corresponding to the subset of simple roots $\Pi_{\bM'} = -w_\circ(\Pi_{\bM})$.  Let $\fm$ denote a right $\cH$-module.  As $C$-vector spaces, we have
$$L_{\cH_M}^{\cH}(\fm) = \fm\otimes_{\cH_{M'}^-}\cH_{M'},$$
where we view $\cH_{M'}^-$ as a subalgebra of $\cH$ via the morphism \eqref{negsubalg}.  The right action of $\T_{w}^M\in \cH_M$ on the vector space above is given by the right action of $\T_{\widehat{w_{\bM,\circ}w_\circ}^{-1}w\widehat{w_{\bM,\circ}w_\circ}}^{M'}\in \cH_{M'}$ (recall that $w_{\bM,\circ}$ is the longest element of $W_{\bM,0}$).  If we let $\lambda^-\in \widetilde{\Lambda}$ denote an element which is central in $\tW_{\bM'}$ and which satisfies $\langle\alpha,\nu(\lambda^-)\rangle > 0$ for all $\alpha\in \Phi^+ \smallsetminus\Phi_{\bM'}^+$, then the algebra $\cH_{M'}$ is a localization of $\cH_{M'}^-$ at the element $\T_{\lambda^-}^{M',*}$.  Therefore, we have an isomorphism of $C$-vector spaces
$$L_{\cH_M}^{\cH}(\fm) = \fm[(\T_{\lambda^-}^*)^{-1}].$$
In particular, $L_{\cH_M}^{\cH}$ is exact.

If $\fm$ is a right $\cH$-module, the right adjoint of parabolic induction is given by
$$R_{\cH_M}^{\cH}(\fm) = \Hom_{\cH_M^+}(\cH_M,\fm),$$
with the right action of $\cH_M$ being the evident one, and where we view $\fm$ as an $\cH_M^+$-module via \eqref{possubalg}.  If we let $\lambda^+\in \widetilde{\Lambda}$ denote an element which is central in $\tW_{\bM}$ and which satisfies $\langle\alpha,\nu(\lambda^+)\rangle < 0$ for all $\alpha\in \Phi^+\smallsetminus\Phi_{\bM}^+$, then the algebra $\cH_{M}$ is the localization of $\cH_{M}^+$ at the element $\T_{\lambda^+}^M$.  Therefore, we have an isomorphism of $C$-vector spaces
\begin{eqnarray*}
R_{\cH_M}^{\cH}(\fm) & \stackrel{\sim}{\longrightarrow} & \varprojlim_{v\mapsto v\cdot\T_{\lambda^+}}\fm\\
f & \longmapsto & \left(f((\T_{\lambda^+}^M)^{-n})\right)_{n\geq0}
\end{eqnarray*}
In particular, if $\fm$ is finite-dimensional over $C$, then the tower $(\fm)_{n\geq 0}$ satisfies the Mittag-Leffler condition, and we obtain $\varprojlim^1_{v\mapsto v\cdot\T_{\lambda^+}}\fm = 0$.  (Here $\varprojlim^1$ denotes the first derived functor of the inverse limit.)  Therefore, the functor $R_{\cH_M}^{\cH}$ is exact on the category of finite-dimensional right $\cH$-modules.  Moreover, if $\fm$ is finite-dimensional, then we have an isomorphism of $C$-vector spaces
\begin{eqnarray*}
R_{\cH_M}^{\cH}(\fm) & \stackrel{\sim}{\longrightarrow} & \bigcap_{n\geq 0}\fm\cdot\T_{\lambda^+}^n \subset \fm\\
f & \longmapsto & f(1).
\end{eqnarray*}

\subsection{Supersingular $\cH$-modules}

We first recall two characters of $\cH$.  

\begin{defn}
\begin{enumerate}[(a)]
\item The \emph{trivial and sign characters} $\chi_{\textnormal{triv}},\chi_{\textnormal{sign}}:\cH\longrightarrow C$ are defined by
$$\chi_{\textnormal{triv}}(\T_{w}) = q^{\boldl(w)} \qquad \textnormal{and}\qquad \chi_{\textnormal{sign}}(\T_{w}) = (-1)^{\boldl(w)},$$
respectively (using the convention that $0^0 = 1$).  
\item Given any subalgebra $\cA$ of $\cH$ (e.g., $\cH_\aff$, etc.), we define the \emph{trivial and sign characters of $\cA$} to be the restrictions of $\chi_{\textnormal{triv}}$ and $\chi_{\textnormal{sign}}$ to $\cA$.   
\item Given a character $\xi:T_0 \longrightarrow C^\times$ which satisfies $\xi\circ\alpha^\vee([x]) = 1$ for every $\alpha\in \Pi$ and every $x\in k_F^\times$, we define the \emph{twists of $\chi_{\textnormal{triv}}$ and $\chi_{\textnormal{sign}}$ by $\xi$} to be the characters of $\cH_\aff$ defined by
$$\begin{cases}(\xi\otimes\chi_{\textnormal{triv}})(\T_{t_0}) = \xi(t_0)^{-1} & \textnormal{if}~t_0\in T(k_F), \\ (\xi\otimes\chi_{\textnormal{triv}})(\T_{\salpha}) = 0 & \textnormal{if}~\fat{\alpha}\in \Pi_\aff,\end{cases}$$
$$\begin{cases}(\xi\otimes\chi_{\textnormal{sign}})(\T_{t_0}) = \xi(t_0)^{-1} & \textnormal{if}~t_0\in T(k_F), \\ (\xi\otimes\chi_{\textnormal{sign}})(\T_{\salpha}) = -1 & \textnormal{if}~\fat{\alpha}\in \Pi_\aff.\end{cases}$$
\end{enumerate}
\end{defn}

\begin{defn}
A simple right $\cH$-module is said to be \emph{supersingular} if it is not isomorphic to a subquotient of $\ind_{\cH_M}^{\cH}(\fn)$ for any simple right $\cH_M$-module $\fn$ and any proper Levi $\bM$.  A finite-dimensional right $\cH$-module is said to be \emph{supersingular} if every simple subquotient is supersingular.  
\end{defn}

\begin{rmk}
The definition given here is equivalent to the definition of supersingularity given in \cite{vigneras:hecke3}.  Both definitions are equivalent to the following.  Recall that a supersingular character of $\cH_\aff$ is one whose restriction to each irreducible component of $\cH_\aff$ is different from a twist of the trivial or sign character.  Then a simple right $\cH$-module $\fm$ is supersingular if and only if the restriction of $\fm$ to $\cH_\aff$ contains a supersingular character (\emph{op. cit.}, Corollary 6.13(2) and Theorem 6.15).  
\end{rmk}

\begin{rmk}
We have yet another characterization of supersingular modules.  Suppose $\fm$ is a simple right $\cH$-module.  Then $\fm$ is supersingular if and only if $L_{\cH_M}^{\cH}(\fm) = 0$ and $R_{\cH_M}^{\cH}(\fm) = 0$ for all proper standard Levi subgroups $\bM\subsetneq \bG$ (\cite[Proposition 5.18 and Theorem 5.20]{abe:inductions}).  Since the functors $L_{\cH_M}^{\cH}$ and $R_{\cH_M}^{\cH}$ are exact on the category of finite-dimensional $\cH$-modules, we see that a finite-dimensional $\cH$-module $\fm$ is supersingular if and only if $L_{\cH_M}^{\cH}(\fm) = 0$ and $R_{\cH_M}^{\cH}(\fm) = 0$ for all proper standard Levi subgroups $\bM\subsetneq \bG$.  
\end{rmk}

\subsection{$G$-representations}

We briefly recall parabolic induction for smooth representations.

Let $\bP = \bM\ltimes\bN$ denote a standard parabolic subgroup, and let $\sigma$ be a smooth $M$-representation.  The parabolic induction $\ind_P^G(\sigma)$ of $\sigma$ is defined as the space of locally constant functions $\mathsf{f}:G\longrightarrow \sigma$ which satisfy $\mathsf{f}(pg) = p.\mathsf{f}(g)$ for $p\in P, g\in G$ (we view $\sigma$ as a representation of $P$ via the projection $P \longtwoheadrightarrow M$).  The group $G$ acts on $\mathsf{f}\in \ind_P^G(\sigma)$ by $(g.\mathsf{f})(h) = \mathsf{f}(hg)$, giving a smooth $G$-representation.

The space $\sigma^{I_1\cap M} = \textnormal{H}^0(I_1\cap M, \sigma)$ naturally has a right action of $\cH_M$, and the space $\ind_P^G(\sigma)^{I_1} = \textnormal{H}^0(I_1,\ind_P^G(\sigma))$ naturally has a right action of $\cH$ (see \cite[Proof of Lemma 4.5]{olliviervigneras}, or the following section).  They are related by \cite[Proposition 4.4]{olliviervigneras}: we have an isomorphism of $\cH$-modules
\begin{equation}\label{h0}
\textnormal{H}^0\big(I_1,\ind_P^G(\sigma)\big)  \cong  \ind_{\cH_M}^{\cH}\left(\textnormal{H}^0(I_1\cap M, \sigma)\right).
\end{equation}
Our goal will be to understand the structure of the higher cohomology groups when $\bP = \bB$.

Two final points of notation: if $\chi:T\longrightarrow C^\times$ is a smooth character of $T$, then $\chi^{I_1\cap T} = \chi^{T_1} = \chi$ inherits the structure of a right $\cH_T$-module.  If $\lambda\in \widetilde{\Lambda}$ and $v\in \chi$, we have
$$v\cdot \T_{\lambda}^T = \chi(\lambda)^{-1}v.$$
We will use the letter $\chi$ to refer to both the character of $T$ and the $\cH_T$-module described above.  The meaning should be clear from context.  Finally, if $\chi:T\longrightarrow C^\times$ is a smooth character of $T$ and $\alpha\in \Phi$, we let $\chi^{s_\alpha}:T\longrightarrow C^\times$ denote the character defined by $\chi^{s_\alpha}(t) := \chi(\widehat{s_\alpha}t\widehat{s_\alpha}^{-1})$.

\subsection{Hecke action on cohomology}

We now discuss Hecke actions on (continuous) group cohomology.  For a reference, see \cite{lee:heckeaction}, \cite{kugaparrysah}, and \cite{rhiewhaples}.

Let $g\in G$, and let us write
\begin{equation}\label{xiforcoh}
I_1gI_1 = \bigsqcup_{x\in \cX} I_1g_x
\end{equation}
where $\cX$ is some finite index set.  Given $h\in I_1$, we have $g_xh\in I_1gI_1$, so we may write
$$g_xh = \xi_x(h)g_{x(h)}$$
for some $\xi_x(h)\in I_1$ and $x(h)\in \cX$ (we suppress the dependence on $g$ from notation).  If $h, h'\in I_1$, we have the cocycle conditions
\begin{equation}\label{cocycle}
x(hh') = x(h)(h'),\qquad\textnormal{and}\qquad \xi_x(hh') = \xi_x(h)\xi_{x(h)}(h').
\end{equation}

Suppose now that we have a smooth $G$-representation $\pi$ and an inhomogeneous cocycle $f\in \textnormal{Z}^i(I_1,\pi)$.  We define another cocycle $f\cdot\T_g$ by
\begin{equation}\label{actoncocyc}
(f\cdot\T_g)(h_1,\ldots, h_i) = \sum_{x\in \cX}g_x^{-1}.f\left(\xi_x(h_1),~ \xi_{x(h_1)}(h_2),~ \xi_{x(h_1h_2)}(h_3),~ \ldots~,~ \xi_{x(h_1\cdots h_{i - 1})}(h_i)\right).  
\end{equation}
In particular, if $g$ normalizes $I_1$, this reduces to
\begin{equation}\label{actoncocycsimple}
(f\cdot\T_g)(h_1,\ldots, h_i) = g^{-1}.f\left(gh_1g^{-1},~ gh_2g^{-1},~ \ldots~,~ gh_ig^{-1}\right).
\end{equation}
By passing to cohomology, equation \eqref{actoncocyc} gives an well-defined action of $\cH$ on $\textnormal{H}^i(I_1,\pi)$.

\begin{rmk}
This definition agrees with the ``functorial'' definition 
$$\T_g = \textnormal{cor}^{I_1}_{I_1\cap g^{-1}I_1g} \circ g^{-1}_* \circ \textnormal{res}^{I_1}_{I_1\cap gI_1g^{-1}},$$ 
where the rightmost map is the restriction from $\textnormal{H}^i(I_1,\pi)$ to $\textnormal{H}^i(I_1\cap gI_1g^{-1},\pi)$, the middle map is the conjugation by $g^{-1}$, and the leftmost map is the corestriction from $\textnormal{H}^i(I_1\cap g^{-1}I_1g,\pi)$ to $\textnormal{H}^i(I_1,\pi)$.  
\end{rmk}

\subsection{Calculations for Hecke action}\label{calcofxi}

In this subsection we calculate the maps $\xi_x(h)$ and $x(h)$ which appear in the action of $\cH$ on cohomology.  Since $\cH$ is generated by $\{\T_{\salpha}\}_{\fat{\alpha}\in \Pi_\aff}$ and $\{\T_\omega\}_{\omega\in \widetilde{\Omega}}$, and the action of $\T_{\omega}$ is given by equation \eqref{actoncocycsimple}, it suffices to understand $\xi_x(h)$ and $x(h)$ when $g = \salpha$.

Taking $g = \salpha$ in the decomposition \eqref{xiforcoh} and using the Iwahori decomposition of $I_1$ gives
$$I_1\widehat{s_{\fat{\alpha}}}I_1 = \bigsqcup_{x\in k_F}I_1\widehat{s_{\fat{\alpha}}}u_{\fat{\alpha}}([x]).$$
Therefore, our index set $\cX$ is $k_F$, and for $x\in k_F$, we have $g_x = \salpha u_{\fat{\alpha}}([x])$.  Since $I_1$ is generated by $T_1$ and $u_{\fat{\beta}}(\fo)$ for all $\fat{\beta} \in (\Phi^+\times\{0\}) \sqcup (\Phi^-\times\{1\})$, by equations \eqref{cocycle} it suffices to understand $\xi_x(t), \xi_x(u_{\fat{\beta}}(y)), x(t)$, and $x(u_{\fat{\beta}}(y))$, where $t\in T_1, \fat{\beta}\in (\Phi^+\times\{0\}) \sqcup (\Phi^-\times\{1\})$ and $y\in \fo$.

\begin{lemma}
\label{conj}
Let $\fat{\alpha} = (\alpha,\ell)\in  \Pi_\aff$, $\fat{\beta} = (\beta,m)\in (\Phi^+\times\{0\})\sqcup (\Phi^-\times\{1\})$ and $x\in k_F, y\in \fo$.  We have:
\begin{enumerate}[(a)]
\item if $\fat{\alpha} = \fat{\beta}$, then
\begin{eqnarray*}
x(u_{\fat{\alpha}}(y)) & = & x + \overline{y},\\
\xi_x(u_{\fat{\alpha}}(y)) & = & u_{-\fat{\alpha}}([x + \overline{y}] - [x] - y);
\end{eqnarray*}\label{conja}
\item if $\fat{\alpha} \neq \fat{\beta}$ and $\alpha\neq -\beta$, then
\begin{eqnarray*}
x(u_{\fat{\beta}}(y)) & = & x,\\
\xi_x(u_{\fat{\beta}}(y)) & = & \left(\prod_{\sub{i,j > 0}{i\alpha + j\beta\in \Phi}}u_{s_\alpha(i\alpha + j\beta)}\left(d_{\alpha,i\alpha + j\beta}c_{\alpha,\beta;i,j}\varpi^{-i\ell + jm - j\ell\langle \beta,\alpha^\vee\rangle}[x]^iy^j\right)\right)\\
& &  \qquad \cdot ~u_{s_\alpha(\beta)}\left(d_{\alpha,\beta}\varpi^{m - \ell\langle\beta,\alpha^\vee\rangle}y\right);
\end{eqnarray*}\label{conjb}
\item if $\fat{\alpha}\neq\fat{\beta}$ and $\alpha = -\beta$, then
\begin{eqnarray*}
x(u_{\fat{\beta}}(y)) & = & x,\\
\xi_x(u_{\fat{\beta}}(y)) & = & u_{\fat{\alpha}}(-\varpi y\nu^{-1})\alpha^\vee(\nu^{-1})u_{\fat{\beta}}([x]^2y\nu^{-1})\\
 & = & u_{\fat{\beta}}([x]^2y\nu'^{-1})\alpha^\vee(\nu')u_{\fat{\alpha}}(-\varpi y\nu'^{-1}),
\end{eqnarray*}
where $\nu := 1 + \varpi[x]y$ and $\nu' := 1 - \varpi[x]y$; \label{conjc}
\item for $t\in T_1$, we have
\begin{eqnarray*}
x(t) & = & x,\\ 
\xi_x(t) & = & t^{s_\alpha}u_{-\fat{\alpha}}\left((1 - \alpha(t)^{-1})[x]\right),
\end{eqnarray*}
where $t^{s_\alpha} := \salpha t \salpha^{-1}$.  \label{conjd}
\end{enumerate}
\end{lemma}

\begin{proof}
See Lemma 4.1 of \cite{koziol:h1triv}.
\end{proof}

We record one more lemma which will be used later.

\begin{lemma}
\label{conj2}
Let $\fat{\alpha} = (\alpha,\ell)\in  \Pi_\aff$, $\fat{\beta} = (\beta,m)\in (\Phi^+\times\{0\})\sqcup (\Phi^-\times\{1\})$ and $x\in k_F^\times, y\in \fo$.  We have:
\begin{enumerate}[(a)]
\item if $\fat{\alpha} \neq \fat{\beta}$ and $\alpha\neq -\beta$, then
\begin{eqnarray*}
u_{\fat{\alpha}}([x]^{-1})\xi_x(u_{\fat{\beta}}(y))u_{\fat{\alpha}}(-[x]^{-1}) & = & \left(\prod_{\sub{i,j > 0}{-i\alpha + j\beta\in \Phi}} u_{-i\alpha + j\beta}\left(c_{-\alpha,\beta;i,j}(-1)^{j\langle\beta,\alpha^\vee\rangle}\varpi^{-i\ell + jm}[x]^{i - j\langle\beta,\alpha^\vee\rangle}y^j\right)\right)\\
& & \qquad\cdot~u_{\beta}\left(\varpi^m(-[x]^{-1})^{\langle\beta,\alpha^\vee\rangle}y\right);
\end{eqnarray*}\label{conj2a}
\item if $\fat{\alpha}\neq\fat{\beta}$ and $\alpha = -\beta$, then
$$u_{\fat{\alpha}}([x]^{-1})\xi_x(u_{\fat{\beta}}(y))u_{\fat{\alpha}}(-[x]^{-1})  =  u_{\fat{\beta}}([x]^2y);$$
\label{conj2b}
\item for $t\in T_1$, we have
$$u_{\fat{\alpha}}([x]^{-1})\xi_x(t)u_{\fat{\alpha}}(-[x]^{-1})  =  tu_{-\fat{\alpha}}\left((\alpha(t) - 1)[x]\right).$$
\label{conj2c}
\end{enumerate}
\end{lemma}

\begin{proof}
Suppose that $h\in I_1$ is such that $x(h) = x$, so that 
$$\xi_x(h) = \salpha u_{\fat{\alpha}}([x]) h u_{\fat{\alpha}}(-[x])\salpha^{-1}.$$  
By equation \eqref{unip} we have $u_{\fat{\alpha}}([x]^{-1})\salpha u_{\fat{\alpha}}([x]) = u_{-\fat{\alpha}}([x])\alpha^\vee(-[x]^{-1})$, and therefore
$$u_{\fat{\alpha}}([x]^{-1})\xi_x(h)u_{\fat{\alpha}}(-[x]^{-1}) = u_{-\fat{\alpha}}([x])\alpha^\vee(-[x]^{-1})h\alpha^\vee(-[x])u_{-\fat{\alpha}}(-[x]).$$
The result is then a straightforward computation (using \eqref{comm} for part (a)).  
%
%
\end{proof}

\bigskip

\section{Principal series}\label{prinseries}

\textbf{We assume henceforth that $\Phi$ is irreducible.}  We let $\chi:T\longrightarrow C^\times$ denote a smooth character of $T$, and consider the $\cH$-action on $\textnormal{H}^1(I_1,\ind_B^G(\chi))$.

\subsection{Hecke action}
First, note that the Mackey formula gives an $I_1$-equivariant isomorphism
\begin{eqnarray*}
\ind_B^G(\chi)|_{I_1} & \stackrel{\sim}{\longrightarrow} & \bigoplus_{w\in W_0} \ind_{I_1\cap w^{-1}Bw}^{I_1}(C)\\
\mathsf{f} & \longmapsto & (\mathsf{f}'_w)_{w\in W_0},
\end{eqnarray*}
where $\mathsf{f}'_w(j) := \mathsf{f}(\widehat{w}j)$ for $j\in I_1$, and where $C$ denotes the trivial $I_1\cap w^{-1}Bw$-representation.  By \cite[Proposition 10]{serre:galoiscoh}, the map
\begin{eqnarray*}
 \ind_{I_1\cap w^{-1}Bw}^{I_1}(C) & \longrightarrow & C\\
 \mathsf{f}' & \longmapsto & \mathsf{f}'(1)
\end{eqnarray*}
induces an isomorphism
$$\textnormal{H}^i\big(I_1, \ind_{I_1\cap w^{-1}Bw}^{I_1}(C)\big) \stackrel{\sim}{\longrightarrow} \textnormal{H}^i\big(I_1\cap w^{-1}Bw,C\big).$$

Combining these facts, we obtain an isomorphism
\begin{equation}\label{shapiro}
\begin{aligned}
\textnormal{H}^1\big(I_1,\ind_B^G(\chi)\big)~ & \stackrel{\sim}{\longrightarrow}~  \bigoplus_{w\in W_0} \Hom^{\textnormal{cts}}(I_1\cap w^{-1}Bw, C) \\
{} [f]~ & \longmapsto~  (\psi_w)_{w\in W_0},
\end{aligned}
\end{equation}
where $\psi_w: I_1\cap w^{-1}Bw \longrightarrow C$ is the homomorphism defined by $\psi_w(h) = f(h)(\widehat{w})$ for $h\in I_1\cap w^{-1}Bw$.  In particular, for $w$ and $h$ fixed, the expression $f(h)(\widehat{w})$ depends only on the cohomology class of $f$.  Conversely, given any collection of homomorphisms $(\psi_w')_{w\in W_0}\in \bigoplus_{w\in W_0} \Hom(I_1\cap w^{-1}Bw, C)$, there exists a cocycle $f'\in \textnormal{Z}^1(I_1,\ind_B^G(\chi))$, unique up to coboundary, such that $f'(h)(\widehat{w}) = \psi_w'(h)$ for $h\in I_1\cap w^{-1}Bw$.

We will use the isomorphism above to describe the action of $\cH$ on $\textnormal{H}^1(I_1,\ind_B^G(\chi))$.

\begin{lemma}
\label{mainlemma}
Let $[f]\in \textnormal{H}^1(I_1,\ind_B^G(\chi))$ be associated to $(\psi_w)_{w\in W_0}\in \bigoplus_{w\in W_0}\Hom^{\textnormal{cts}}(I_1\cap w^{-1}Bw,C)$ via \eqref{shapiro}, and let $\fat{\alpha} = (\alpha,\ell)\in \Pi_\aff$, $w\in W_0$ and $h\in I_1\cap w^{-1}Bw$.  We then have
\begin{equation}\label{PSaction}
 (f\cdot \T_{\salpha})(h)(\widehat{w}) = \begin{cases} \zeta\sum_{x\in k_F} \psi_{ws_\alpha}\left(\xi_x(h)\right) & \textnormal{if}~w(\alpha) \in \Phi^+, \\ 
\zeta\cdot \psi_{ws_\alpha}\left(\xi_0(h)\right) & \textnormal{if}~w(\alpha)\in \Phi^-. \\ \quad + \sum_{x\in k_F^\times}\chi\circ w(\alpha^\vee)(-[x])\cdot \psi_w\left(u_{\fat{\alpha}}([x]^{-1})\xi_x(h)u_{\fat{\alpha}}(-[x]^{-1})\right) & \end{cases} 
\end{equation}
Here $\zeta := \chi(\widehat{w}\salpha^{-1}\widehat{ws_\alpha}^{-1})$.  
\end{lemma}

\begin{proof}
By equation \eqref{actoncocyc}, we have
$$(f\cdot \T_{\salpha})(h) = \sum_{x\in k_F}u_{\fat{\alpha}}(-[x])\salpha^{-1}.f\left(\xi_x(h)\right).$$

Suppose first that $w(\alpha)\in \Phi^+$.  We have
$$(f\cdot \T_{\salpha})(h)(\widehat{w})  =  \left(\sum_{x\in k_F} u_{\fat{\alpha}}(-[x])\salpha^{-1}.f\left(\xi_x(h)\right)\right)\left(\widehat{w}\right) =  \sum_{x\in k_F} f\left(\xi_x(h)\right)\left(\widehat{w}u_{\fat{\alpha}}(-[x])\salpha^{-1}\right).
$$
Since $\widehat{w}u_{\fat{\alpha}}(-[x])\widehat{w}^{-1} \in U$ and $f\left(\xi_x(h)\right) \in \ind_B^G(\chi)$, the above expression is equal to 
$$\sum_{x\in k_F} f\left(\xi_x(h)\right)\left(\widehat{w}\salpha^{-1}\right) =  \zeta\sum_{x\in k_F} f\left(\xi_x(h)\right)\left(\widehat{ws_\alpha}\right) =  \zeta\sum_{x\in k_F} \psi_{ws_\alpha}\left(\xi_x(h)\right),$$
with $\zeta$ as in the statement of the lemma.

Assume now that $w(\alpha)\in \Phi^-$.  We then have
\begin{eqnarray*}
(f\cdot \T_{\salpha})(h)(\widehat{w}) &  = &  \left(\sum_{x\in k_F} u_{\fat{\alpha}}(-[x])\salpha^{-1}.f\left(\xi_x(h)\right)\right)\left(\widehat{w}\right)\\
 & = &  \sum_{x\in k_F} f\left(\xi_x(h)\right)\left(\widehat{w}u_{\fat{\alpha}}(-[x])\salpha^{-1}\right)~ =~  \sum_{x\in k_F} f\left(\xi_x(h)\right)\left(\widehat{w}\salpha^{-1}u_{-\fat{\alpha}}([x])\right).
\end{eqnarray*}
By applying equation \eqref{unip} to the $x\neq 0$ terms, this expression becomes
\begin{equation}\label{mainlemI}
f\left(\xi_0(h)\right)(\widehat{w}\salpha^{-1}) + \sum_{x\in k_F^\times} f\left(\xi_x(h)\right)\left(\widehat{w}\salpha^{-1}u_{\fat{\alpha}}([x]^{-1})\alpha^\vee([x]^{-1})\salpha^{-1}u_{\fat{\alpha}}([x]^{-1})\right).
\end{equation}
Since $w(\alpha)\in \Phi^-$, we have $\widehat{w}\salpha^{-1}u_{\fat{\alpha}}([x]^{-1})\salpha\widehat{w}^{-1}\in U$.  Using that $f\left(\xi_x(h)\right) \in \ind_B^G(\chi)$ and the identity $\widehat{w}\salpha^{-1}\alpha^\vee([x]^{-1})\salpha^{-1}\widehat{w}^{-1} = \widehat{w}\alpha^\vee(-[x])\widehat{w}^{-1} = w(\alpha^\vee)(-[x])$, we obtain
\begin{eqnarray}
f\left(\xi_x(h)\right)\left(\widehat{w}\salpha^{-1}u_{\fat{\alpha}}([x]^{-1})\alpha^\vee([x]^{-1})\salpha^{-1}u_{\fat{\alpha}}([x]^{-1})\right) & = & f\left(\xi_x(h)\right)\left(\widehat{w}\salpha^{-1}\alpha^\vee([x]^{-1})\salpha^{-1}u_{\fat{\alpha}}([x]^{-1})\right) \notag \\ 
 & = & \chi\circ w(\alpha^\vee)(-[x]) \cdot f\left(\xi_x(h)\right)\left(\widehat{w}u_{\fat{\alpha}}([x]^{-1})\right). \label{mainlemII}
\end{eqnarray}
We now use twice the fact that $f$ is an $\ind_B^G(\chi)$-valued 1-cocycle to get 
\begin{eqnarray*}
f\left(\xi_x(h)\right)\left(\widehat{w}u_{\fat{\alpha}}([x]^{-1})\right) & = & \Big[u_{\fat{\alpha}}([x]^{-1}).f\left(\xi_x(h)\right)\Big]\left(\widehat{w}\right)\\
 & = & \Big[f\left(u_{\fat{\alpha}}([x]^{-1})\xi_x(h)\right) - f\left(u_{\fat{\alpha}}([x]^{-1})\right)\Big]\left(\widehat{w}\right)\\
 & = & \Big[u_{\fat{\alpha}}([x]^{-1})\xi_x(h)u_{\fat{\alpha}}(-[x]^{-1}).f\left(u_{\fat{\alpha}}([x]^{-1})\right)\\
 & &  + f\left(u_{\fat{\alpha}}([x]^{-1})\xi_x(h)u_{\fat{\alpha}}(-[x]^{-1})\right) - f\left(u_{\fat{\alpha}}([x]^{-1})\right)\Big]\left(\widehat{w}\right).
\end{eqnarray*}
Lemma \ref{conj2} and equation \eqref{cocycle} show that if $x\in k_F^\times$ and $h\in I_1 \cap w^{-1}Bw$, then $u_{\fat{\alpha}}([x]^{-1})\xi_x(h)u_{\fat{\alpha}}(-[x]^{-1}) \in I_1\cap w^{-1}Bw$ as well.  This implies that the first and third terms above cancel, which gives
\begin{equation}\label{mainlemIII}
f\left(\xi_x(h)\right)\left(\widehat{w}u_{\fat{\alpha}}([x]^{-1})\right) = f\left(u_{\fat{\alpha}}([x]^{-1})\xi_x(h)u_{\fat{\alpha}}(-[x]^{-1})\right)\left(\widehat{w}\right).
\end{equation}
Combining equations \eqref{mainlemI}, \eqref{mainlemII}, and \eqref{mainlemIII}, and using the definition of the tuple $(\psi_w)_{w\in W_0}$, we finally obtain
$$(f\cdot \T_{\salpha})(h)(\widehat{w})   =  \zeta\cdot \psi_{ws_\alpha}\left(\xi_0(h)\right)  + \sum_{x\in k_F^\times}\chi\circ w(\alpha^\vee)(-[x])\cdot \psi_w\left(u_{\fat{\alpha}}([x]^{-1})\xi_x(h)u_{\fat{\alpha}}(-[x]^{-1})\right).$$

\end{proof}

We now specialize the choice of $h$ in the statement of Lemma \ref{mainlemma}.

\begin{cor}
\label{Tsalphatorus}
Let $[f]\in \textnormal{H}^1(I_1,\ind_B^G(\chi))$ be associated to $(\psi_w)_{w\in W_0}\in \bigoplus_{w\in W_0}\Hom^{\textnormal{cts}}(I_1\cap w^{-1}Bw,C)$ via \eqref{shapiro}, and let $\fat{\alpha} = (\alpha,\ell)\in \Pi_\aff$, $w\in W_0$ and $t\in T_1$.  We then have
\begin{equation*}
(f\cdot \T_{\salpha})(t)(\widehat{w}) = \begin{cases}0 & \textnormal{if}~w(\alpha) \in \Phi^+, \\ 
\zeta\cdot\psi_{ws_\alpha}\left(t^{s_\alpha}\right) - \delta_{\chi\circ w(\alpha^\vee),1}\cdot \psi_w(t)  & \textnormal{if}~w(\alpha) \in \Phi^-. \\ \quad  + \sum_{x\in k_F^\times}\chi\circ w(\alpha^\vee)([x])\cdot \psi_w\left(u_{-\fat{\alpha}}((1 - \alpha(t))[x])\right) & \end{cases}
\end{equation*}
Here $\zeta := \chi(\widehat{w}\salpha^{-1}\widehat{ws_\alpha}^{-1})$.  
\end{cor}

\begin{proof}
If $w(\alpha)\in \Phi^+$, then we have
\begin{eqnarray*}
(f\cdot \T_{\salpha})(t)(\widehat{w}) & \stackrel{\eqref{PSaction}, \textnormal{Lem.}~\ref{conj}\eqref{conjd}}{=} &  \zeta\sum_{x\in k_F} \psi_{ws_\alpha}\left(t^{s_\alpha}u_{-\fat{\alpha}}((1 - \alpha(t)^{-1})[x])\right)\\
 & = & \zeta \sum_{x\in k_F} \Big(\psi_{ws_\alpha}(t^{s_\alpha}) + \psi_{ws_\alpha}\left(u_{-\fat{\alpha}}((1 - \alpha(t)^{-1})[x])\right)\Big)\\
 & = & \zeta\cdot\psi_{ws_\alpha}\left(u_{-\fat{\alpha}}\left((1 - \alpha(t)^{-1})\sum_{x\in k_F}[x]\right)\right)\\
 & = & 0,
\end{eqnarray*}
where the last line follows from the equality $\sum_{x\in k_F}[x] = 0$ (recall that $p\geq 3$).  On the other hand, if $w(\alpha)\in \Phi^-$, then
\begin{eqnarray*}
(f\cdot \T_{\salpha})(t)(\widehat{w}) & \stackrel{\sub{\eqref{PSaction},\textnormal{Lems.}}{\ref{conj}\eqref{conjd}, \ref{conj2}\eqref{conj2c}}}{=} & \zeta\cdot\psi_{ws_\alpha}(t^{s_\alpha}) + \sum_{x\in k_F^\times}\chi\circ w(\alpha^\vee)(-[x])\cdot \psi_w\big(tu_{-\fat{\alpha}}\left((\alpha(t) - 1)[x]\right)\big)\\
 & = & \zeta\cdot\psi_{ws_\alpha}(t^{s_\alpha}) + \left(\sum_{x\in k_F^\times}\chi\circ w(\alpha^\vee)(-[x]) \right) \psi_w(t)\\
 & & \qquad\qquad + \sum_{x\in k_F^\times}\chi\circ w(\alpha^\vee)(-[x])\cdot \psi_w\big(u_{-\fat{\alpha}}(\alpha(t) - 1)[x])\big).
\end{eqnarray*}
\end{proof}

\begin{lemma}
\label{Tomega}
Let $[f]\in \textnormal{H}^1(I_1,\ind_B^G(\chi))$ be associated to $(\psi_w)_{w\in W_0}\in \bigoplus_{w\in W_0}\Hom^{\textnormal{cts}}(I_1\cap w^{-1}Bw,C)$ via \eqref{shapiro}, and let $w\in W_0$ and $h\in I_1\cap w^{-1}Bw$.  Let $\omega\in \widetilde{\Omega}$.  We then have
$$(f\cdot \T_\omega)(h)(\widehat{w}) = \zeta\cdot\psi_{w\bar{\omega}^{-1}}\left(\omega h\omega^{-1}\right),$$
where $\zeta := \chi\big(\widehat{w\bar{\omega}^{-1}}\omega\widehat{w}^{-1}\big)^{-1}$.  In particular, when $\omega = t_0\in T_0$, we have 
$$(f\cdot \T_{t_0})(h)(\widehat{w}) = \chi\left(\widehat{w}t_0\widehat{w}^{-1}\right)^{-1}\cdot\psi_{w}\left(t_0 ht_0^{-1}\right).$$
\end{lemma}

\begin{proof}
Since $\omega$ normalizes $I_1$, equation \eqref{actoncocycsimple} gives
\begin{eqnarray*}
(f\cdot \T_\omega)(h)(\widehat{w}) & = & \left(\omega^{-1}.f(\omega h \omega^{-1})\right)(\widehat{w})\\
 & = & f(\omega h \omega^{-1})(\widehat{w}\omega^{-1})\\
 & = & \zeta\cdot f(\omega h \omega^{-1})\left(\widehat{w\bar{\omega}^{-1}}\right)\\
 & = & \zeta\cdot \psi_{w\bar{\omega}^{-1}}(\omega h \omega^{-1}).
\end{eqnarray*}
\end{proof}

\subsection{Height filtration}\label{hgtfilt}
We now define a certain filtration $\fil_\bullet$ on the vector space $\bigoplus_{w\in W_0}\Hom^{\textnormal{cts}}(I_1\cap w^{-1}Bw,C)\cong\textnormal{H}^1(I_1,\ind_B^G(\chi))$.

Let $h$ denote the Coxeter number of $\Phi$ (so that the height of the unique highest root of $\Phi$ is $h - 1$), and for $0\leq a \leq h - 1$ and $w\in W_0$, we define 
$$V_{w,a} := \left\langle U_{w^{-1}(\beta)}: \beta \in \Phi^+~\textnormal{with}~\textnormal{ht}(\beta) > a\right\rangle.$$ 
Since $\{\beta \in \Phi^+:\textnormal{ht}(\beta) > a\}$ is a closed subset of $\Phi^+$, $V_{w,a}$ is a subgroup of $w^{-1}Uw$ and multiplication induces a bijection $\prod_{\textnormal{ht}(\beta) > a}U_{w^{-1}(\beta)} \stackrel{\sim}{\longrightarrow} V_{w,a}$ (see \cite[Section II.1.7]{jantzen}).  We set $\fil_a := 0$ for $a < 0$, $\fil_a := \textnormal{H}^1(I_1,\ind_B^G(\chi))$ for $a > h - 1$, and for $0\leq a \leq h - 1$, we define
$$\fil_a := \Bigg\{(\psi_w)_{w\in W_0}\in \bigoplus_{w\in W_0}\Hom^{\textnormal{cts}}(I_1\cap w^{-1}Bw,C) : ~~ \psi_w\left(I_1 \cap V_{w,a}\right) = \{0\}~\textnormal{for all}~w\in W_0\Bigg\}.$$
In particular, $\fil_0$ consists of all homomorphisms supported only on $T_1$, and $\fil_{h - 1} = \textnormal{H}^1(I_1,\ind_B^G(\chi))$.

\begin{propn}
\label{filt}
The filtration $\fil_\bullet$ is $\cH$-stable.  
\end{propn}

\begin{proof}
Let $0\leq a \leq h - 1$, let $[f]\in \fil_a$ be associated to $(\psi_w)_{w\in W_0}$ via \eqref{shapiro}, and fix $w\in W_0$.

Suppose first that $\fat{\alpha} = (\alpha,\ell)\in \Pi_{\aff}$, let $\beta\in \Phi^+$ with $\textnormal{ht}(\beta) > a$ and $y\in \fp^{\mathbf{1}_{\Phi^-}(w^{-1}(\beta))}$ (so that $u_{w^{-1}(\beta)}(y)\in I_1\cap V_{w,a}$).  We examine the functions $\xi_x(u_{w^{-1}(\beta)}(y))$ associated to $\salpha$ appearing in equation \eqref{PSaction}.  
\begin{enumerate}[$\bullet$]
\item If $w(\alpha) \in \Phi^+$, then by applying Lemma \ref{conj}\eqref{conja} (when $w(\alpha) = \beta$) or Lemma \ref{conj}\eqref{conjb} (when $w(\alpha)\neq \beta$), we see that $\xi_x(u_{w^{-1}(\beta)}(y))$ is a product elements of the form $u_{s_\alpha w^{-1}(\gamma)}(y')$, with $\hgt(\gamma) > a$.  Therefore, we have
$$\xi_x\left(u_{w^{-1}(\beta)}(y)\right) \in I_1 \cap V_{ws_\alpha,a}.$$
\item If $w(\alpha) \in \Phi^-$, then by applying Lemma \ref{conj}\eqref{conjc} (when $w(\alpha) = -\beta$) or Lemma \ref{conj}\eqref{conjb} (when $w(\alpha) \neq -\beta$), we see that 
$$\xi_0\left(u_{w^{-1}(\beta)}(y)\right) = u_{s_\alpha w^{-1}(\beta)}(y') \in I_1 \cap V_{ws_\alpha,a}.$$ 
Likewise, by applying Lemma \ref{conj2}\eqref{conj2b} (when $w(\alpha) = -\beta$) or Lemma \ref{conj2}\eqref{conj2a} (when $w(\alpha) \neq -\beta$), we see that $u_{\fat{\alpha}}([x]^{-1})\xi_x(u_{w^{-1}(\beta)}(y))u_{\fat{\alpha}}(-[x]^{-1})$ is a product elements of the form $u_{w^{-1}(\gamma)}(y')$, with $\hgt(\gamma) > a$.  Therefore, we have
$$u_{\fat{\alpha}}([x]^{-1})\xi_x\left(u_{w^{-1}(\beta)}(y)\right)u_{\fat{\alpha}}(-[x]^{-1}) \in I_1 \cap V_{w,a}.$$
\end{enumerate}

\vspace{10pt}

\noindent These two points and equation \eqref{cocycle} show that if $h\in I_1 \cap V_{w,a}$, then 
$$\xi_x(h)\in I_1 \cap V_{ws_\alpha,a} \qquad \textnormal{and} \qquad u_{\fat{\alpha}}([x]^{-1})\xi_x\left(h\right)u_{\fat{\alpha}}(-[x]^{-1}) \in I_1 \cap V_{w,a}.$$
Equation \eqref{PSaction} then implies $[f]\cdot \T_{\salpha} \in \fil_a$.

Now let $\omega \in \widetilde{\Omega}$.  Then, for $h\in I_1\cap V_{w,a}$ and with notation as in Lemma \ref{Tomega}, we have
$$(f\cdot \T_\omega)(h)(\widehat{w}) = \zeta\cdot\psi_{w\bar{\omega}^{-1}}(\omega h\omega^{-1}) \in \zeta\cdot\psi_{w\bar{\omega}^{-1}}(I_1 \cap V_{w\bo^{-1},a}) = \{0\}.$$
Thus, we obtain $[f]\cdot\T_{\omega}\in \fil_a$.  
\end{proof}

For $a\in \bbZ$, define
$$\gr_a := \fil_a\big/\fil_{a - 1},$$
the associated graded module.  Given $f\in \fil_a$, we denote by $\bar{f}$ the image of $f$ in $\gr_a$.  The next step will be to determine some of the graded pieces $\gr_a$.

\subsection{Zeroth graded piece}
\begin{propn}\label{gr0}
Let $d := \dim_C(\Hom^{\textnormal{cts}}(T_1,C)) = \textnormal{rk}(\bG)([F:\qp] + \mathbf{1}_F(\zeta_p))$.  Then 
$$\gr_0 = \fil_0 \cong \textnormal{H}^0\big(I_1,\ind_B^G(\chi)\big)^{\oplus d} \cong \ind_{\cH_T}^{\cH}(\chi)^{\oplus d}.$$
\end{propn}

\begin{proof}
A basis for $\ind_B^G(\chi)^{I_1}$ is given by $\{\mathsf{f}_{w}\}_{w\in W_0}$, where $\mathsf{f}_{w}(b\widehat{w}'j) = \chi(b)\delta_{w,w'}$ for $b\in B, j\in I_1$, and $w'\in W_0$.  Let $\fat{\alpha} = (\alpha, \ell)\in \Pi_{\aff}$.  By the same calculation as in Lemma \ref{mainlemma}, we obtain
$$(\mathsf{f}_{w}\cdot\T_{\salpha})(\widehat{w}') = \begin{cases} 0 & \textnormal{if}~w'(\alpha)\in \Phi^+,\\ \zeta\cdot \mathsf{f}_{w}(\widehat{w's_\alpha}) + \left(\sum_{x\in k_F^\times}\chi\circ w'(\alpha^\vee)([x])\right)\mathsf{f}_{w}(\widehat{w}') & \textnormal{if}~w'(\alpha)\in \Phi^-,\end{cases}$$
where $\zeta = \chi(\widehat{w's_\alpha}\salpha\widehat{w}'^{-1})^{-1}$.  In particular, this gives
$$\mathsf{f}_{w}\cdot\T_{\salpha} = \begin{cases} \zeta'\cdot \mathsf{f}_{ws_\alpha} & \textnormal{if}~ w(\alpha) \in \Phi^+,\\ -\delta_{\chi\circ w(\alpha^\vee),1}\mathsf{f}_{w} & \textnormal{if}~w(\alpha)\in \Phi^-,\end{cases}$$
where $\zeta' = \chi(\widehat{ws_\alpha}\salpha^{-1}\widehat{w}^{-1})$.  Likewise, if $\omega \in \widetilde{\Omega}$, then the same computation as in Lemma \ref{Tomega} gives
$$(\mathsf{f}_{w}\cdot\T_{\omega})(\widehat{w}') = \mathsf{f}_{w}(\widehat{w}'\omega^{-1}) = \xi\cdot \mathsf{f}_{w}\left(\widehat{w'\bar{\omega}^{-1}}\right),$$
where $\xi = \chi\big(\widehat{w'\bar{\omega}^{-1}}\omega\widehat{w}'^{-1}\big)^{-1}$.  Thus,
$$\mathsf{f}_{w}\cdot\T_{\omega} = \xi'\cdot \mathsf{f}_{w\bar{\omega}},$$
where $\xi' = \chi(\widehat{w\bar\omega}\omega^{-1}\widehat{w}^{-1})$.

Now let $\theta\in \Hom^{\textnormal{cts}}(T_1,C)$, and for $w\in W_0$, let $\theta^{w}\in \Hom^{\textnormal{cts}}(I_1\cap w^{-1}Bw,C)$ denote the homomorphism defined by 
$$\theta^{w}(t) = \theta(\widehat{w}t\widehat{w}^{-1})\qquad\textnormal{if}~t\in T_1,\qquad \theta^{w}(u) = 0\qquad\textnormal{if}~u \in I_1\cap w^{-1}Uw.$$
By abuse of notation, we will also use $\theta^w$ to denote the element of $\textnormal{H}^1(I_1,\ind_B^G(\chi))$ associated to the tuple in $\bigoplus_{w'\in W_0} \Hom^{\textnormal{cts}}(I_1\cap w'^{-1}Bw',C)$ which is 0 in coordinate $w'\neq w$, and equal to the homomorphism $\theta^w$ in coordinate $w$.  Thus, for $t\in T_1$, Corollary \ref{Tsalphatorus} gives
$$(\theta^w\cdot \T_{\salpha})(t)(\widehat{w}') = \begin{cases}0 & \textnormal{if}~w'(\alpha) \in \Phi^+, \\ \zeta\delta_{w,w's_{\alpha}}\cdot\theta^w\left(t^{s_\alpha}\right) - \delta_{\chi\circ w'(\alpha^\vee),1}\delta_{w,w'}\cdot \theta^w(t)  & \textnormal{if}~ w'(\alpha)\in \Phi^-\end{cases}$$
(with $\zeta$ as above).  This gives
$$\theta^w\cdot \T_{\salpha} = \begin{cases}\zeta'\cdot\theta^{ws_\alpha} & \textnormal{if}~w(\alpha) \in \Phi^+, \\  -\delta_{\chi\circ w(\alpha^\vee),1}\theta^{w}  & \textnormal{if}~ w(\alpha)\in \Phi^-.\end{cases}$$
Likewise, if $\omega\in \widetilde{\Omega}$, then
$$(\theta^w\cdot \T_\omega)(t)(\widehat{w}') = \xi\delta_{w,w'\bo^{-1}}\cdot\theta^w(\bar{\omega} t\bar{\omega}^{-1}),$$
which implies
$$\theta^w\cdot \T_\omega = \xi'\cdot \theta^{w\bar{\omega}}.$$

Hence, we see that subspace spanned by $\{\theta^{w}\}_{w\in W_0}$ is $\cH$-stable, and isomorphic to $\ind_B^G(\chi)^{I_1}$ via $\theta^w\longmapsto \mathsf{f}_w$.  The formula for $d$ follows from the isomorphisms $\Hom^{\textnormal{cts}}(T_1,C) \cong \Hom^{\textnormal{cts}}(T_1/T_1^p,C)$ and $T_1/T_1^p \cong ((1 + \fp)/(1 + \fp)^p)^{\oplus \textnormal{rk}(\bG)}$, along with \cite[Ch. II, Proposition 5.7(i)]{neukirch:ant}.  
\end{proof}

\subsection{Basis for higher graded pieces}
In order to understand the higher graded pieces of $\textnormal{H}^1(I_1,\ind_B^G(\chi))$, it will be necessary to understand homomorphisms $I_1\cap w^{-1}Bw\longrightarrow C$, or equivalently, the abelianizations $(I_1\cap w^{-1}Bw)^{\textnormal{ab}}$.  These seem to be tricky to classify in general; for example, one can show, using \eqref{comm} and our assumption on $p$, that
$$(I_1\cap B)^{\textnormal{ab}} = T_1\oplus\bigoplus_{\alpha\in \Pi}u_{\alpha}(\fo/\fp)$$
and 
$$(I_1\cap w_\circ^{-1}Bw_\circ)^{\textnormal{ab}} = T_1\oplus\bigoplus_{\beta\in \Phi^-}u_{\beta}(\fp/\fp^2).$$
In order to handle this subtlety, we introduce a certain subset of $W_0$.

Let $\beta\in \Phi^+$.  We define a subset $\cW_\beta'$ of $W_0$ by
$$\cW_\beta' := \left\{w\in W_0: I_1 \cap U_{w^{-1}(\beta)}\not\subset [I_1\cap w^{-1}Bw, I_1\cap w^{-1}Bw]\right\}.$$
That is, $w\in \cW_\beta'$ if and only if the image of the root subgroup $I_1 \cap U_{w^{-1}(\beta)}$ is nontrivial in the abelianization $(I_1\cap w^{-1}Bw)^{\textnormal{ab}}$.  Some straightforward properties of $\cW_\beta'$:
\begin{enumerate}[$\bullet$]
\item The commutator subgroup $[I_1\cap w_\circ^{-1}Bw_\circ, I_1\cap w_\circ^{-1}Bw_\circ]$ is equal to the subgroup of $I_1\cap w_\circ^{-1}Bw_\circ$ generated by $u_{\beta}(\fp^2)$ for every $\beta\in \Phi^-$.  Therefore, $w_\circ\in \cW_\beta'$ for every $\beta\in \Phi^+$.  In particular, $\cW_\beta'$ is nonempty.  
\item If $\beta\in \Pi$, then $\cW_\beta' = W_0$ (this follows from equation \eqref{comm}).
\item If $w\in \cW_\beta'$ and $\omega\in \widetilde{\Omega}$, then using the fact that $\omega$ normalizes $I_1$ we get
\begin{eqnarray*}
I_1 \cap U_{\bo^{-1}w^{-1}(\beta)} & = & \omega^{-1}(I_1\cap U_{w^{-1}(\beta)})\omega \\ 
 & \not\subset & \omega^{-1}[I_1\cap w^{-1}Bw, I_1\cap w^{-1}Bw]\omega \\
 & = & [I_1\cap \bo^{-1}w^{-1}Bw\bo, I_1\cap \bo^{-1}w^{-1}Bw\bo].
\end{eqnarray*}
Therefore, $\overline{\Omega}$ acts on the right on $\cW_\beta'$, and $\cW_\beta'$ is a union of $\overline{\Omega}$-cosets.  
\end{enumerate}

Now let $\beta\in \Phi^+$ and $w\in \cW_\beta'$.  For $0\leq r \leq f - 1$, we define $\eta_{w^{-1}(\beta),r}^w \in \Hom(I_1\cap w^{-1}Bw,C)$ to be the homomorphism with support in $I_1\cap U_{w^{-1}(\beta)}$, and which satisfies
$$\eta_{w^{-1}(\beta),r}^w\left(u_{w^{-1}(\beta)}\left(y\right)\right) = \overline{\varpi^{-\mathbf{1}_{\Phi^-}(w^{-1}(\beta))}y}^{~p^r},$$
where $y\in \fp^{\mathbf{1}_{\Phi^-}(w^{-1}(\beta))}$ (note that $u_{w^{-1}(\beta)}(y)\in I_1\cap U_{w^{-1}(\beta)}$).  In particular, we have $\eta^w_{w^{-1}(\beta),r}\in \fil_{\hgt(\beta)}$.  Using these homomorphisms, we get an isomorphism of $C$-vector spaces
\begin{eqnarray*}
\textnormal{H}^1\big(I_1,\ind_B^G(\chi)\big) & \cong & \bigoplus_{w\in W_0}\Hom^{\textnormal{cts}}(I_1\cap w^{-1}Bw,C)\\
 &  = & \bigoplus_{w\in W_0}\Hom^{\textnormal{cts}}(T_1,C)\oplus\bigoplus_{\beta\in \Phi^+}\bigoplus_{w\in \cW_\beta'}\bigoplus_{r = 0}^{f - 1}C\eta^w_{w^{-1}(\beta),r}.
 \end{eqnarray*}
In what follows, we will abuse notation and use $\eta_{w^{-1}(\beta),r}^w$ to denote the element in $\textnormal{H}^1(I_1,\ind_B^G(\chi))$ associated to the tuple in $\bigoplus_{w'\in W_0}\Hom^{\textnormal{cts}}(I_1\cap w'^{-1}Bw',C)$ which is $0$ in coordinate $w'\neq w$ and the homomorphism $\eta_{w^{-1}(\beta),r}^w$ in coordinate $w$.

The next task will be to compute the action of $\cH$ on the elements $\eta^w_{w^{-1}(\beta),r}$ (or, more precisely, their images in the associated graded module).  

\begin{lemma}
\label{TsalphaGr}
Fix $\beta\in \Phi^+$, $w\in \cW_\beta'$, and $\fat{\alpha} = (\alpha,\ell) \in \Pi_\aff$.  If $w(\alpha) = -\beta$ or $w(\alpha)\in \Phi^+\smallsetminus\{\beta\}$, then $ws_\alpha\in \cW_\beta'$.  Consequently, we have the following equality in $\gr_{\textnormal{ht}(\beta)}$:
\begin{flushleft}
$\bar{\eta}_{w^{-1}(\beta),r}^w\cdot\T_{\salpha} = -\mathbf{1}_{\Phi^-}(w(\alpha))\delta_{(\chi\overline{\beta}^{-p^r})\circ w(\alpha^\vee),1}\cdot\bar{\eta}_{w^{-1}(\beta),r}^w$
\end{flushleft}
\begin{flushright}
$ + \zeta\cdot\left(d_{\alpha,s_\alpha w^{-1}(\beta)}\mathbf{1}_{\Phi^+\smallsetminus\{\beta\}}(w(\alpha))  - \delta_{w(\alpha),-\beta}\delta_{F,\qp}\right)\cdot\bar{\eta}_{s_\alpha w^{-1}(\beta),r}^{ws_\alpha}$,
\end{flushright}
where $\zeta := \chi(\widehat{ws_\alpha}\salpha^{-1}\widehat{w}^{-1})$.  
\end{lemma}

\begin{proof}
Since we are considering the image of $\eta_{w^{-1}(\beta),r}^w\cdot\T_{\salpha}$ in $\gr_{\textnormal{ht}(\beta)}$, it suffices to evaluate this function on the elements $u_{w'^{-1}(\gamma)}(y)$ for $w'\in W_0$, $\gamma\in \Phi^+$ with $\textnormal{ht}(\gamma) = \textnormal{ht}(\beta)$, and $y\in \fp^{\mathbf{1}_{\Phi^-}(w'^{-1}(\gamma))}$.  By Lemma \ref{mainlemma}, we have $(\eta_{w^{-1}(\beta),r}^w\cdot\T_{\salpha})(u_{w'^{-1}(\gamma)}(y))(\widehat{w}') = 0$ unless $w' = w$ or $w' = ws_\alpha$, and $\gamma = \beta$.  Since the homomorphism $\eta_{w^{-1}(\beta),r}^w$ is only supported on $I_1\cap U_{w^{-1}(\beta)}$, we get 
\begin{eqnarray*}
(\eta_{w^{-1}(\beta),r}^w\cdot \T_{\salpha})(u_{w^{-1}(\beta)}(y))(\widehat{w}) & \stackrel{\eqref{PSaction}}{=} & \begin{cases} 0 & (\textnormal{I})  \\
\sum_{x\in k_F^\times}\chi\circ w(\alpha^\vee)(-[x]) & (\textnormal{II}) \\ \qquad \cdot\eta_{w^{-1}(\beta),r}^w\left(u_{\fat{\alpha}}([x]^{-1})\xi_x(u_{w^{-1}(\beta)}(y))u_{\fat{\alpha}}(-[x]^{-1})\right) & \end{cases} \\
& \stackrel{\textnormal{Lem.}~\ref{conj2}\eqref{conj2a}}{=}  & \begin{cases} 0 & \\
\sum_{x\in k_F^\times}\chi\circ w(\alpha^\vee)([x])\cdot\eta_{w^{-1}(\beta),r}^w\left(u_{w^{-1}(\beta)}([x]^{-\langle w^{-1}(\beta),\alpha^\vee\rangle}y)\right) & \end{cases} \\
 & =  & \begin{cases} 0 &  \\
-\delta_{(\chi\overline{\beta}^{-p^r})\circ w(\alpha^\vee),1}\cdot\eta_{w^{-1}(\beta),r}^w\left(u_{w^{-1}(\beta)}(y)\right) & \end{cases} 
\end{eqnarray*}
where $w(\alpha)\in \Phi^+$ in case $(\textnormal{I})$ and $w(\alpha)\in \Phi^-$ in case $(\textnormal{II})$.  Similarly,
\begin{flushleft}
$(\eta_{w^{-1}(\beta),r}^w\cdot \T_{\salpha})(u_{s_\alpha w^{-1}(\beta)}(y))(\widehat{ws_\alpha})$
\end{flushleft}

\begin{eqnarray*}
 & \stackrel{\eqref{PSaction}}{=} & \begin{cases}\zeta\sum_{x\in k_F}\eta_{w^{-1}(\beta),r}^w\left(\xi_x(u_{s_\alpha w^{-1}(\beta)}(y))\right) & (\textnormal{I}) \\ \zeta\sum_{x\in k_F}\eta_{w^{-1}(\beta),r}^w\left(\xi_x(u_{s_\alpha w^{-1}(\beta)}(y))\right) & (\textnormal{II}) \\ \zeta\cdot \eta_{w^{-1}(\beta),r}^w\left(\xi_0(u_{s_\alpha w^{-1}(\beta)}(y))\right) & (\textnormal{III}) \\ \zeta\cdot \eta_{w^{-1}(\beta),r}^w\left(\xi_0(u_{s_\alpha w^{-1}(\beta)}(y))\right) & (\textnormal{IV})\end{cases} \\
 & \stackrel{\sub{\textnormal{Lems.}}{\ref{conj}\eqref{conja},\eqref{conjb},\eqref{conjc}}}{=} & \begin{cases}  \zeta\cdot\eta^w_{w^{-1}(\beta),r}\left(u_{-\fat{\alpha}}\left(\sum_{x\in k_F}[x + \overline{\varpi^{-\ell}y}] - [x] - \varpi^{-\ell}y\right)\right)  &  \\ 
\zeta\sum_{x\in k_F} \eta^w_{w^{-1}(\beta),r}\left(u_{w^{-1}(\beta)}(d_{\alpha,s_\alpha w^{-1}(\beta)}\varpi^{\ell\langle w^{-1}(\beta),\alpha^\vee\rangle}y)\right) & \\
\zeta\cdot\eta^w_{w^{-1}(\beta),r}\left(u_{w^{-1}(\beta)}(-\varpi^{2\ell}y)\right) & \\
\zeta\cdot \eta_{w^{-1}(\beta),r}^w\left(u_{w^{-1}(\beta)}(d_{\alpha,s_\alpha w^{-1}(\beta)}\varpi^{\ell\langle w^{-1}(\beta),\alpha^\vee\rangle}y)\right)   & \end{cases}\\
& = & \begin{cases}  \zeta\cdot\eta^w_{w^{-1}(\beta),r}\left(u_{-\alpha}(-q\varpi^{-2\ell}y)\right)  &  \\ 
0 & \\
0 & \\
\zeta d_{\alpha,s_\alpha w^{-1}(\beta)}\cdot \eta_{w^{-1}(\beta),r}^w\left(u_{w^{-1}(\beta)}(\varpi^{\ell\langle w^{-1}(\beta),\alpha^\vee\rangle}y)\right)   & \end{cases}\\
 & = & \begin{cases}  -\zeta\delta_{F,\qp}\cdot\eta^w_{w^{-1}(\beta),r}\left(u_{-\alpha}(p^{1-2\ell}y)\right)  &  \\ 
0 & \\
0 & \\
\zeta d_{\alpha,s_\alpha w^{-1}(\beta)}\cdot \eta_{w^{-1}(\beta),r}^w\left(u_{w^{-1}(\beta)}(\varpi^{\ell\langle w^{-1}(\beta),\alpha^\vee\rangle}y)\right)   & \end{cases} 
\end{eqnarray*}
where $w(\alpha) = -\beta$ in case $(\textnormal{I})$; $w(\alpha) \in \Phi^-\smallsetminus\{-\beta\}$ in case $(\textnormal{II})$; $w(\alpha) = \beta$ in case $(\textnormal{III})$; and $w(\alpha) \in \Phi^+\smallsetminus\{\beta\}$ in case $(\textnormal{IV})$.  Furthermore, $\zeta := \chi(\widehat{ws_\alpha}\salpha^{-1}\widehat{w}^{-1})$, the third equality follows from the fact that $u_{w^{-1}(\beta)}(-\varpi^{2\ell}y)$ is contained in the derived subgroup of $I_1\cap w^{-1}Bw$ when $w(\alpha) = \beta$, and the last equality follows from the fact that $\textnormal{val}(q) = [F:\qp]$.  This calculation shows that the homomorphism $(f\cdot \T_{\salpha})(-)(\widehat{ws_\alpha})\in \Hom^{\textnormal{cts}}(I_1\cap s_\alpha w^{-1}Bws_\alpha,C)$ takes nonzero values if $w(\alpha) = -\beta$ or $w(\alpha)\in \Phi^+\smallsetminus\{\beta\}$, and thus $ws_\alpha\in \cW_\beta'$ in these cases.  This gives the claim.  
\end{proof}

\begin{lemma}\label{TomegaGr}
Let $\beta\in \Phi^+$ and $w\in \cW_\beta'$.   
\begin{enumerate}[(a)]
\item If $t_0 \in T_0$, then 
$$\bar{\eta}^w_{w^{-1}(\beta),r}\cdot \T_{t_0} = (\chi\overline{\beta}^{-p^r})\left(\widehat{w}t_0\widehat{w}^{-1}\right)^{-1}\cdot\bar{\eta}^w_{w^{-1}(\beta),r}.$$
\item If $\omega\in \widetilde{\Omega}$ and $\omega\widehat{\bo}^{-1} = \lambda(\varpi)$ for some $\lambda\in X_*(\bT)$, then
$$\bar{\eta}^w_{w^{-1}(\beta),r}\cdot\T_{\omega} = \chi\left(\widehat{w\bo}\omega^{-1}\widehat{w}^{-1}\right)d_{\bo,\bo^{-1}w^{-1}(\beta)}\cdot\bar{\eta}^{w\bo}_{\bo^{-1}w^{-1}(\beta),r}.$$
\end{enumerate}
\end{lemma}

\begin{proof}  
If $\omega$ is an arbitrary element of $\widetilde{\Omega}$, then $\eta^w_{w^{-1}(\beta),r}\cdot\T_\omega$ is associated to an element of $\Hom^{\textnormal{cts}}(I_1\cap \bo^{-1}w^{-1}Bw\bo,C)$, which is supported on $I_1\cap U_{\bo^{-1}w^{-1}(\beta)}$.  In particular, $\eta^w_{w^{-1}(\beta),r}\cdot\T_\omega$ is a scalar multiple of $\eta^{w\bo}_{\bo^{-1}w^{-1}(\beta),r}$, and it suffices to evaluate $(\eta^w_{w^{-1}(\beta),r}\cdot\T_{\omega})(u_{\bo^{-1}w^{-1}(\beta)}(y))(\widehat{w\bo})$.  

We leave case (a) as a straightforward exercise.  Thus, assume that $\omega\in \widetilde{\Omega}$ is such that $\omega\widehat{\bo}^{-1}$ is of the form $\lambda(\varpi)$ for some $\lambda\in X_*(\bT)$.  Using Lemma \ref{Tomega} and the definition of the structure constants, we get
\begin{eqnarray*}
(\eta^w_{w^{-1}(\beta),r}\cdot\T_{\omega})(u_{\bo^{-1}w^{-1}(\beta)}(y))(\widehat{w\bo}) & = & \chi\left(\widehat{w\bo}\omega^{-1}\widehat{w}^{-1}\right)\cdot\eta^w_{w^{-1}(\beta),r}\left(\omega u_{\bo^{-1}w^{-1}(\beta)}(y)\omega^{-1}\right)\\
& = & \chi\left(\widehat{w\bo}\omega^{-1}\widehat{w}^{-1}\right)\cdot\eta^w_{w^{-1}(\beta),r}\left(\lambda(\varpi)\widehat{\bo} u_{\bo^{-1}w^{-1}(\beta)}(y)\widehat{\bo}^{-1}\lambda(\varpi)^{-1}\right)\\
& = & \chi\left(\widehat{w\bo}\omega^{-1}\widehat{w}^{-1}\right)\cdot\eta^w_{w^{-1}(\beta),r}\left(u_{w^{-1}(\beta)}(d_{\bo,\bo^{-1}w^{-1}(\beta)}\varpi^{\langle w^{-1}(\beta),\lambda\rangle}y)\right)\\
& = & \chi\left(\widehat{w\bo}\omega^{-1}\widehat{w}^{-1}\right)d_{\bo,\bo^{-1}w^{-1}(\beta)}\cdot\eta^w_{w^{-1}(\beta),r}\left(u_{w^{-1}(\beta)}(\varpi^{\langle w^{-1}(\beta),\lambda\rangle}y)\right)\\
& = & \chi\left(\widehat{w\bo}\omega^{-1}\widehat{w}^{-1}\right)d_{\bo,\bo^{-1}w^{-1}(\beta)}\cdot\eta^{w\bo}_{\bo^{-1}w^{-1}(\beta),r}\left(u_{\bo^{-1}w^{-1}(\beta)}(y)\right),
\end{eqnarray*}
where the last line follows from the fact that $\omega$ normalizes $I_1$.  Taking the image in $\gr_{\textnormal{ht}(\beta)}$ gives the claim.
\end{proof}

\begin{cor}\label{Tsalpha*Gr}
Let $\beta\in \Phi^+, w\in \cW_\beta'$, and $\fat{\alpha} = (\alpha,\ell) \in \Pi_\aff$.  We then have the following equality in $\gr_{\hgt(\beta)}$:
\begin{flushleft}
$\bar{\eta}_{w^{-1}(\beta),r}^w\cdot\T_{\salpha}^* = \mathbf{1}_{\Phi^+}(w(\alpha))\delta_{(\chi\overline{\beta}^{-p^r})\circ w(\alpha^\vee),1}\cdot\bar{\eta}_{w^{-1}(\beta),r}^w$
\end{flushleft}
\begin{flushright}
$ + \zeta\cdot\left(d_{\alpha,s_\alpha w^{-1}(\beta)}\mathbf{1}_{\Phi^+\smallsetminus\{\beta\}}(w(\alpha))  - \delta_{w(\alpha),-\beta}\delta_{F,\qp}\right)\bar{\eta}_{s_\alpha w^{-1}(\beta),r}^{ws_\alpha}$,
\end{flushright}
where $\zeta := \chi(\widehat{ws_\alpha}\salpha^{-1}\widehat{w}^{-1})$.  
\end{cor}

\begin{proof}
Since $\T_{\salpha}^* = \T_{\salpha} - c_{\fat{\alpha}} = \T_{\salpha} - \sum_{x\in k_F^\times} \T_{\alpha^\vee([x])} $, the action of $\T_{\salpha}^*$ is given by combining Lemmas \ref{TsalphaGr} and \ref{TomegaGr}.  
\end{proof}

\subsection{First graded piece}\label{gr1sect}

Fix now $\beta\in \Pi$ and recall that $\bM_\beta$ is the standard Levi subgroup associated to $\{\beta\}$.  We let $0\leq r \leq f - 1$, and define $\fn_{F,\beta,r}$ to be the two-dimensional $\cH_{M_\beta}$-module with $C$-basis $\{v_1, v_2\}$, and $\cH_{M_\beta}$-action given by
\begin{eqnarray*}
 v_1\cdot\T_{\widehat{s_\beta}}^{M_\beta} & = & 0,\\ 
 v_2\cdot\T_{\widehat{s_\beta}}^{M_\beta} & = & -\chi\circ\beta^\vee(-1)\delta_{F,\qp}\cdot v_1 - \delta_{(\chi\overline{\beta}^{-p^r})\circ\beta^\vee,1}\cdot v_2,\\ 
 v_1\cdot\T_{\widehat{s_{\beta^*}}}^{M_\beta} & = & -\delta_{(\chi\overline{\beta}^{-p^r})\circ\beta^\vee,1}\cdot v_1 - \chi\circ\beta^\vee(-\varpi)\delta_{F,\qp}\cdot v_2,\\
 v_2\cdot\T_{\widehat{s_{\beta^*}}}^{M_\beta} & = & 0.
 \end{eqnarray*}
 Here $\beta^*$ denotes the simple affine root $(-\beta,1)$ of $\bM_\beta$, so that $\widehat{s_{\beta^*}} = \widehat{s_\beta}\beta^\vee(-\varpi)$.   Moreover, if $t\in T$ satisfies $\langle\beta,\nu(t)\rangle = 0$, we define 
\begin{eqnarray*}
 v_1\cdot\T_{t}^{M_\beta} & = & (\chi\overline{\beta}^{-p^r})(t)^{-1}\cdot v_1,\\
 v_2\cdot\T_{t}^{M_\beta} & = & (\chi^{s_\beta}\overline{\beta}^{p^r})(t)^{-1}\cdot v_2,
\end{eqnarray*}
while if $t\in T$ satisfies $\langle\beta,\nu(t)\rangle = 1$, we define
\begin{eqnarray*}
 v_1\cdot\T_{t\widehat{s_\beta}}^{M_\beta} & = & -(\chi\overline{\beta}^{-p^r})(t)^{-1}\cdot v_2,\\
 v_2\cdot\T_{t\widehat{s_\beta}}^{M_\beta} & = & -(\chi^{s_\beta}\overline{\beta}^{p^r})(t\beta^\vee(-1))^{-1}\cdot v_1.
\end{eqnarray*}
(Note that the last two sets of equations define the action of $\widetilde{\Omega}_{\bM_{\beta}}$ on $\fn_{F,\beta,r}$.)  We have
$$\fn_{F,\beta,r} \cong \begin{cases}\ind_{\cH_{T}}^{\cH_{M_\beta}}\big(\chi^{s_\beta}\overline{\beta}\big) & (\textnormal{I})\\ \chi^{s_\beta}\overline{\beta}\otimes(\textnormal{nonsplit extension of}~\chi_{\textnormal{triv}}~ \textnormal{by}~ \chi_{\textnormal{sign}}^\star) & (\textnormal{II}) \\ \textnormal{semisimple supersingular module} & (\textnormal{III}) \end{cases}$$
where $F = \qp$ and $(\chi^{s_\beta}\overline{\beta})\circ\beta^\vee$ is a nontrivial character of $\qp^\times$ in case $(\textnormal{I})$; $F = \qp$ and $(\chi^{s_\beta}\overline{\beta})\circ\beta^\vee$ is the trivial character of $\qp^\times$ in case $(\textnormal{II})$; and $F\neq \qp$ in case $(\textnormal{III})$.  Here $\chi_{\textnormal{sign}}^\star$ is the character of $\cH_{M_\beta}$ given by the $\cH_{M_\beta}$-action on the $(I_1\cap M_\beta)$-invariants of the Steinberg representation of $M_\beta$ (note: $\chi_{\textnormal{sign}}^\star$ is a twist of $\chi_{\textnormal{sign}}$, and the two characters agree on $\cH_{M_\beta,\aff}$).

\begin{propn}\label{gr1}
We have
$$\gr_1 \cong \bigoplus_{\beta\in \Pi}\bigoplus_{r = 0}^{f - 1} \ind_{\cH_{M_{\beta}}}^{\cH}\left(\fn_{F,\beta,r}\right).$$
\end{propn}

\begin{proof}
Since $\cW_\beta' = W_0$ for $\beta\in \Pi$, we have an isomorphism of $C$-vector spaces
$$\gr_1 = \bigoplus_{\beta\in \Pi}\bigoplus_{r = 0}^{f - 1}\bigoplus_{w\in \cW_\beta'}C\bar{\eta}^w_{w^{-1}(\beta),r}  = \bigoplus_{\beta\in \Pi}\bigoplus_{r = 0}^{f - 1}\bigoplus_{w\in W_0}C\bar{\eta}^w_{w^{-1}(\beta),r}.$$
By Lemmas \ref{TsalphaGr} and \ref{TomegaGr}, the space $\bigoplus_{w\in W_0}C\bar{\eta}^w_{w^{-1}(\beta),r}$ is stable by the action of $\cH$, and therefore it suffices to prove $\bigoplus_{w\in W_0}C\bar{\eta}^w_{w^{-1}(\beta),r}\cong \ind_{\cH_{M_{\beta}}}^{\cH}\left(\fn_{F,\beta,r}\right)$ as $\cH$-modules.  For the proof of this, see Appendix \ref{app}.
\end{proof}

\subsection{Higher graded pieces}

In general, the higher graded pieces of $\textnormal{H}^1(I_1,\ind_B^G(\chi))$ seem difficult to understand, due our lack of understanding of the sets $\cW_\beta'$.  In the next section, we will attempt to give some more information in the case $\bG = \bG\bL_n$.  In general, we have the following.

\begin{propn}\label{ssingprop}
If $h > 2$ (that is, if $\Phi$ is not of type $A_1$), then $\gr_{h - 1}$ contains a supersingular $\cH$-module.  
\end{propn}

\begin{proof}
Let $\alpha_{0}\in \Phi^+$ denote the highest root of $\Phi$ (so that $\textnormal{ht}(\alpha_{0}) = h - 1$).  Then $w_\circ\in \cW_{\alpha_0}'$ and $w_\circ(\alpha_{0}) = -\alpha_{0}$.  If $\fat{\alpha} = (\alpha,\ell)\in \Pi_{\aff},$ then $\alpha\in \Pi$ or $\alpha = -\alpha_{0}$, and since $\Phi$ is not of type $A_1$, we have $w_\circ(\alpha)\not\in\Phi^+\smallsetminus\{\alpha_{0}\}$ and $w_\circ(\alpha) \neq - \alpha_0$.  Therefore by Lemma \ref{TsalphaGr}, we have 
\begin{equation}
\label{ssingeqn}
\bar{\eta}_{w_{\circ}^{-1}(\alpha_0),r}^{w_{\circ}}\cdot\T_{\salpha} = -\mathbf{1}_{\Phi^+}(\alpha)\delta_{(\chi\overline{\alpha_0}^{-p^r})\circ w_\circ(\alpha^\vee),1}\cdot \bar{\eta}_{w_{\circ}^{-1}(\alpha_0),r}^{w_{\circ}},
\end{equation}
and using Lemma \ref{TomegaGr} shows that the one-dimensional vector space spanned by $\bar{\eta}_{w_{\circ}^{-1}(\alpha_0),r}^{w_{\circ}}$ is $\cH_\aff$-stable.  We claim this $\cH_\aff$-character is supersingular.  To verify this, it suffices to check that it is not equal to a twist of the trivial character (note that equation \eqref{ssingeqn} gives $\bar{\eta}_{w_{\circ}^{-1}(\alpha_0),r}^{w_{\circ}}\cdot\T_{\widehat{s_{(-\alpha_0,1)}}} = 0$, so the $\cH_\aff$-character cannot be a twist of the sign character).  If this were the case, then the character of $\T_{t_0}$ on $\bar{\eta}_{w_{\circ}^{-1}(\alpha_0),r}^{w_{\circ}}$, given by $\T_{t_0}\longmapsto (\chi\overline{\alpha_0}^{-p^r})(\widehat{w_\circ}t_0\widehat{w_\circ}^{-1})^{-1}$, would satisfy $(\chi\overline{\alpha_0}^{-p^r})\circ w_\circ(\alpha^\vee)([x]) = 1$ for all $\alpha\in \Pi$ and $x\in k_F^\times$.  However, equation \eqref{ssingeqn} would then imply $\bar{\eta}_{w_{\circ}^{-1}(\alpha_0),r}^{w_{\circ}}\cdot\T_{\widehat{s_\alpha}} = -\bar{\eta}_{w_{\circ}^{-1}(\alpha_0),r}^{w_{\circ}}$ for $\alpha\in \Pi$, which contradicts that $\bar{\eta}_{w_{\circ}^{-1}(\alpha_0),r}^{w_{\circ}}$ is a twist of the trivial character.  
\end{proof}

\begin{cor}
The $\cH$-module $\textnormal{H}^1(I_1,\ind_B^G(\chi))$ contains no supersingular subquotients if and only if $F = \qp$ and $\Phi$ is of type $A_1$.  
\end{cor}

\begin{proof}
This follows from Propositions \ref{ssingprop}, \ref{gr1}, and \ref{gr0}.
\end{proof}

\begin{rmk}
If $h = 2$ (that is, if $\Phi$ is of type $A_1$), then it may happen that $\textnormal{H}^i(I_1,\ind_B^G(\chi))$ has no supersingular subquotients for all $i$.  For an example, let $\bG = \bS\bL_2$ over $\qp$, let $\bB$ denote the upper triangular Borel subgroup, and let $I_1$ denote the ``upper triangular mod-$p$'' pro-$p$-Iwahori subgroup of $G = \textnormal{SL}_2(\qp)$.  Since $p\geq 5$, $I_1$ is torsion-free.  Write $\Pi = \{\alpha\}$; we then have
\begin{eqnarray*}
\textnormal{H}^0\big(I_1,\ind_B^G(\chi)\big) & \cong & \ind_{\cH_T}^{\cH}(\chi),\\
\textnormal{H}^1\big(I_1,\ind_B^G(\chi)\big) & \cong & \textnormal{(possibly split) extension of}~ \fn_{\qp,\alpha,0}~\textnormal{by}~\ind_{\cH_T}^{\cH}(\chi),\\
\textnormal{H}^2\big(I_1,\ind_B^G(\chi)\big) & \cong & \fn_{\qp,\alpha,0},\\
\textnormal{H}^i\big(I_1,\ind_B^G(\chi)\big) & = & 0~\textnormal{for}~i\geq 3.
\end{eqnarray*}
The first isomorphism follows from equation \eqref{h0}, the second from Propositions \ref{gr0} and \ref{gr1}, and the last from the isomorphism of $C$-vector spaces
$$\textnormal{H}^i\big(I_1,\ind_B^G(\chi)\big) \cong \textnormal{H}^i\left(I_1\cap B,C\right)~\oplus~ \textnormal{H}^i\left(I_1\cap s_\alpha^{-1} Bs_\alpha,C\right) = 0~\textnormal{for}~i\geq 3$$
(note that $I_1\cap B$ and $I_1\cap s_\alpha^{-1}Bs_\alpha$ are both compact, torsion-free, $p$-adic analytic groups of dimension 2, and therefore have cohomological dimension 2, cf. \cite[Section 4.5, Example 3]{serre:galoiscoh}).

We sketch a proof of the third isomorphism above.  The case $\chi = \overline{\alpha}$ is handled using a similar method as \cite[Corollary 10.10]{paskunas:montrealfunctor}, so we assume $\chi\neq \overline{\alpha}$.  
\begin{enumerate}[(a)]
\item First, note that we have an isomorphism $\textnormal{H}^2(I_1,\ind_B^G(\chi)) \cong \textnormal{H}^2(I_1\cap B,C)\oplus \textnormal{H}^2(I_1\cap s_\alpha^{-1} Bs_\alpha,C)$.  Since both $I_1\cap B$ and $I_1\cap s_\alpha^{-1}Bs_\alpha$ are Poincar\'e groups of dimension 2, we conclude that $\textnormal{H}^2(I_1,\ind_B^G(\chi))$ is two-dimensional.  
\item Since we have an equivalence of categories between $T$-representations generated by their $T_1$-invariant vectors and $\cH_T$-modules, we get a spectral sequence (cf. \cite[equation (32)]{paskunas:exts})
$$E_2^{i,j} = \Ext^i_{\cH_T}\left(\tau^{T_1},\textnormal{H}^j(T_1,\pi)\right) \Longrightarrow \Ext^{i + j}_T\left(\tau,\pi\right)$$
where $\tau$ and $\pi$ are smooth $T$-representations, and $\tau$ is generated by $\tau^{T_1}$.  Since $\cH_T$ has global dimension 1, and $T_1\cong \zp$ has cohomological dimension 1, the spectral sequence above degenerates at $E_2$ when $\tau$ and $\pi$ are equal to characters of $T$.  From this and the calculation of Ext-spaces for $\cH_T$-modules, we obtain
$$\dim_C\left(\Ext^i_{T}(\chi,\chi)\right) = \binom{2}{i},\quad\textnormal{and}\quad\dim_C\left(\Ext^i_{T}(\chi',\chi)\right) = 0~\textnormal{if}~\chi'\neq\chi.$$
\item Let $\textnormal{H}^i\textnormal{Ord}_B$ denote Emerton's $\delta$-functors of derived ordinary parts from locally admissible $G$-representations over $C$ to locally admissible $T$-representations over $C$ (\cite{emerton:ordII}).  Since any admissible representation of $\textnormal{SL}_2(\qp)$ can be embedded into an admissible representations of $\textnormal{GL}_2(\qp)$, one can proceed as in \cite{emertonpaskunas} to conclude that the functors $\textnormal{H}^i\textnormal{Ord}_B$ for $i\geq 1$ are effaceable, and therefore isomorphic to $\textnormal{R}^i\textnormal{Ord}_B$.  
Then \cite[(3.7.4)]{emerton:ordII} gives a spectral sequence
\begin{equation}\label{h2ss1}
E_2^{i,j} = \Ext^i_{T}\left(\tau,\textnormal{H}^j\textnormal{Ord}_B(\pi)\right) \Longrightarrow \Ext^{i + j}_G\left(\ind_B^G(\tau^{s_\alpha}),\pi\right),
\end{equation}
where $\tau$ is a locally admissible $T$-representation, and $\pi$ is a locally admissible $G$-representation.  
\item Since we have an equivalence of categories between $G$-representations generated by their $I_1$-invariant vectors and $\cH$-modules (\cite[Corollary 5.3]{koziol:glnsln}), we get a spectral sequence (cf. \cite[equation (32)]{paskunas:exts} again)
$$E_2^{i,j} = \Ext^i_{\cH}\left(\tau^{I_1},\textnormal{H}^j(I_1,\pi)\right) \Longrightarrow \Ext^{i + j}_G\left(\tau,\pi\right)$$
where $\tau$ and $\pi$ are smooth $G$-representations, and $\tau$ is generated by $\tau^{I_1}$.  In particular, if $\tau = \ind_B^G(\chi')$ and $\pi = \ind_B^G(\chi)$, then the $\cH$-module $\tau^{I_1}$ has projective dimension 1 (\cite[Remarks 3.1 and 8.2]{koziol:homdim}), and $\textnormal{H}^j(I_1,\pi)$ vanishes for $j\geq 3$.  In this case, the spectral sequence above degenerates at $E_2$, and gives
\begin{equation}\label{h2abut}
\Ext^1_{\cH}\left(\tau^{I_1},\textnormal{H}^2\big(I_1,\ind_B^G(\chi)\big)\right)  \cong \Ext^3_G\left(\tau,\ind_B^G(\chi)\right).
\end{equation}
\end{enumerate}

Now set $\tau := \ind_B^G(\chi^{s_\alpha}\overline{\alpha})$ (note that $\tau$ is irreducible since $\chi^{s_\alpha}\overline{\alpha}$ is not equal to the trivial character).  Since $\textnormal{Ord}_B(\ind_B^G(\chi)) = \chi^{s_\alpha}$, $\textnormal{H}^1\textnormal{Ord}_B(\ind_B^G(\chi)) = \chi\overline{\alpha}^{-1}$, and $\textnormal{H}^j\textnormal{Ord}_B(\ind_B^G(\chi)) = 0$ for $j\geq 2$, we get
\begin{eqnarray*}
\Ext^1_{\cH}\left(\tau^{I_1},\textnormal{H}^2\big(I_1,\ind_B^G(\chi)\big)\right) & \stackrel{\eqref{h2abut}}{\cong} & \Ext^3_G\left(\tau,\ind_B^G(\chi)\right) \\
 & \stackrel{\eqref{h2ss1}}{\cong} & \Ext^2_T\left(\chi\overline{\alpha}^{-1},\chi\overline{\alpha}^{-1}\right) \stackrel{\textnormal{part (b)}}{\neq} 0.
\end{eqnarray*}
By modifying the proof of \cite[Proposition 3.6(2)]{abe:extensions}, we get
$$\Ext^1_{\cH}\left(\tau^{I_1},\textnormal{H}^2\big(I_1,\ind_B^G(\chi)\big)\right)\cong \Ext^1_{\cH_{T}}\left(\chi^{s_\alpha}\overline{\alpha},R_{\cH_T}^{\cH}\big(\textnormal{H}^2\big(I_1,\ind_B^G(\chi)\big)\big)\right)\neq 0.$$
Since $\Ext^i_{\cH_T}(\chi',\chi)\neq 0$ implies $\chi = \chi'$, again using \emph{loc. cit.} gives
$$\Hom_{\cH}\left(\tau^{I_1},\textnormal{H}^2\big(I_1,\ind_B^G(\chi)\big)\right) \cong \Hom_{\cH_{T}}\left(\chi^{s_\alpha}\overline{\alpha},R_{\cH_T}^{\cH}\big(\textnormal{H}^2\big(I_1,\ind_B^G(\chi)\big)\big)\right) \neq 0.$$
Since $\tau^{I_1}$ and $\textnormal{H}^2(I_1,\ind_B^G(\chi))$ are both two-dimensional and $\tau^{I_1}$ is a simple $\cH$-module, we conclude
$$\textnormal{H}^2\big(I_1,\ind_B^G(\chi)\big) \cong \tau^{I_1} \cong \ind_{\cH_T}^{\cH}(\chi^{s_\alpha}\overline{\alpha}).$$
\end{rmk}

\bigskip

\section{Principal series for $\textnormal{GL}_n$}\label{prinseriesgln}

We now specialize to the case $\bG = \bG\bL_n$.  Since the case $n = 2$ is covered by the previous results, we assume $n\geq 3$.  We let $\bT$ denote the diagonal maximal torus, and for $1\leq j , k \leq n$, $j\neq k$, we define $\alpha_{j,k}\in X^*(\bT)$ by
$$\alpha_{j,k}\left(\begin{pmatrix}t_1 & & &  \\ & t_2 & & \\ & & \ddots & \\ & & & t_n\end{pmatrix}\right) = t_jt_k^{-1}.$$
We then have $\Phi = \{\alpha_{j,k}\}_{1\leq j, k \leq n, j\neq k}$.  We define $\bB$ to be the upper triangular Borel subgroup, and $I_1$ to be the subgroup of $\textnormal{GL}_n(\fo)$ which is upper-triangular and unipotent modulo $\varpi$.  Then we have $\Phi^+ = \{\alpha_{j,k}\}_{1\leq j < k \leq n}$ and 
$$\Pi = \{\alpha_i\}_{1\leq i \leq n - 1},$$
where $\alpha_i := \alpha_{i,i+1}$.  Given $1\leq i \leq n - 1$, we set $\Pi_{(i)}:= \Pi\smallsetminus\{\alpha_i\}, \bM_{(i)} := \bM_{\Pi_{(i)}}, \Phi_{(i)} := \Phi_{\bM_{(i)}},$ and $\Phi_{(i)}^+ := \Phi_{\bM_{(i)}}^+$.

We have $W_0 \cong \fS_n$, the symmetric group on $n$ letters, with the obvious isomorphism.  We let $s_i := s_{\alpha_i}$ denote the simple reflection associated to $\alpha_i$, so that $s_i = (i, i + 1)$, and we also define $W_{(i),0} := W_{\bM_{(i)},0}, W_0^{(i)}:=W_0^{\bM_{(i)}}$, and ${}^{(i)}W_0:={}^{\bM_{(i)}}W_0$.

Set
$$\omega := \begin{pmatrix} 0 & 1 & & \\  &  \ddots & \ddots & \\ & & \ddots & 1 \\ \varpi & & & 0\end{pmatrix}\in N_G(T);$$
the element $\omega$ generates $\Omega \cong \bbZ$.  Its reduction $\bo$ in $W_0$ is the $n$-cycle $(n, n - 1, \ldots, 2, 1) = s_{n - 1}s_{n - 2}\cdots s_1$, and the group $\overline{\Omega}$ is cyclic, generated by $\bo$.   The element $\bo$ has order $n$, and for $1\leq i \leq n - 1$, the powers of $\bo$ have the following reduced decomposition:
$$\bo^{- i} = \underbrace{\underbrace{(s_{i}s_{i - 1}\cdots s_{1})}_{i~\textnormal{terms}}\underbrace{(s_{i + 1}s_{i}\cdots s_{2})}_{i~\textnormal{terms}}\cdots \underbrace{(s_{n - 1}s_{n - 2}\cdots s_{n - i})}_{i ~\textnormal{terms}}}_{n - i~ \textnormal{sequences of}~i~\textnormal{terms each}}.$$
In particular, $\boldl(\bo^{ - i}) = i(n - i)$.

Throughout this section, we fix $\beta\in \Phi^+$ such that $\textnormal{ht}(\beta)\geq 2$, and $0\leq r \leq f - 1$.  For brevity and typographical ease, we write $\bar{\eta}^w$ for $\bar{\eta}^w_{w^{-1}(\beta),r}$.

\subsection{Combinatorial preliminaries}

We begin with several combinatorial lemmas.

\begin{lemma}
\label{comb0}
We have $\bo^{-i}\in {}^{(i)}W_0$. More precisely, we have $\bo^{-i} = w_{(i),\circ}w_\circ$, where $w_{(i),\circ}$ denotes the longest element of $W_{(i),0}$.  
\end{lemma}

\begin{proof}
We have
$$\bo^{i}(1) = n - i + 1 < \bo^{i}(2) = n - i + 2 < \cdots < \bo^{i}(i) = n$$
$$\textnormal{and}$$
$$\bo^{i}(i + 1) = 1 < \bo^{i}(i + 2) = 2 < \cdots < \bo^{i}(n) = n - i,$$
so $\bo^{i} \in W^{(i)}_0$ by \cite[Lemma 2.4.7]{bjornerbrenti}, and therefore $\bo^{-i}\in (W^{(i)}_0)^{-1} = {}^{(i)}W_0$.  Since $W_{(i),0}\cong \fS_{i}\times \fS_{n - i}$, we obtain
\begin{eqnarray*}
\boldl(w_{(i),\circ}\bo^{-i}) & = & \boldl(w_{(i),\circ}) + \boldl(\bo^{-i})\\
 & = & \frac{i(i - 1)}{2} + \frac{(n - i)(n - i - 1)}{2} + i(n - i)\\
 & = & \frac{n(n - 1)}{2}\\
 & = & \boldl(w_\circ),
\end{eqnarray*}
and therefore we must have $w_{(i),\circ}\bo^{-i} = w_\circ$ by uniqueness of the longest element.  
\end{proof}

We define integers $a_1, a_2, \ldots, a_{i(n - i)}\in \{1,\ldots, n - 1\}$ by
$$s_{a_1}s_{a_2}\cdots s_{a_{i(n - i)}} := (s_{i}s_{i - 1}\cdots s_{1})(s_{i + 1}s_{i}\cdots s_{2})\cdots (s_{n - 1}s_{n - 2}\cdots s_{n - i}) = \bo^{-i},$$
and consider the set
$$\left\{ \alpha_{a_1},~s_{a_1}(\alpha_{a_2}),~s_{a_1}s_{a_2}(\alpha_{a_3}),~\ldots,~ s_{a_1}\cdots s_{a_{i(n - i) - 1}}(\alpha_{a_{i(n - i)}})\right\}.$$
By \cite[Section 1.7]{humphreys:book}, this set is exactly $\{\alpha\in \Phi^+ : \bo^{i}(\alpha)\in \Phi^-\}$, which is easily checked to be equal to
$$\Phi^+\smallsetminus\Phi^+_{(i)} = \{\alpha_{j,k}: 1\leq j \leq i~\textnormal{and}~i + 1 \leq k \leq n\}.$$

\begin{lemma}
\label{comb1}
Let $1\leq i \leq n - 1$, and let $w\in W_0$.  Consider the following condition on $w$:
\begin{equation}
\label{bruhat}
w < ws_{a_1} < ws_{a_1}s_{a_2} < \cdots < ws_{a_1}s_{a_2}\cdots s_{a_{i(n - i)}} \tag{$\gemini$}
\end{equation}
\begin{enumerate}[(a)]
\item The element $w$ satisfies \eqref{bruhat} if and only if $w\in W_{(i),0}$.  \label{comb1a}
\item The element $w$ violates \eqref{bruhat} at exactly one place if and only if $w\in s_{i}W_{(i),0}$.  In this case, there is exactly one $1\leq j\leq i(n - i) $ such that $ws_{a_1}\cdots s_{a_{j - 1}} > ws_{a_1}\cdots s_{a_{j - 1}}s_{a_j}$, and we have
$$ws_{a_1}\cdots s_{a_{j - 1}}(\alpha_{a_{j}}) = -\alpha_{i}.$$\label{comb1b}
\end{enumerate}
\end{lemma}

\begin{proof}
\begin{enumerate}[(a)]
\item By equation \eqref{bruhateqs}, $w$ satisfies \eqref{bruhat} if and only if
$$w(\alpha_{a_1})  \in \Phi^+,\quad ws_{a_1}(\alpha_{a_2})  \in  \Phi^+,\quad \ldots,\quad ws_{a_1}\cdots s_{a_{i(n - i) - 1}}(\alpha_{a_{i(n - i)}}) \in  \Phi^+,$$
if and only if 
$$w(\Phi^+\smallsetminus\Phi^+_{(i)}) \subset \Phi^+$$  
(by the remarks preceding the lemma).  It is well-known that $W_{(i),0}$ preserves $\Phi^+\smallsetminus\Phi^+_{(i)}$.  Therefore it suffices to prove that if $w\in W_0$ satisfies $w(\Phi^+\smallsetminus\Phi_{(i)}^+)\subset \Phi^+$, then $w\in W_{(i),0}$.  Let us write $w = vu$, with $u\in W_{(i),0}, v\in W_0^{(i)}$; then by the previous comment the condition $w(\Phi^+\smallsetminus\Phi^+_{(i)})\subset \Phi^+$ implies $v(\Phi^+\smallsetminus\Phi^+_{(i)})\subset \Phi^+$.  However, equation \eqref{bruhateqs} and the definition of $W_0^{(i)}$ imply that $v(\Phi^+_{(i)})\subset \Phi^+$, so that $v(\Phi^+)\subset \Phi^+$.  Consequently, $\boldl(v) = 0$, and the claim follows.

\item  Similarly to part (a), $w$ violates \eqref{bruhat} at exactly one place if and only if there exists exactly one element $\alpha\in \Phi^+\smallsetminus\Phi_{(i)}^+$ for which $w(\alpha)\in \Phi^-$.  Suppose first that $w = s_{i}u$ for $u\in W_{(i),0}$.  Since $W_{(i),0}$ preserves $\Phi^+\smallsetminus\Phi_{(i)}^+$, we have $w(\Phi^+\smallsetminus\Phi_{(i)}^+) = s_{i}(\Phi^+\smallsetminus\Phi_{(i)}^+)$, and there is exactly one element of this set which is negative, namely $-\alpha_{i}$.  Conversely, suppose $w\in W_0$ is such that exactly one element of $w(\Phi^+\smallsetminus\Phi_{(i)}^+)$ is negative.  Write $w = vu$, with $u\in W_{(i),0}, v\in W_0^{(i)}$.  Again, since $w(\Phi^+\smallsetminus\Phi_{(i)}^+) = v(\Phi^+\smallsetminus\Phi_{(i)}^+)$ and $v(\Phi^+_{(i)})\subset \Phi^+$, we conclude that $v(\Phi^+)$ contains exactly one negative element.  Consequently, $\boldl(v) = 1$, and the only simple reflection not contained in $W_{(i),0}$ is $s_{i}$.  
\end{enumerate}
\end{proof}

We now describe the set $\cW_\beta'$.  We let $W_{\beta,0}$ denote the Weyl group of $\bM_\beta$.  Thus, if $\beta = \alpha_{m,m + a}$, then $\hgt(\beta) = a$, we have $\bM_\beta \cong \bG_{\textnormal{m}}^{m - 1}\times \bG\bL_{a + 1}\times \bG_{\textnormal{m}}^{n - m - a}$, and $W_{\beta,0}\cong \fS_{a + 1}$ is generated by $\{s_m, \ldots, s_{m + a - 1}\}$.  We let $\Omega_{\bM_\beta}$ denote the group of length 0 elements in the extended affine Weyl group of $M_{\beta}$, and let $\overline{\Omega_{\bM_\beta}}$ denote its image in $W_{\beta,0}\subset W_0$, which is a cyclic group of order $a + 1$.  Finally, we let ${}^\beta W_0$ denote the set of minimal coset representatives of $W_{\beta,0}\backslash W_0$.

\begin{propn}\label{Wbeta'}
Let $\beta = \alpha_{m,m+a}\in \Phi^+$.  We then have
$$\cW_\beta' = w_{\beta,\circ}\textnormal{stab}_{W_{\beta,0}}(\beta)\cdot \overline{\Omega_{\bM_\beta}}\cdot {}^\beta W_0,$$
where $w_{\beta,\circ}$ denotes the longest element of $W_{\beta,0}$.  In particular, we have $|\cW_\beta'| = n!/a = n!/\hgt(\beta)$.  
\end{propn}

\begin{proof}
Recall that the set $\cW_\beta'$ is defined as the set of $w\in W_0$ for which $I_1\cap U_{w^{-1}(\beta)}\not\subset[I_1\cap w^{-1}Bw, I_1\cap w^{-1}Bw]$.  
This condition holds if and only if, for every $m < j < m + a$, we have 
\begin{equation}\label{commforW'}
I_1\cap U_{w^{-1}(\beta)} = I_1\cap U_{w^{-1}(\alpha_{m,m+a})}\supsetneq[I_1\cap U_{w^{-1}(\alpha_{m,j})},I_1\cap U_{w^{-1}(\alpha_{j,m+a})}],
\end{equation}
which holds if and only if
\begin{equation}\label{relnforW'}
1 + \mathbf{1}_{\Phi^-}\left(w^{-1}(\alpha_{m,m+a})\right) = \mathbf{1}_{\Phi^-}\left(w^{-1}(\alpha_{m,j})\right) + \mathbf{1}_{\Phi^-}\left(w^{-1}(\alpha_{j,m+a})\right)
\end{equation}
for every $m < j < m + a$.

We first show that $w_{\beta,\circ}\textnormal{stab}_{W_{\beta,0}}(\beta)\cdot \overline{\Omega_{\bM_\beta}}\cdot {}^\beta W_0$ is contained in $\cW_\beta'$.  If $w\in w_{\beta,\circ}\textnormal{stab}_{W_{\beta,0}}(\beta)$, then 
$$w^{-1}(\alpha_{m,m + a}) = \alpha_{m + a, m}\in \Phi^-,$$ 
and 
$$w^{-1}(\alpha_{m,j}) = \alpha_{m + a,w^{-1}(j)}\in \Phi^-,~ w^{-1}(\alpha_{j,m + a}) = \alpha_{w^{-1}(j),m}\in \Phi^-$$ 
for all $m < j < m + a$ as well (since $\textnormal{stab}_{W_{\beta,0}}(\beta)$ preserves $\{m + 1, \ldots, m + a - 1\}$).  Therefore, both sides of equation \eqref{relnforW'} are equal to 2, so $w_{\beta,\circ}\textnormal{stab}_{W_{\beta,0}}(\beta)\subset \cW_\beta'$.  Next, if $w\in W_{\beta,0}\cap \cW_\beta'$, then the subgroups appearing in equation \eqref{commforW'} are contained in $I_1\cap M_{\beta}$, which is preserved by conjugation by $\Omega_{\bM_{\beta}}$.  Hence, if $w\in W_{\beta,0}\cap \cW_\beta'$, then $w\bo'\in W_{\beta,0}\cap \cW_\beta'$ for every $\bo'\in \overline{\Omega_{\bM_\beta}}$, which shows $w_{\beta,\circ}\textnormal{stab}_{W_{\beta,0}}(\beta)\cdot \overline{\Omega_{\bM_\beta}}\subset \cW_\beta'$.  Finally, if $v\in {}^\beta W_0$ then $v^{-1}(\Phi^+_{\bM_\beta})\subset \Phi^+$ by equation \eqref{bruhateqs} and the definition of ${}^\beta W_0$.  This implies $\mathbf{1}_{\Phi^-}(v^{-1}(\alpha)) = \mathbf{1}_{\Phi^-}(\alpha)$ for $\alpha\in \Phi_{\bM_\beta}$.  Combining these facts with equation \eqref{relnforW'} gives the claim.

Next, we show that $w_{\beta,\circ}\textnormal{stab}_{W_{\beta,0}}(\beta)\cdot \overline{\Omega_{\bM_\beta}}\cdot {}^\beta W_0$ has order $n!/a$.  Since $w_{\beta,\circ}\textnormal{stab}_{W_{\beta,0}}(\beta)\cdot \overline{\Omega_{\bM_\beta}} \subset W_{\beta,0}$, the product map
$$w_{\beta,\circ}\textnormal{stab}_{W_{\beta,0}}(\beta)\cdot \overline{\Omega_{\bM_\beta}} \times {}^\beta W_0 \longrightarrow w_{\beta,\circ}\textnormal{stab}_{W_{\beta,0}}(\beta)\cdot \overline{\Omega_{\bM_\beta}}\cdot {}^\beta W_0$$
is a bijection.  Furthermore, since $\overline{\Omega_{\bM_\beta}}$ fixes no elements of $\Phi_{\bM_\beta}$, we have $\textnormal{stab}_{W_{\beta,0}}(\beta)\cap \overline{\Omega_{\bM_\beta}} = \{1\}$.  This implies 
\begin{eqnarray*}
\left|w_{\beta,\circ}\textnormal{stab}_{W_{\beta,0}}(\beta)\cdot \overline{\Omega_{\bM_\beta}}\cdot {}^\beta W_0\right| & = & \left|\textnormal{stab}_{W_{\beta,0}}(\beta)\right|\cdot \left|\overline{\Omega_{\bM_\beta}}\right|\cdot \left|{}^\beta W_0\right|\\
 & = & (a - 1)!(a + 1)\frac{n!}{(a + 1)!} \\
 & = & \frac{n!}{a}.
\end{eqnarray*}

Finally, we compute the order of $\cW_\beta'$.  Let $w\in \cW_\beta'$; since $\cW_\beta'$ is stable by right multiplication by the cyclic group $\overline{\Omega}$, we may multiply $w$ by an element of $\overline{\Omega}$ and suppose $w^{-1}(m) = n$.  This implies $w^{-1}(\alpha_{m,m + a}) = \alpha_{n,w^{-1}(m + a)}\in \Phi^-$.  Therefore in order for equation \eqref{relnforW'} to be satisfied, we must have
$$n > w^{-1}(j) > w^{-1}(m + a)$$
for every $m < j < m + a$.  For $0\leq k \leq n - a - 1$, the number of $w\in W_0$ with $w^{-1}(m) = n$, $w^{-1}(m + a) = n - a - k$, and $n > w^{-1}(j) > w^{-1}(m + a)$ for $m < j < m + a$ is equal to $\binom{a - 1 + k}{a - 1}(a-1)!(n - a - 1)!$.  Hence, we obtain
\begin{eqnarray*}
\left|\cW_\beta'\right| & = & \left|\overline{\Omega}\right|\sum_{k = 0}^{n - a - 1}\binom{a - 1 + k}{a - 1}(a-1)!(n - a - 1)!\\
 & = & n(a - 1)!(n - a - 1)!\sum_{k = 0}^{n - a - 1}\binom{a - 1 + k}{a - 1}\\
 & = & n(a - 1)!(n - a - 1)!\binom{n - 1}{a}\\
 & = & \frac{n!}{a},
\end{eqnarray*}
and the proposition follows.  
\end{proof}

\subsection{Right adjoint calculation -- Maximal Levis}

Our next goal will be to calculate the right adjoints of parabolic induction on $\gr_\bullet$.  In order to do this, we need to compute the action of certain diagonal elements.  For brevity, we write $\T_j$ for the operator $\T_{\widehat{s_j}}$.  Recall that, by Lemmas \ref{TsalphaGr} and \ref{TomegaGr}, we have
\begin{eqnarray*}
\bar{\eta}^w\cdot\T_{j} & = & \delta\mathbf{1}_{\Phi^-}(w(\alpha_j))\cdot\bar{\eta}^w + \zeta\left(\mathbf{1}_{\Phi^+\smallsetminus\{\beta\}}(w(\alpha_j)) + \varepsilon \delta_{w(\alpha_j),-\beta}\delta_{F,\qp}\right)\cdot \bar{\eta}^{ws_j}\\
\bar{\eta}^w\cdot\T_{\omega^{\pm i}} & = & \zeta'\cdot \bar{\eta}^{w\bo^{\pm i}}\\
\bar{\eta}^w\cdot\T_{t_0} & = & \zeta''\cdot\bar{\eta}^{w}
\end{eqnarray*}
for some $\delta\in C$ and $\varepsilon, \zeta,\zeta',\zeta''\in C^\times$.  We will often write such formulas as
$$\bar{\eta}^w\cdot\T_{j}  =  \delta\cdot\bar{\eta}^w + \delta'\cdot \bar{\eta}^{ws_j},$$
for $\delta,\delta'\in C$, with the understanding that $\delta' = 0$ if $ws_j\not\in \cW_{\beta}'$.

Now define
$$\fm_{\beta,r} := \textnormal{span}\{\bar{\eta}^{w}\}_{w\in \cW_\beta'},$$
so that $\gr_a = \bigoplus_{\hgt(\beta) = a}\bigoplus_{r = 0}^{f - 1}\fm_{\beta,r}$.  By Proposition \ref{Wbeta'}, this is a $C$-vector space of dimension $n!/\hgt(\beta)$.  By the equations above, $\fm_{\beta,r}$ is stable by the action of $\cH$.

Fix $1\leq i \leq n - 1$.  We will first compute the right adjoints on $\fm_{\beta,r}$ for the maximal Levi subgroups $\bM_{(i)}$.  Set 
$$\lambda_i := \textnormal{diag}(\underbrace{1,\ldots, 1}_{i~\textnormal{times}}, \underbrace{\varpi, \ldots, \varpi}_{n - i~\textnormal{times}})\in T.$$
We have a reduced decomposition in $\tW$
\begin{eqnarray*}
\lambda_{i}^{-1} & = &  \widehat{\bo^{-i}}\omega^{-(n - i)}t_0\\
 & = &  (\widehat{s_{i}}\widehat{s_{i - 1}}\cdots \widehat{s_{1}})(\widehat{s_{i + 1}}\widehat{s_{i}}\cdots \widehat{s_{2}})\cdots (\widehat{s_{n - 1}}\widehat{s_{n - 2}}\cdots \widehat{s_{n - i}})\omega^{-(n - i)}t_0
\end{eqnarray*}
for some $t_0\in T_0$.  This gives
$$\T_{\lambda_{i}^{-1}}  =  \T_{i}\cdots \T_{1}\T_{i + 1}\cdots \T_{2}\cdots \T_{n - 1}\cdots \T_{n - i}\T_{\omega^{-(n - i)}}\T_{t_0}.$$
In order to compute $R_{\cH_{M_{(i)}}}^{\cH}(\fm_{\beta,r})$, we must choose an element $\lambda^+$ which is central in $\tW_{\bM_{(i)}}$ and which satisfies $\langle\alpha,\nu(\lambda^+)\rangle < 0$ for every $\alpha\in \Phi^+\smallsetminus\Phi^+_{\bM_{(i)}}$.  We take $\lambda^+ = \lambda_{i}^{-1}$, and obtain an isomorphism of vector spaces
$$R_{\cH_{M_{(i)}}}^{\cH}(\fm_{\beta,r})\cong \bigcap_{N\geq 0}\fm_{\beta,r}\cdot \T_{\lambda_{i}^{-1}}^N.$$

\begin{lemma}
\label{Tlambda1}
Let $w\in \cW_\beta'\cap W_{(i),0}$.  We then have
$$\bar{\eta}^w\cdot\T_{\lambda_{i}^{-1}} = \zeta\mathbf{1}_{\Phi^+_{(i)}}(\beta)\cdot\bar{\eta}^{w}$$
for some $\zeta\in C^\times$.  In particular, the vector space
$$\textnormal{span}\{\bar{\eta}^u\}_{u\in \cW_\beta'\cap W_{(i),0}}$$
is stable by the action of $\T_{\lambda_{i}^{-1}}$.  
\end{lemma}

Recall that we have defined integers $a_1, a_2, \ldots, a_{i(n - i)}$ by 
$$\bo^{-i} = s_{a_1}s_{a_2}\cdots s_{a_{i(n - i)}} := (s_{i}s_{i - 1}\cdots s_{1})(s_{i + 1}s_{i}\cdots s_{2})\cdots (s_{n - 1}s_{n - 2}\cdots s_{n - i}).$$

\begin{proof}
If $w\in \cW_\beta'\cap W_{(i),0}$, then $w$ satisfies \eqref{bruhat}, so that $ws_{a_1}\cdots s_{a_{j - 1}}(\alpha_{a_j})\in \Phi^+$ for all $1\leq j \leq i(n - i)$.  By Lemma \ref{TsalphaGr}, this gives
\begin{eqnarray*}
 \bar{\eta}^w\cdot\T_{\lambda_{i}^{-1}} & = & \bar{\eta}^w\cdot \T_{a_1}\cdots \T_{a_{i(n - i)}}\T_{\omega^{-(n - i)}}\T_{t_0}\\
 & = & \zeta\mathbf{1}_{\Phi^+\smallsetminus\{\beta\}}\left(w(\alpha_{a_1})\right)\cdot\bar{\eta}^{ws_{a_1}}\cdot\T_{a_2}\cdots \T_{a_{i(n - i)}}\T_{\omega^{-(n - i)}}\T_{t_0}\\
 & \vdots &  \\
 & = & \zeta'\prod_{j = 1}^{i(n - i)}\mathbf{1}_{\Phi^+\smallsetminus\{\beta\}}\left(ws_{a_1}\cdots s_{a_{j - 1}}(\alpha_{a_j})\right)\cdot\bar{\eta}^{ws_{a_1}\cdots s_{a_{i(n - i)}}}\cdot \T_{\omega^{-(n - i)}}\T_{t_0}\\
 & = & \zeta''\prod_{j = 1}^{i(n - i)}\mathbf{1}_{\Phi^+\smallsetminus\{\beta\}}\left(ws_{a_1}\cdots s_{a_{j - 1}}(\alpha_{a_j})\right)\cdot\bar{\eta}^{w} \\
 &\stackrel{\star}{=} & \zeta'' \prod_{\alpha\in \Phi^+\smallsetminus\Phi^+_{(i)}}\mathbf{1}_{\Phi^+\smallsetminus\{\beta\}}\left(\alpha\right)\cdot\bar{\eta}^w~ =~   \zeta''\mathbf{1}_{\Phi^+_{(i)}}(\beta)\cdot \bar{\eta}^w,
\end{eqnarray*}
where $\zeta,\zeta', \zeta''\in C^\times$, and where the equality ($\star$) follows from the discussion preceding Lemma \ref{comb1}.  This gives the lemma.
\end{proof}

\begin{lemma}
\label{Tlambda}
Let $w\in \cW_\beta'$.  Then there exists an integer $N_w \geq 0$ such that if $N \geq N_w$, we have
$$\bar{\eta}^w\cdot\T_{\lambda_{i}^{-1}}^N \in \textnormal{span}\{\bar{\eta}^u\}_{u\in \cW_\beta'\cap W_{(i),0}}.$$
\end{lemma}

\begin{proof}
Let $w'\in W_0$, and consider the element $\bar{\eta}^{w's_{a_1}\cdots s_{a_{j - 1}}}$.  Suppose 
$$w's_{a_1}\cdots s_{a_{j - 1}}s_{a_j} < w's_{a_1}\cdots s_{a_{j - 1}};$$ 
then, by the exchange condition on $W_0$ (\cite[Section 1.7]{humphreys:book}), we can write $w's_{a_1}\cdots s_{a_{j - 1}} = w''s_{a_1}\cdots s_{a_{j - 1}}s_{a_j}$ for some $w'' < w'$.  This gives
\begin{eqnarray*}
\bar{\eta}^{w's_{a_1}\cdots s_{a_{j - 1}}}\cdot \T_{a_j} & = & \begin{cases} \delta\cdot\bar{\eta}^{w's_{a_1}\cdots s_{a_{j - 1}}s_{a_j}} & (\textnormal{I}) \\ \delta'\cdot\bar{\eta}^{w's_{a_1}\cdots s_{a_{j - 1}}} + \delta''\cdot\bar{\eta}^{w's_{a_1}\cdots s_{a_{j - 1}}s_{a_j}} & (\textnormal{II}) \end{cases} \\
 & = & \begin{cases} \delta\cdot\bar{\eta}^{w's_{a_1}\cdots s_{a_{j - 1}}s_{a_j}} &  \\ \delta'\cdot\bar{\eta}^{w''s_{a_1}\cdots s_{a_{j - 1}}s_{a_j}} + \delta''\cdot\bar{\eta}^{w's_{a_1}\cdots s_{a_{j - 1}}s_{a_j}} & \end{cases} 
\end{eqnarray*}
where $w's_{a_1}\cdots s_{a_{j - 1}}(\alpha_{a_j}) \in \Phi^+$ in case $(\textnormal{I})$, $w's_{a_1}\cdots s_{a_{j - 1}}(\alpha_{a_{j}}) \in \Phi^-$ in case $(\textnormal{II})$, and where $\delta,\delta', \delta''\in C$.  Hence, we see that $\bar{\eta}^{w's_{a_1}\cdots s_{a_{j - 1}}}\cdot \T_{a_j}$ can be written as a linear combination of elements $\bar{\eta}^{vs_{a_1}\cdots s_{a_{j - 1}}s_{a_j}}$ with $v \leq w'$.

Suppose first that $w\in s_{i}W_{(i),0}$, and let $1\leq j \leq i(n - i)$ denote the unique index where $w$ violates \eqref{bruhat}.  Recall that Lemma \ref{comb1}\eqref{comb1b} implies that $ws_{a_1}\cdots s_{a_{j - 1}}(\alpha_{a_j}) = -\alpha_{i}\in \Phi^-$.  We then have
\begin{eqnarray*}
\bar{\eta}^w\cdot\T_{a_1}\cdots\T_{a_j} & = & \delta\cdot \bar{\eta}^{ws_{a_1}\cdots s_{a_{j - 1}}}\cdot \T_{a_j}\\
 & = & \delta'\cdot \bar{\eta}^{ws_{a_1}\cdots s_{a_{j - 1}}} + \delta''\delta_{ws_{a_1}\cdots s_{a_{j - 1}}(\alpha_{a_j}),-\beta}\cdot\bar{\eta}^{ws_{a_1}\cdots s_{a_{j - 1}}s_{a_j}}\\
 & = & \delta'\cdot \bar{\eta}^{w's_{a_1}\cdots s_{a_{j - 1}}s_{a_j}},
\end{eqnarray*}
where $\delta,\delta',\delta''\in C$, where $w' < w$ (as in the first paragraph), and where the last equality follows from the fact that $\textnormal{ht}(\beta) \geq 2$.  Using the first paragraph, we obtain
\begin{eqnarray}
\bar{\eta}^w\cdot\T_{\lambda_{i}^{-1}} & = & \bar{\eta}^w\cdot \T_{a_1}\cdots \T_{a_{i(n - i)}}\T_{\omega^{-(n - i)}}\T_{t_0} \notag\\
 & = & \delta\cdot\bar{\eta}^{w's_{a_1}\cdots s_{a_j}}\cdot\T_{a_{j + 1}}\cdots \T_{a_{i(n - i)}}\T_{\omega^{-(n - i)}}\T_{t_0} \notag\\
 & = & \left(\sum_{\sub{v\in \cW_\beta'}{v < w}}\delta_v\cdot\bar{\eta}^{vs_{a_1}\cdots s_{a_{i(n - i)}}}\right)\cdot\T_{\omega^{-(n - i)}}\T_{t_0} \notag\\
 & = & \sum_{\sub{v\in \cW_\beta'}{v < w}}\delta_v'\cdot\bar{\eta}^{v},\label{Tlambdaaction1}
\end{eqnarray}
 where $\delta,\delta_v,\delta_v'\in C$.

Suppose now that $w\not\in W_{(i),0}\sqcup s_{i}W_{(i),0}$, so that $w$ violates \eqref{bruhat} at least twice.  Let $1\leq j < k \leq i(n - i)$ denote the first two instances where this happens.  We then have
\begin{eqnarray*}
 \bar{\eta}^w\cdot\T_{a_1}\cdots \T_{a_k} & = & \delta\cdot\bar{\eta}^{ws_{a_1}\cdots s_{a_{j - 1}}}\cdot\T_{a_j}\cdots\T_{a_k}\\
 & = & \left(\delta'\cdot\bar{\eta}^{w's_{a_1}\cdots s_{a_{j - 1}}s_{a_j}}\right. \\
 & & \qquad\qquad + \left.\delta''\delta_{ws_{a_1}\cdots s_{a_{j - 1}}(\alpha_{a_j}),-\beta}\cdot\bar{\eta}^{ws_{a_1}\cdots s_{a_{j - 1}}s_{a_j}}\right)\cdot\T_{a_{j + 1}}\cdots\T_{a_k}\\
  & = & \delta'\cdot\bar{\eta}^{w's_{a_1}\cdots s_{a_{j - 1}}s_{a_j}}\cdot\T_{a_{j + 1}}\cdots\T_{a_k}\\
   & & \qquad+ \delta_{ws_{a_1}\cdots s_{a_{j - 1}}(\alpha_{a_j}),-\beta}\Big(\delta'''\cdot\bar{\eta}^{ws_{a_1}\cdots s_{a_{k - 1}}} \\
   & & \qquad\qquad + \delta''''\delta_{ws_{a_1}\cdots s_{a_{k - 1}}(\alpha_{a_k}),-\beta}\cdot\bar{\eta}^{ws_{a_1}\cdots s_{a_{k - 1}}s_{a_k}}\Big)\\
  & = & \delta'\cdot\bar{\eta}^{w's_{a_1}\cdots s_{a_{j - 1}}s_{a_j}}\cdot\T_{a_{j + 1}}\cdots\T_{a_k}\\
  &  &  \qquad + \delta_{ws_{a_1}\cdots s_{a_{j - 1}}(\alpha_{a_j}),-\beta}\delta'''\cdot\bar{\eta}^{w''s_{a_1}\cdots s_{a_{k - 1}}s_{a_{k}}},
\end{eqnarray*}
where $\delta, \delta', \delta'', \delta''', \delta''''\in C$, and where $w', w'' < w$ by applying the exchange condition (as in the first paragraph).  Moreover, the last equality above follows from the fact that $\delta_{ws_{a_1}\cdots s_{a_{j - 1}}(\alpha_{a_j}),-\beta} = 1$ and $\delta_{ws_{a_1}\cdots s_{a_{k - 1}}(\alpha_{a_k}),-\beta} = 1$ cannot hold simultaneously: if this was the case, we would obtain
$$s_{a_1}\cdots s_{a_{j - 1}}(\alpha_{a_j}) = s_{a_1}\cdots s_{a_{k - 1}}(\alpha_{a_k});$$
however, by the discussion preceding Lemma \ref{comb1}, these are distinct elements of $\Phi^+\smallsetminus\Phi^+_{(i)}$, and we obtain a contradiction.  Therefore, again using the first paragraph we obtain
\begin{eqnarray}
\bar{\eta}^w\cdot\T_{\lambda_{i}^{-1}} & = & \bar{\eta}^w\cdot \T_{a_1}\cdots \T_{a_{i(n - i)}}\T_{\omega^{-(n - i)}}\T_{t_0} \notag\\
 & = & \left(\sum_{\sub{v\in \cW_\beta'}{v < w}}\delta_v\cdot\bar{\eta}^{vs_{a_1}\cdots s_{a_{i(n - i)}}}\right)\cdot\T_{\omega^{-(n - i)}}\T_{t_0}\notag\\
 & = & \sum_{\sub{v\in \cW_\beta'}{v < w}}\delta_v'\cdot\bar{\eta}^{v}, \label{Tlambdaaction2}
\end{eqnarray}
where $\delta_v,\delta_v'\in C$.

We now prove the lemma by induction on the length of $w$.  If $w\in \cW_\beta'\cap W_{(i),0}$, then by Lemma \ref{Tlambda1} we may take $N_w = 0$.  On the other hand, if $w \in \cW_\beta'\smallsetminus W_{(i),0}$ has minimal length, equations \eqref{Tlambdaaction1} and \eqref{Tlambdaaction2} show that $\bar{\eta}^w\cdot\T_{\lambda_{i}^{-1}} \in \textnormal{span}\{\bar{\eta}^u\}_{u\in \cW_\beta'\cap W_{(i),0}}$.  We may now proceed by induction on the length of an element in $w \in \cW_\beta'\smallsetminus W_{(i),0}$ and use equations \eqref{Tlambdaaction1} and \eqref{Tlambdaaction2} along with Lemma \ref{Tlambda1} to conclude.  
\end{proof}

\begin{propn}
\label{rightadj}
For $\textnormal{ht}(\beta)\geq 2$ and $0\leq r \leq f - 1$, we have an isomorphism of vector spaces
$$R_{\cH_{M_{(i)}}}^{\cH}(\fm_{\beta,r})  \cong \begin{cases}\textnormal{span}\{\bar{\eta}^{u}\}_{u\in \cW_\beta' \cap W_{(i),0}} &  \textnormal{if}~\beta\in \Phi^+_{(i)},\\ 0 & \textnormal{if}~\beta\not\in \Phi^+_{(i)}. \end{cases}$$
In particular, $R_{\cH_{M_{(i)}}}^{\cH}(\fm_{\beta,r})\neq 0$ if and only if $\beta\in \Phi^+_{(i)}$.  
\end{propn}

\begin{proof}
Suppose first that $\beta\in \Phi^+_{(i)}$.  If $w\in \cW_\beta'\cap W_{(i),0}$, Lemma \ref{Tlambda1} gives
$$\bar{\eta}^w\cdot \T_{\lambda_{i}^{-1}} = \zeta\cdot \bar{\eta}^w$$
for some $\zeta\in C^\times$.  Combining this with Lemma \ref{Tlambda}, we obtain
$$\textnormal{span}\{\bar{\eta}^{u}\}_{u\in \cW_\beta' \cap W_{(i),0}} \subset \bigcap_{N\geq 0}\fm_{\beta,r}\cdot\T_{\lambda_{i}^{-1}}^N \subset \textnormal{span}\{\bar{\eta}^{u}\}_{u\in \cW_\beta' \cap W_{(i),0}}.$$

Assume now that $\beta\not\in \Phi^+_{(i)}$.  If $w\in \cW_\beta'$, by Lemma \ref{Tlambda} there exists $N\geq 0$ such that 
$$\bar{\eta}^w\cdot \T_{\lambda_{i}^{-1}}^N = \sum_{u\in \cW_\beta' \cap W_{(i),0}}\delta_u\cdot\bar{\eta}^u\in \textnormal{span}\{\bar{\eta}^{u}\}_{u\in \cW_\beta' \cap W_{(i),0}},$$
and using Lemma \ref{Tlambda1} gives
$$\bar{\eta}^w\cdot \T_{\lambda_{i}^{-1}}^{N + 1} = \sum_{u\in \cW_\beta' \cap W_{(i),0}}\delta_u\cdot\bar{\eta}^u\cdot\T_{\lambda_{i}^{-1}} = 0.$$
This implies
$$\bigcap_{N\geq 0}\fm_{\beta,r}\cdot\T_{\lambda_{i}^{-1}}^N = 0.$$

To prove the final claim, it suffices to show that $\cW_\beta'\cap W_{(i),0}\neq \emptyset$ when $\beta\in \Phi^+_{(i)}$.  This follows from Lemma \ref{comb0}, since we have $w_\circ\bo^{i} = w_{(i), \circ}\in \cW_\beta'\cap W_{(i),0}$.  
\end{proof}

\subsection{Right adjoint calculation -- $\bM_\beta$}

The next goal will be to calculate $R_{\cH_{M_\beta}}^{\cH}(\fm_{\beta,r})$.  Let us write $\beta = \alpha_{m,m+a}$.  We have $\bM_\beta = \bigcap_{\sub{i \leq m - 1}{\textnormal{or}~i \geq m + a}}\bM_{(i)}$; therefore, by successively applying the (commuting) elements $\T_{\lambda_{i}^{-1}}$ for $i$ as in the intersection defining $\bM_\beta$, Proposition \ref{rightadj} above shows that the vector space $R_{\cH_{M_\beta}}^{\cH}(\fm_{\beta,r})$ has basis $\{\bar{\eta}^{w}\}_{w\in \bigcap_{i}\cW_{\beta}'\cap W_{(i),0}}$, where $i$ is as in the intersection defining $\bM_\beta$.  By Proposition \ref{Wbeta'}, we have
\begin{eqnarray*}
\bigcap_{i}\cW_{\beta}'\cap W_{(i),0} & = & \cW_\beta' \cap \bigcap_i W_{(i),0}\\
 & = & \cW_\beta'\cap W_{\beta,0}\\
 & = & w_{\beta,\circ}\textnormal{stab}_{W_{\beta,0}}(\beta)\overline{\Omega_{\bM_\beta}}.
\end{eqnarray*}

We define a $C$-vector space filtration on $R_{\cH_{M_\beta}}^{\cH}(\fm_{\beta,r})$ indexed by $\textnormal{stab}_{W_{\beta,0}}(\beta)$ as follows (for generalities on filtrations indexed by posets, see \cite[Section 2.2]{hauseux:parabind}): for $w\in \textnormal{stab}_{W_{\beta,0}}(\beta)$, we set
$$\fil_w' := \textnormal{span}\left\{\bar{\eta}^{w_{\beta,\circ}w'\bo'}\right\}_{w' \leq w, \bo'\in \overline{\Omega_{\bM_\beta}}}.$$

\begin{lemma}\label{Wfilt}
The filtration $\fil'_\bullet$ is $\cH_{M_\beta}$-stable.
\end{lemma}

\begin{proof}
Write $\beta = \alpha_{m,m + a}$, and let $\omega_\beta$ denote the block diagonal matrix
$$\omega_\beta := \textnormal{diag}(\underbrace{\varpi,~ \ldots,~ \varpi}_{m - 1~\textnormal{times}},~ \omega_{a + 1}, \underbrace{1,~\ldots,~ 1}_{n - m - a~\textnormal{times}})\in N_{M_\beta}(T),$$
where $\omega_{a+1}$ denotes the matrix $\omega$ of size $(a + 1)\times(a + 1)$.  This element has length 0 in the extended affine Weyl group of $M_\beta$.  The algebra $\cH_{M_\beta}$ is generated by $\T_{\omega_\beta}^{M_\beta}, \T_{\widehat{s_m}}^{M_\beta}$ and the elements $\T_{t}^{M_\beta}$ for $t\in Z_{M_\beta}T_0$, where $Z_{M_\beta}$ denotes the center of $M_\beta$.  It therefore suffices to show that the filtration is stable by these elements.  

\begin{enumerate}[(a)]
\item Let $w'\in \textnormal{stab}_{W_{\beta,0}}(\beta)$ and $\bar{\omega}'\in \overline{\Omega_{\bM_\beta}}$.  Since $\omega_\beta$ is $M_\beta$-positive, the action of $\T_{\omega_\beta}^{M_\beta}$ on $\bar{\eta}^{w_{\beta,\circ}w'\bo'}\in R_{\cH_{M_\beta}}^{\cH}(\fm_{\beta,r})$ is given by
$$\bar{\eta}^{w_{\beta,\circ}w'\bo'}\cdot \T_{\omega_\beta}^{M_\beta} = \bar{\eta}^{w_{\beta,\circ}w'\bo'}\cdot \T_{\omega_\beta}.$$
The element $\omega_\beta$ has a minimal expression 
$$\omega_\beta = (\widehat{s_{m + a}}\cdots \widehat{s_{n - 1}})(\widehat{s_{m - 1}}\cdots \widehat{s_{n - 2}})\cdots (\widehat{s_{1}}\cdots \widehat{s_{n - m}})\omega^m t_0$$
for some $t_0\in T_0$, and one easily checks that
$$w_{\beta,\circ}w'\bo'\left\{\alpha_{m + a},~s_{m + a}(\alpha_{m + a + 1}),~\ldots~, ~s_{m + a}\cdots s_{n - m + 1}(\alpha_{n - m})\right\}\subset \Phi^+\smallsetminus\{\beta\}.$$
Therefore, using Lemma \ref{TsalphaGr}, we obtain
\begin{eqnarray*}
\bar{\eta}^{w_{\beta,\circ}w'\bo'}\cdot \T_{\omega_\beta} & = &  \bar{\eta}^{w_{\beta,\circ}w'\bo'}\cdot\T_{m + a}\cdots \T_{n - m}\T_{\omega^m}\T_{t_0}\\
& = & \zeta\cdot \bar{\eta}^{w_{\beta,\circ}w'\bo's_{m + a}\cdots s_{n - m}\bo^m}\\
& = & \zeta\cdot \bar{\eta}^{w_{\beta,\circ}w'\bo'\bo_\beta},
\end{eqnarray*}
where $\zeta\in C^\times$.  Therefore, $\fil'_\bullet$ is stable under $\T_{\omega_\beta}^{M_\beta}$.

\item Next, we compute the action of $\T_{\widehat{s_m}}^{M_\beta}$.  Since $\widehat{s_m}$ is $M_\beta$-positive, we have
$$\bar{\eta}^{w_{\beta,\circ}w'\bo'}\cdot \T_{\widehat{s_m}}^{M_\beta} = \bar{\eta}^{w_{\beta,\circ}w'\bo'}\cdot \T_{m}.$$
Notice that $w_{\beta,\circ}w'\bo'(\alpha_m)$ cannot equal $-\beta$; if this were the case, we would have $\bo'(\alpha_m) = \beta.$  However, the left-hand side is contained in $\Pi_{\bM_\beta} \sqcup \{-\beta\}$, and since $\hgt(\beta)\geq 2$ the equality cannot hold.  Therefore, using Lemma \ref{TsalphaGr}, we get
$$\bar{\eta}^{w_{\beta,\circ}w'\bo'}\cdot \T_{m}  =   \delta \cdot\bar{\eta}^{w_{\beta,\circ}w'\bo'} + \zeta\mathbf{1}_{\Phi^+\smallsetminus\{\beta\}}\left(w_{\beta,\circ}w'\bo'(\alpha_m)\right)\cdot\bar{\eta}^{w_{\beta,\circ}w'\bo's_m},$$
where $\delta\in C, \zeta\in C^\times$.  Now, we have 
\begin{eqnarray*}
\mathbf{1}_{\Phi^+\smallsetminus\{\beta\}}\left(w_{\beta,\circ}w'\bo'(\alpha_{m})\right) = 1 & \Longleftrightarrow & w_{\beta,\circ}w'\bo'(\alpha_m)\in \Phi_{\bM_\beta}^+\smallsetminus\{\beta\}\\
& \Longleftrightarrow & w'\bo'(\alpha_m)\in \Phi_{\bM_\beta}^-\smallsetminus\{-\beta\}.
\end{eqnarray*}
The latter condition implies $\bo'(\alpha_m)\in \Pi_{\bM_\beta}\smallsetminus\{\alpha_m,\alpha_{m + a - 1}\}$ and $w's_{\bo'(\alpha_m)} < w'$ with $w's_{\bo'(\alpha_m)}\in \textnormal{stab}_{W_{\beta,0}}(\beta)$.  In this case, we have $w_{\beta,\circ}w'\bo's_{m} = w_{\beta,\circ}w's_{\bo'(\alpha_m)}\bo'$, and therefore $\fil'_\bullet$ is stable by $\T_{\widehat{s_m}}^{M_\beta}$.

\item Finally, we compute the action of $\T_{t}^{M_\beta}$.  One easily checks using Lemma \ref{TomegaGr} that the filtration is stable under the action of $\T_{t}^{M_\beta}$ for $t\in T_0$.  Thus, let $t := \lambda(\varpi)$ for $\lambda\in X_*(\bZ_{\bM_\beta})$, and let $\lambda_{\beta}^+$ denote the element $\prod_{\sub{i \leq m - 1}{\textnormal{or}~i \geq m + a}}\lambda_{i}^{-1}$.  The element $\lambda_{\beta}^+$ is strictly $M_\beta$-positive, and therefore there exists $N\geq 0$ such that $(\lambda_{\beta}^+)^Nt$ is $M_\beta$-positive.  Note that any $M_\beta$-positive element in $Z_{M_\beta}T_0$ can be written as a product of an element of $T_0$, an element of $Z$, and nonnegative integral powers of the elements $\lambda_{i}^{-1}$ for $i\leq m - 1$ and $i\geq m + a$.  Therefore, using the isomorphism
$$R_{\cH_{M_\beta}}^{\cH}(\fm_{\beta,r}) = \textnormal{Hom}_{\cH_{M_\beta}^+}(\cH_{M_\beta},\fm_{\beta,r}) \cong \varprojlim_{v\mapsto v\cdot\T_{\lambda_\beta^+}}\fm_{\beta,r}$$
and applying the calculation from Lemma \ref{Tlambda1} twice along with Lemma \ref{TomegaGr} (note that $t$ has length 0 in the extended affine Weyl group of $M_\beta$), we obtain
$$\bar{\eta}^{w_{\beta,\circ}w'\bo'}\cdot \T_{t}^{M_\beta} = \zeta^{-1}\cdot \bar{\eta}^{w_{\beta,\circ}w'\bo'}\cdot \T_{(\lambda_\beta^+)^Nt}^{M_\beta} = \zeta^{-1}\zeta'\cdot \bar{\eta}^{w_{\beta,\circ}w'\bo'},$$
where $\zeta\in C^\times$ (resp. $\zeta'\in C^\times$) is the action of $\T_{(\lambda_\beta^+)^N}$ (resp. $\T_{(\lambda_\beta^+)^Nt}$) on $\bar{\eta}^{w_{\beta,\circ}w'\bo'}$.  Thus, $\fil'_\bullet$ is stable by $\T_{t}^{M_\beta}$, and the lemma follows.  
\end{enumerate}
\end{proof}

Since $\fil'_\bullet$ is stable by the action of $\cH_{M_\beta}$, we may define the associated graded $\cH_{M_\beta}$-module: for $w\in \textnormal{stab}_{W_{\beta,0}}(\beta)$, set
$$\gr'_w := \fil'_w\Big/ \sum_{w' < w}\fil'_{w'}.$$
For an element $\bar{\eta}^{v}\in \fil'_w$, we denote by $\bar{\bar{\eta}}^{v}$ its image in $\gr'_w$.

\begin{rmk}\label{fil'notsplit}
The filtration $\fil'_\bullet$ does not split in general.  As an example, we may take $G = \textnormal{GL}_4(\qp)$, and $\beta = \alpha_{1,4}$, the highest root.  Then $\textnormal{stab}_{W_{\beta,0}}(\beta) = \langle s_2\rangle$, and $\cW_\beta' = w_\circ \textnormal{stab}_{W_{\beta,0}}(\beta) \overline{\Omega} = w_\circ\langle s_2\rangle\langle s_3s_2s_1\rangle$.  Let $\chi$ denote a character of $T$ such that $w(\chi\overline{\beta}^{-1})|_{T_0} \neq \chi\overline{\beta}^{-1}|_{T_0}$ for every $w\in W_0, w\neq 1$.  This implies that $\chi\overline{\beta}^{-1}\circ \alpha^\vee$ is a nontrivial character of $\fo^\times$ for every $\alpha\in \Phi$ (and therefore the $\delta$-functions in Lemma \ref{TsalphaGr} are all equal to 0).  Consider the short exact sequence of $\cH$-modules
\begin{flalign}\label{gr'notsplit}
0 \longrightarrow \gr_1' = \textnormal{span}\left\{ \bar{\eta}^{w_\circ},~\bar{\eta}^{s_2s_1s_2},~\bar{\eta}^{s_1s_3},~\bar{\eta}^{s_2s_3s_2} \right\} \longrightarrow \fm_{\beta,0} & & 
\end{flalign}
\begin{flushright}
$ \longrightarrow \gr_{s_2}' = \textnormal{span}\left\{ \bar{\bar{\eta}}^{s_1s_2s_3s_2s_1},~\bar{\bar{\eta}}^{s_1s_2},~\bar{\bar{\eta}}^{s_2s_1s_3},~\bar{\bar{\eta}}^{s_3s_2} \right\} \longrightarrow 0.$
 \end{flushright}
Using Lemma \ref{TsalphaGr}, we have 
\begin{center}
\begin{tabular}{rl}
 $\bar{\eta}^{w_\circ}\cdot\T_1 = 0$, & $\hspace{32pt} \bar{\eta}^{s_2s_1s_2} \cdot \T_1 = 0,$ \\
  $\bar{\eta}^{s_1s_3} \cdot \T_1 = 0$, & $\hspace{32pt} \bar{\eta}^{s_2s_3s_2}\cdot \T_1 = 0$,\\
  & \\
 $\bar{\eta}^{s_1s_2s_3s_2s_1} \cdot \T_1 = 0$, & $\qquad\qquad \bar{\eta}^{s_1s_2}\cdot \T_1 = d_{\alpha_1,s_1s_2s_1(\beta)}\bar{\eta}^{s_1s_2s_1},$  \\
  $\bar{\eta}^{s_2s_1s_3} \cdot \T_1 = 0$, & $\qquad\qquad \bar{\eta}^{s_3s_2}\cdot \T_1 = 0.$
\end{tabular}
\end{center}
Now, since $\bar{\bar{\eta}}^{s_1s_2}\in \gr_{s_2}'$ is killed by $\T_1$, its image under any splitting $\gr_{s_2}'\longrightarrow \fm_{\beta,0}$ must also be killed by $\T_1$.  The equations above show that this cannot happen, and therefore \eqref{gr'notsplit} does not split.  
\end{rmk}

We may now describe $R_{\cH_{M_\beta}}^{\cH}(\fm_{\beta,r})$.

\begin{propn}\label{rightadjbeta}
Let $\beta = \alpha_{m,m+a}\in \Phi^+$, with $\hgt(\beta)\geq 2$.  The $\cH_{M_\beta}$-module $R_{\cH_{M_\beta}}^{\cH}(\fm_{\beta,r})$ is supersingular of dimension $|\textnormal{stab}_{W_{\beta,0}}(\beta)|\cdot|\overline{\Omega_{\bM_\beta}}| = (a + 1)!/a$.  More precisely, $R_{\cH_{M_\beta}}^{\cH}(\fm_{\beta,r})$ has an $\cH_{M_\beta}$-stable filtration indexed by $\textnormal{stab}_{W_{\beta,0}}(\beta)$.  The graded piece indexed by $w\in \textnormal{stab}_{W_{\beta,0}}(\beta)$ is an $(a + 1)$-dimensional cyclic supersingular module, and the action of $\T_{t_0}^{M_\beta}, t_0\in T_0,$ on this graded piece is given by the characters
$$\T_{t_0}^{M_\beta}\longmapsto (\chi\overline{\beta}^{-p^r})\left(\widehat{w_{\beta,\circ}w\bo'}t_0\widehat{w_{\beta,\circ}w\bo'}^{-1}\right)^{-1}$$
for $\bo'\in \overline{\Omega_{\bM_\beta}}$.  
\end{propn}

\begin{proof}
The definition of the filtration and its graded pieces is as above, and the precise action of $\T_{t_0}^{M_\beta}$ on $\gr'_w$ follows from Lemma \ref{TomegaGr}.  To verify the claim about supersingularity, we must compute the action of $\cH_{M_\beta,\aff}$, which is generated by $\{\T_{t_0}^{M_\beta}\}_{t_0\in T_0}$ and $\{\T_{\omega_\beta^b\widehat{s_m}\omega_\beta^{-b}}^{M_\beta}\}_{0\leq b \leq a}$.  If $\bo'\in \overline{\Omega_{\bM_\beta}}$, then the proof of Lemma \ref{Wfilt} shows that
\begin{eqnarray*}
\bar{\bar{\eta}}^{w_{\beta,\circ}w\bo'}\cdot\T_{\omega_\beta^b\widehat{s_m}\omega_\beta^{-b}}^{M_\beta} & = & \bar{\bar{\eta}}^{w_{\beta,\circ}w\bo'}\cdot\left(\T_{\omega_\beta}^{M_\beta}\right)^b\T_{\widehat{s_m}}^{M_\beta}\left(\T_{\omega_\beta}^{M_\beta}\right)^{-b} \\
&  = & -\mathbf{1}_{\Phi^-}\left(w_b(\alpha_m)\right)\delta_{(\chi\overline{\beta}^{-p^r})\circ w_b(\alpha_m^\vee),1}\bar{\bar{\eta}}^{w_{\beta,\circ}w\bo'},
\end{eqnarray*}
where $w_b := w_{\beta,\circ}w\bo'\bo_\beta^b$.  Proceeding as in the proof of Proposition \ref{ssingprop} shows that the $\cH_{M_\beta,\aff}$-character spanned by $\bar{\bar{\eta}}^{w_{\beta,\circ}w\bo'}$ is supersingular, and consequently $\gr_w'$ is supersingular.  
\end{proof}

\begin{rmk}\label{gr'irred}
Suppose $\chi$ is sufficiently generic; namely, suppose that the characters 
$$T_0 \ni t_0\longmapsto (\chi\overline{\beta}^{-p^r})\left(\widehat{w_{\beta,\circ}w\bo'}t_0\widehat{w_{\beta,\circ}w\bo'}^{-1}\right)^{-1}$$
for $w\in \textnormal{stab}_{W_{\beta,0}}$ and $\bo'\in \overline{\Omega_{\bM_\beta}}$ are all distinct.  Then the graded modules $\gr'_w$ are all \emph{simple} supersingular $\cH_{M_\beta}$-modules.  Indeed, if $\fv$ is a nonzero $\cH_{M_\beta}$-stable submodule of $\gr_w'$, and $v\in \fv$ is a nonzero vector, then Fourier inversion on the elements $\{\T_{t_0}^{M_\beta}\}_{t_0\in T_0}$ shows that the $T_0$-isotypic components of $v$ are also contained in $\fv$.  By the genericity assumption, the $T_0$-eigenspaces of $\gr_w'$ are all one-dimensional, which implies that $\bar{\bar{\eta}}^{w_{\beta,\circ}w\bo'}\in \fv$ for some $\bo'\in \overline{\Omega_{\bM_\beta}}$.  Since this element generates $\gr_w'$, we obtain $\fv = \gr_w'$ and the claim follows.  
\end{rmk}

With the above results now in place, we describe the $\cH$-module $\fm_{\beta,r}$.

\begin{thm}\label{mbetaind}
For $\hgt(\beta)\geq 2$ and $0\leq r \leq f - 1$, we have an isomorphism of right $\cH$-modules
$$\fm_{\beta,r}\cong \ind_{\cH_{M_\beta}}^{\cH}\left(R_{\cH_{M_\beta}}^{\cH}(\fm_{\beta,r})\right),$$
where $R_{\cH_{M_\beta}}^{\cH}(\fm_{\beta,r})$ is the supersingular $\cH_{M_\beta}$-module described in Proposition \ref{rightadjbeta}.  In particular, if $\alpha_0$ denotes the highest root, then the $\cH$-module $\fm_{\alpha_0,r}$ is supersingular.  
\end{thm}

\begin{proof}
The counit of the adjunction $(\ind_{\cH_{M_\beta}}^{\cH},~R_{\cH_{M_\beta}}^{\cH})$ is given by the $\cH$-equivariant map
\begin{eqnarray*}
\kappa: \ind_{\cH_{M_\beta}}^{\cH}\left(R_{\cH_{M_\beta}}^{\cH}(\fm_{\beta,r})\right) & \longrightarrow &  \fm_{\beta,r}\\
\bar{\eta}\otimes h & \longmapsto & \bar{\eta}\cdot h.
\end{eqnarray*}
By the discussion preceding Lemma \ref{Wfilt} and \cite[Remark 4.1]{olliviervigneras}, $\ind_{\cH_{M_\beta}}^{\cH}(R_{\cH_{M_\beta}}^{\cH}(\fm_{\beta,r}))$ has a basis given by
$$\left\{\bar{\eta}^{w_{\beta,\circ}w\bo'}\otimes \T_{\widehat{w}'}\right\},$$
where $w\in \textnormal{stab}_{W_{\beta,0}}(\beta), \bo'\in \overline{\Omega_{\bM_\beta}}$, and $w'\in {}^\beta W_0$.  On the other hand, by Proposition \ref{Wbeta'}, $\fm_{\beta,r}$ has a basis
$$\left\{\bar{\eta}^{w_{\beta,\circ}w\bo'w'}\right\},$$
where $w\in \textnormal{stab}_{W_{\beta,0}}(\beta), \bo'\in \overline{\Omega_{\bM_\beta}}$, and $w'\in {}^\beta W_0$.  It therefore suffices to show that $\kappa(\bar{\eta}^{w_{\beta,\circ}w\bo'}\otimes \T_{\widehat{w}'}) = \bar{\eta}^{w_{\beta,\circ}w\bo'}\cdot\T_{\widehat{w}'}$ is a nonzero scalar multiple of $\bar{\eta}^{w_{\beta,\circ}w\bo'w'}$.

Suppose $w' = s_{b_1}\cdots s_{b_k}$ is a minimal expression of $w'\in {}^\beta W_0$.  Then \cite[Lemma 2.4.3]{bjornerbrenti} implies that each subexpression $s_{b_1}\cdots s_{b_j}$ for $1 \leq j \leq k$ also lies in ${}^\beta W_0$, and therefore
\begin{eqnarray*}
\boldl(w_{\beta,\circ}w\bo's_{b_1}\cdots s_{b_{j - 1}}s_{b_j}) & = & \boldl(w_{\beta,\circ}w\bo') + \boldl(s_{b_1}\cdots s_{b_{j - 1}}s_{b_j})\\
 & = & \boldl(w_{\beta,\circ}w\bo') + \boldl(s_{b_1}\cdots s_{b_{j - 1}})  + \boldl(s_{b_j})\\
 & = & \boldl(w_{\beta,\circ}w\bo's_{b_1}\cdots s_{b_{j - 1}})  + \boldl(s_{b_j}).
\end{eqnarray*}
Thus, by equation \eqref{bruhateqs}, we have $w_{\beta,\circ}w\bo's_{b_1}\cdots s_{b_{j - 1}}(\alpha_{b_j}) \in \Phi^+$ for every $1\leq j \leq k$.  Next, we claim that $w_{\beta,\circ}w\bo's_{b_1}\cdots s_{b_{j - 1}}(\alpha_{b_j}) \neq \beta$ for every $1\leq j \leq k$.  Indeed, if we had $w_{\beta,\circ}w\bo's_{b_1}\cdots s_{b_{j - 1}}(\alpha_{b_j}) = \beta$, then we would obtain
$$s_{b_1}\cdots s_{b_{j - 1}}(\alpha_{b_j}) = -\bo'^{-1}(\beta).$$
However, by \cite[Section 1.7]{humphreys:book}, the left-hand side is contained in $\{\alpha\in \Phi^+: w^{-1}(\alpha)\in \Phi^-\}$, while the right-hand side is contained in $\Pi_{\bM_\beta}\sqcup\{-\beta\}$.  Since these sets have empty intersection, we obtain a contradiction.

We therefore see that $w_{\beta,\circ}w\bo's_{b_1}\cdots s_{b_{j - 1}}(\alpha_{b_j}) \in \Phi^+\smallsetminus\{\beta\}$ for every $1\leq j \leq k$.  The claim of the first paragraph now follows from Lemma \ref{TsalphaGr}.  
\end{proof}

\begin{cor}\label{mbetaindcor}
For $2\leq a \leq h - 1$, the associated graded module $\gr_a$ has the form
$$\gr_a \cong \bigoplus_{\sub{\beta\in \Phi^+}{\hgt(\beta) = a}}\bigoplus_{r = 0}^{f - 1}\ind_{\cH_{M_\beta}}^{\cH}(\fn_{\beta,r}),$$
where $\fn_{\beta,r}$ is the supersingular $\cH_{M_\beta}$-module described in Proposition \ref{rightadjbeta}.  In particular, the $\cH$-module $\gr_{h - 1}$ is supersingular.  
\end{cor}

\subsection{Height filtration redux}

We now descibe when the filtration $\fil_\bullet$ splits.

\begin{lemma}\label{splitfilt}
The short exact sequence of $\cH$-modules
$$0\longrightarrow \gr_0\longrightarrow \fil_1\longrightarrow \gr_1\longrightarrow 0$$
is nonsplit if and only if $F = \qp$ and $\chi = \chi^{s_\beta}\overline{\beta}$ for some $\beta\in \Pi$.  
\end{lemma}

\begin{proof}
By Propositions \ref{gr0} and \ref{gr1}, the short exact sequence above gives an element of 
$$\Ext^1_{\cH}\left(\gr_1,\gr_0\right) \cong \bigoplus_{\beta\in \Pi}\bigoplus_{r = 0}^{f - 1}\Ext^1_{\cH}\left(\ind_{\cH_{M_{\beta}}}^{\cH}(\fn_{F,\beta,r}),\ind_{\cH_T}^{\cH}(\chi)\right)^{\oplus d},$$
where $d$ is as in Proposition \ref{gr0}, and the module $\fn_{F,\beta,r}$ was defined in Subsection \ref{gr1sect} (recall that $\fn_{F,\beta,r}$ is supersingular if and only if $F\neq \qp$).  Since the left adjoint of parabolic induction is exact, we have
$$\Ext^1_{\cH}\left(\ind_{\cH_{M_{\beta}}}^{\cH}\left(\fn_{F,\beta,r}\right),\ind_{\cH_T}^{\cH}(\chi)\right)  \cong  \displaystyle{\Ext^1_{\cH_T}\left(L_{\cH_T}^{\cH}\left(\ind_{\cH_{M_{\beta}}}^{\cH}(\fn_{F,\beta,r})\right),\chi\right)} \cong  \displaystyle{\Ext^1_{\cH_T}\left(L_{\cH_T}^{\cH_{M_\beta}}(\fn_{F,\beta,r}),\chi\right)},$$
where the last isomorphism follows from \cite[Corollary 5.8]{abe:inductions}.

Suppose that the short exact sequence above is nonsplit, so that $\Ext^1_{\cH}\left(\gr_1,\gr_0\right)\neq 0$.  Then there exists some $\beta\in \Pi$ and $0\leq r \leq f - 1$ such that $\Ext^1_{\cH_T}(L_{\cH_T}^{\cH_{M_\beta}}(\fn_{F,\beta,r}),\chi)\neq 0$, which implies $L_{\cH_T}^{\cH_{M_\beta}}(\fn_{F,\beta,r})\neq 0$.  By the description of the module $\fn_{F,\beta,r}$ we must have $F = \qp$, and thus $L_{\cH_T}^{\cH_{M_\beta}}(\fn_{\qp,\beta,0})\cong \chi^{s_\beta}\overline{\beta}$ by \cite[Corollary 5.8 and Theorem 5.20]{abe:inductions}.  Since $\Ext^1_{\cH_T}(\chi',\chi) \neq 0$ if and only if $\chi' = \chi$, one direction follows.

Suppose now that $F = \qp$ and $\chi = \chi^{s_\beta}\overline{\beta}$ for some $\beta\in \Pi$.  Since the center of $\bG\bL_n$ is connected, this condition implies $\chi\circ\beta^\vee(x) = \overline{p^{-\textnormal{val}(x)}x}$ for all $x\in \qp^\times$, and thus $\fn_{\qp,\beta,0}\cong \ind_{\cH_{T}}^{\cH_{M_\beta}}(\chi^{s_\beta}\overline{\beta})\cong \ind_{\cH_{T}}^{\cH_{M_\beta}}(\chi)$.  Let $\theta:T_1\longrightarrow C$ denote the character $\theta(t) = \overline{p^{-1}(1 - \beta(t^{-1}))}$, and let $\{\theta^w\}_{w\in W_0}\subset \textnormal{Fil}_0$ denote the elements of $\textnormal{H}^1(I_1,\ind_B^G(\chi))$ associated to $\theta$ as in Proposition \ref{gr0}.  Let $\fat{\alpha} = (\alpha,\ell)\in \Pi_\aff$, let $\eta^w$ denote the element $\eta^w_{w^{-1}(\beta),0}$, and $\bar{\eta}^w$ its image in $\gr_1$.  Using Lemma \ref{TsalphaGr} and Corollary \ref{Tsalphatorus}, we have
\begin{eqnarray}
\eta^{w}\cdot \T_{\salpha} & = & -\mathbf{1}_{\Phi^-}(w(\alpha))\delta_{(\chi\overline{\beta}^{-1})\circ w(\alpha^\vee),1}\cdot\eta^w \label{Tsactionfil}\\
 & & + \zeta\cdot\left(d_{\alpha,s_\alpha w^{-1}(\beta)}\mathbf{1}_{\Phi^+\smallsetminus\{\beta\}}(w(\alpha))  - \delta_{w(\alpha),-\beta}\right)\cdot \eta^{ws_\alpha} - \delta_{w(\alpha), -\beta}\cdot \theta^w,\notag
\end{eqnarray}
with $\zeta\in C^\times$ as in Lemma \ref{TsalphaGr}.  Combining this with Lemma \ref{Tomega} and Propositions \ref{gr0} and \ref{gr1} shows that we have a short exact sequence of $\cH$-modules
$$0\longrightarrow \textnormal{span}\left\{\theta^w\right\}_{w\in W_0}\longrightarrow \textnormal{span}\left\{\theta^w,~\eta^w\right\}_{w\in W_0}\longrightarrow \textnormal{span}\left\{\bar{\eta}^w\right\}_{w\in W_0}\longrightarrow 0,$$
which is a self-extension of $\ind_{\cH_{T}}^{\cH}(\chi)$.  We claim this self-extension is nonsplit.  To prove this, it suffices to show that the exact sequence above remains nonsplit after applying the functor $L_{\cH_T}^{\cH}$.

Using the isomorphisms $\textnormal{span}\{\theta^w\}_{w\in W_0}\cong \ind_B^G(\chi)^{I_1}$ of Proposition \ref{gr0} and $\ind_B^G(\chi)^{I_1}\cong \ind_{\cH_T}^{\cH}(\chi)$ of \cite[Proposition 4.4]{olliviervigneras}, we see that $L_{\cH_T}^{\cH}(\textnormal{span}\{\theta^w\}_{w\in W_0}) = C\theta^{w_\circ}\otimes 1$.  Likewise, the isomorphisms $\textnormal{span}\{\bar{\eta}^w\}_{w\in W_0}\cong \ind_{\cH_{M_\beta}}^{\cH}(\fn_{\qp,\beta,0})$ of Proposition \ref{gr1} and $\fn_{\qp,\beta,0}\cong \ind_{\cH_T}^{\cH_{M_\beta}}(\chi)$ show that $L_{\cH_T}^{\cH}(\textnormal{span}\{\bar{\eta}^w\}_{w\in W_0}) = C\bar{\eta}^{s_\beta w_\circ}\otimes 1$.  Therefore, it suffices to show that the exact sequence 
$$0\longrightarrow C\theta^{w_\circ}\otimes 1 \longrightarrow C\theta^{w_\circ}\otimes 1~\oplus~ C\eta^{s_\beta w_\circ}\otimes 1 \longrightarrow C\bar{\eta}^{s_\beta w_\circ}\otimes 1 \longrightarrow 0$$
is a nonsplit extension of $\chi$ by $\chi$ as $\cH_T$-modules.  This holds if and only if there exists $\lambda\in T$ such that
$$(\eta^{s_\beta w_\circ}\otimes 1)\cdot \T_\lambda^T \neq \chi(\lambda)^{-1}\cdot(\eta^{s_\beta w_\circ}\otimes 1).$$
Write $\beta = \alpha_m$, so that $s_\beta = s_m$.  Then the element $\lambda_m^{-1}$ is $T$-positive, the element $\widehat{w_\circ}^{-1}\lambda_m^{-1}\widehat{w_\circ}$ is $T$-negative, and consequently we have
\begin{eqnarray*}
(\eta^{s_\beta w_\circ}\otimes 1)\cdot \T_{\lambda_m^{-1}}^T & = & \eta^{s_\beta w_\circ} \otimes \T_{\widehat{w_\circ}^{-1}\lambda_m^{-1}\widehat{w_\circ}}^{T}\\
& = & \eta^{s_\beta w_\circ}\otimes \T_{\widehat{w_\circ}^{-1}\lambda_m^{-1}\widehat{w_\circ}}^{T,*}\\
& = & \eta^{s_\beta w_\circ}\cdot \T_{\widehat{w_\circ}^{-1}\lambda_m^{-1}\widehat{w_\circ}}^{*}\otimes 1.
\end{eqnarray*}
A minimal expression for $\widehat{w_\circ}^{-1}\lambda_m^{-1}\widehat{w_\circ}$ is given by
$$\widehat{w_\circ}^{-1}\lambda_m^{-1}\widehat{w_\circ} = t_0\omega^{-(n - m)}\widehat{s_m}\cdots \widehat{s_{1}}\widehat{s_{m + 1}}\cdots \widehat{s_2}\cdots \widehat{s_{n + 1}}\cdots \widehat{s_{n - m}}$$
for some $t_0\in T_0$, and one easily checks that the unique element of $\Phi^+\smallsetminus\Phi_{(m)}^+$ which is sent to $\Phi^-$ by the element $s_mw_{(m),\circ}$ is $s_{m}\cdots s_{n - m + 1}(\alpha_{n - m})$ (cf. Lemma \ref{comb1}\eqref{comb1b}).  Therefore, we get
\begin{eqnarray*}
 \eta^{s_\beta w_{\circ}}\cdot \T_{\widehat{w_\circ}^{-1}\lambda_m^{-1}\widehat{w_\circ}}^* & = & \eta^{s_m w_{\circ}}\cdot\T_{t_0}\T_{\omega^{-(n - m)}}\T_{m}^*\cdots \T_1^*\T_{m + 1}^*\cdots \T_2^*\cdots \T_{n - 1}^*\cdots \T_{n - m}^* \\
 & \stackrel{\textnormal{Lem}.~\ref{comb0}}{=} & \zeta\cdot \eta^{s_mw_{(m),\circ}}\cdot \T_{m}^*\cdots \T_1^*\T_{m + 1}^*\cdots \T_2^*\cdots \T_{n - 1}^*\cdots \T_{n - m}^* \\
 & \stackrel{\eqref{Tsactionfil}}{=} & \Bigg(\zeta'\cdot \eta^{s_mw_{(m),\circ}s_m\cdots s_{n - m + 1}} \\
 & &\qquad + \sum_{w' < \bo^{n - m}s_{n - m}}\delta_{w'}\eta^{s_mw_{(m),\circ}w'} + \delta'_{w'}\theta^{s_mw_{(m),\circ}w'}\Bigg)\cdot \T_{n - m}^*\\
 & \stackrel{\eqref{Tsactionfil}}{=} & \zeta'\left( \eta^{s_mw_{(m),\circ}s_m\cdots s_{n - m}} - \theta^{s_mw_{(m),\circ}s_m\cdots s_{n - m + 1}}\right)\\
 & &   + \sum_{w' \neq s_mw_{\circ}}\delta_{w'}''\eta^{w'} + \sum_{w'\neq w_{\circ}}\delta'''_{w'}\theta^{w'},
\end{eqnarray*}
where $\zeta,\zeta'\in C^\times$, $\delta_{w'}, \delta_{w'}', \delta_{w'}'', \delta_{w'}'''\in C$.  Since $s_mw_{(m),\circ}s_m\cdots s_{n - m} = s_mw_\circ = s_\beta w_\circ$ and $s_mw_{(m),\circ}s_m\cdots s_{n - m + 1} = w_\circ$, in the localization we obtain 
$$(\eta^{s_\beta w_\circ}\otimes 1)\cdot \T_{\lambda_m^{-1}}^T = \zeta'\cdot \eta^{s_\beta w_\circ}\otimes 1 - \zeta'\cdot \theta^{w_\circ}\otimes 1\not\in \textnormal{span}\left\{\eta^{s_\beta w_\circ}\otimes 1\right\},$$
which gives the claim.  
\end{proof}

\begin{thm}\label{mainthm}
The filtration $\fil_\bullet$ splits as follows:
\begin{eqnarray*}
\textnormal{H}^1\big(I_1,\ind_B^G(\chi)\big) & \cong & \fil_1 \oplus \bigoplus_{a = 2}^{n - 1}\gr_a\\
 & \cong & \fil_1 \oplus \bigoplus_{\sub{\beta\in \Phi^+}{\hgt(\beta)\geq 2}}\bigoplus_{r = 0}^{f - 1}\ind_{\cH_{M_\beta}}^{\cH}(\fn_{\beta,r}),
\end{eqnarray*}
where $\fn_{\beta,r}$ is the supersingular $\cH_{M_\beta}$-module described in Proposition \ref{rightadjbeta}.  The splitting of $\fil_1$ is governed by Lemma \ref{splitfilt}.  In particular, if $F\neq \qp$, we have
$$\textnormal{H}^1\big(I_1,\ind_B^G(\chi)\big) \cong \ind_{\cH_T}^{\cH}(\chi)^{\oplus n([F:\qp] + \mathbf{1}_F(\zeta_p))} \oplus \bigoplus_{\beta\in \Phi^+}\bigoplus_{r = 0}^{f - 1}\ind_{\cH_{M_\beta}}^{\cH}(\fn_{\beta,r}),$$
with $\fn_{\beta,r}$ supersingular.  
\end{thm}

\begin{proof}
We prove by induction that $\fil_b \cong \fil_1 \oplus \bigoplus_{a = 2}^{b}\gr_a$ as $\cH$-modules, the case $b = 1$ being trivial.  Suppose the claim holds for some integer $b\geq 1$; we then have a short exact sequence of $\cH$-modules
$$0\longrightarrow \fil_b \longrightarrow \fil_{b + 1} \longrightarrow \gr_{b + 1} \longrightarrow 0,$$
which gives an element of 
$$\Ext^1_{\cH}(\gr_{b + 1},\fil_b) \cong \Ext^1_{\cH}(\gr_{b + 1},\fil_1)\oplus\bigoplus_{a = 2}^{b}\Ext^1_{\cH}(\gr_{b + 1},\gr_a).$$

By Corollary \ref{mbetaindcor}, the space $\Ext^1_{\cH}(\gr_{b + 1},\gr_a)$ is a direct sum of Ext-spaces of the form $\Ext^1_{\cH}(\ind_{\cH_{M'}}^{\cH}(\fn'),\ind_{\cH_{M}}^{\cH}(\fn))$, where $M$ (resp. $M'$) is a standard Levi subgroup such that $\Pi_{\bM} = \{\alpha_{m},~\ldots,~\alpha_{m + a - 1}\}$ with $0 \leq a \leq b$ (resp. $\Pi_{\bM'} = \{\alpha_{m'},~\ldots,~\alpha_{m' + b}\}$), and $\fn$ (resp. $\fn'$) is a finite-length supersingular $\cH_M$-module (resp. $\cH_{M'}$-module).  Since the left adjoint of parabolic induction is exact, \cite[Corollary 5.8]{abe:inductions} implies
$$\Ext^1_{\cH}\left(\ind_{\cH_{M'}}^{\cH}(\fn'),\ind_{\cH_{M}}^{\cH}(\fn)\right)  \cong  \Ext^1_{\cH_M}\left(L_{\cH_M}^{\cH}(\ind_{\cH_{M'}}^{\cH}(\fn')),\fn\right) \cong \Ext^1_{\cH_M}\left(\ind_{\cH_{M'\cap M}}^{\cH_{M}}(L_{\cH_{M'\cap M}}^{\cH_{M'}}(\fn')),\fn\right) =  0,$$
where the last line follows form the fact that $L_{\cH_{M'\cap M}}^{\cH_{M'}}(\fn') = 0$ (note that $M'\cap M$ is a proper subgroup of $M'$).  Hence $\Ext^1_{\cH}(\gr_{b + 1},\gr_a) = 0$.

By the long exact sequence of Ext-spaces, in order to prove $\Ext^1_{\cH}(\gr_{b + 1},\fil_1) = 0$, it suffices to show that $\Ext^1_{\cH}(\gr_{b + 1},\gr_0) = \Ext^1_{\cH}(\gr_{b + 1},\gr_1) = 0$.  This follows exactly as above.  Therefore, we obtain $\Ext^1_{\cH}(\gr_{b + 1},\fil_b) = 0$, and the result follows.
\end{proof}

\subsection{Relation to extensions}\label{subsectexts}

To conclude, we discuss an application of the above results to extensions of $G$-representations.  Suppose that $\tau$ is an irreducible, admissible, supersingular representation of $G$, and suppose we have a short exact sequence
\begin{equation}\label{ssingext}
0 \longrightarrow \ind_B^G(\chi) \longrightarrow \pi \longrightarrow \tau \longrightarrow 0.
\end{equation}
Unlike the classical case of complex coefficients, this exact sequence does not necessarily split (cf. \cite{hu:exts}).  We have the following necessary condition for the nonsplitting of the above exact sequence.  

\begin{propn}\label{neccond}
If the exact sequence \eqref{ssingext} is nonsplit, then 
$$\textnormal{Hom}_{\cH}\Big(\tau^{I_1},\bigoplus_{r = 0}^{f - 1} \fm_{\alpha_0,r}\Big) \neq 0,$$
where $\fm_{\alpha_0,r} = \fn_{F,\alpha_0,r}$ if $n = 2$.  
\end{propn}

\begin{proof}
Suppose $\textnormal{Hom}_{\cH}(\tau^{I_1},\bigoplus_{r = 0}^{f - 1} \fm_{\alpha_0,r}) = 0$.  By \cite[Theorem 5.3]{olliviervigneras}, the (finite-dimensional) $\cH$-module $\tau^{I_1}$ is supersingular, and therefore 
$$\textnormal{Hom}_{\cH}\left(\tau^{I_1},\ind_{\cH_M}^{\cH}(\fn)\right) \cong \textnormal{Hom}_{\cH}\left(L_{\cH_M}^{\cH}(\tau^{I_1}),\fn\right) = 0$$
if $\bM$ is a proper Levi subgroup of $\bG$.  Hence, by Theorem \ref{mainthm}, we have
$$\textnormal{Hom}_{\cH}\Big(\tau^{I_1},\textnormal{H}^1\big(I_1,\ind_B^G(\chi)\big)\Big) = 0.$$

Taking $I_1$-invariants of \eqref{ssingext} gives an exact sequence of $\cH$-modules
$$0 \longrightarrow \ind_{\cH_T}^{\cH}(\chi) \longrightarrow \pi^{I_1} \longrightarrow \tau^{I_1} \longrightarrow \textnormal{H}^1\big(I_1,\ind_B^G(\chi)\big).$$
The assumption of the first paragraph guarantees that the last map is 0, so we obtain a short exact sequence
$$0 \longrightarrow \ind_{\cH_T}^{\cH}(\chi) \longrightarrow \pi^{I_1} \longrightarrow \tau^{I_1} \longrightarrow 0.$$
Again using that $\tau^{I_1}$ is supersingular, we get $\Ext_{\cH}^1(\tau^{I_1}, \ind_{\cH_T}^{\cH}(\chi)) = 0$, and therefore the exact sequence above splits.  Let $\sigma:\tau^{I_1}\longrightarrow \pi^{I_1}$ denote an $\cH$-equivariant splitting.  Letting $\bX := \cind_{I_1}^G(C)$ denote the universal module, we thus obtain a short exact sequence of $G$-representations
\begin{center}
\begin{tikzcd}
0 \ar[r] & \ind_{\cH_T}^{\cH}(\chi)\otimes_{\cH}\bX \ar[r] & \pi^{I_1}\otimes_{\cH}\bX \ar[r] & \tau^{I_1}\otimes_{\cH}\bX \ar[r] \ar[l, bend left, "\sigma\otimes 1"] & 0
\end{tikzcd}
\end{center}

Consider now the counit $\kappa_{\pi}:\pi^{I_1}\otimes_{\cH}\bX \longrightarrow \pi$ of the adjunction $(-\otimes_{\cH}\bX,~(-)^{I_1})$.  We have $\kappa_\pi(v\otimes\mathbf{1}_{I_1g^{-1}}) = g.v$, where $v\in \pi^{I_1}$ and $\mathbf{1}_{I_1g^{-1}}\in \bX$ is the characteristic function of $I_1g^{-1}$.  The representation $\ind_B^G(\chi)$ is generated by its space of $I_1$-invariant vectors, and by \cite[Cor. 4.7]{olliviervigneras} the restriction of $\kappa_{\pi}$ to $\ind_{\cH_T}^{\cH}(\chi)\otimes_{\cH}\bX$ is an isomorphism onto $\ind_B^G(\chi)$.  Since $\pi^{I_1}\supsetneq \ind_B^G(\chi)^{I_1}$, we see that $\pi$ is generated by its $I_1$-invariant vectors, and thus $\kappa_\pi$ is surjective.  Finally, since $\tau$ is irreducible, the analogously defined map $\kappa_\tau$ is surjective.  This gives the following commutative diagram with exact rows:
\begin{center}
\begin{tikzcd}
0 \ar[r] & \ind_{\cH_T}^{\cH}(\chi)\otimes_{\cH}\bX \ar[r] \ar[d,"{\rotatebox{90}{$\sim$}}"] & \pi^{I_1}\otimes_{\cH}\bX \ar[r] \ar[d,twoheadrightarrow,"\kappa_\pi"] & \tau^{I_1}\otimes_{\cH}\bX \ar[r] \ar[l, bend left, "\sigma\otimes 1"] \ar[d,twoheadrightarrow,"\kappa_\tau"] & 0\\
0 \ar[r] & \ind_B^G(\chi) \ar[r] & \pi \ar[r] & \tau \ar[r] & 0
\end{tikzcd}
\end{center}
By the snake lemma, we get an isomorphism of $G$-representations $\sigma\otimes 1: \ker(\kappa_\tau) \stackrel{\sim}{\longrightarrow} \ker(\kappa_{\pi})$.

We now define a splitting $\widetilde{\sigma}:\tau\longrightarrow \pi$ as follows.  Let $v\in \tau$, and write $v = \sum_{i}g_i.v_i$, where $v_i\in \tau^{I_1}, g_i\in G$.  We then set $\widetilde{\sigma}(v) := \sum_i g_i.\sigma(v_i)$, so that $\widetilde{\sigma} = \kappa_\pi\circ(\sigma\otimes 1)\circ\kappa_\tau^{-1}$.  The paragraph above implies that this map is well-defined, and it is clearly $G$-equivariant.  Hence, our original short exact sequence splits.  
\end{proof}

The result above appears to be compatible (in a weak sense) with Serre weight conjectures.  We demonstrate an instance of this with an example.  Take $F = \qp, n = 3$, so that $G = \textnormal{GL}_3(\qp)$, and set $K := \textnormal{GL}_3(\zp)$.  Suppose that $\overline{\rho}:\textnormal{Gal}(\qpb/\qp) \longrightarrow \bB(\fpb) \subset \textnormal{GL}_3(\fpb)$ is a sufficiently generic, maximally nonsplit mod-$p$ Galois representation.  The conjectures of Breuil--Herzig--Schraen predict the existence of a smooth admissible finite-length $G$-representation $\Pi(\overline{\rho})$, such that $\Pi(\overline{\rho})$ contains as subquotients two nonsplit extensions of a supersingular representation by a principal series representation, and such that the representations appearing in the $K$-socle of $\Pi(\overline{\rho})^{\textnormal{ss}}$ are given by the set of Serre weights $\textnormal{W}^?(\overline{\rho}^{\textnormal{ss}})$ (see \cite[Definition 9.2.5]{geeherzigsavitt} for the definition of $\textnormal{W}^?$).

Suppose $\overline{\rho}$ satisfies
$$\overline{\rho} \cong \begin{pmatrix} \omega^a\textnormal{nr}_{\mu_1} & *_1 & * \\ 0 & \omega^b\textnormal{nr}_{\mu_2} & *_2 \\ 0 & 0 & \omega^c\textnormal{nr}_{\mu_3}\end{pmatrix},$$
with $*_1, *_2$ nonsplit extensions, and $2\leq a - b, b - c, a - c \leq p - 3$.  Here $\omega$ denotes the mod-$p$ cyclotomic character, and $\textnormal{nr}_\mu$ is the unramified character taking the value $\mu\in \fpb^\times$ on geometric Frobenius elements.  Let $\tau$ denote the (unique) supersingular representation appearing in $\Pi(\overline{\rho})^{\textnormal{ss}}$.  Conjecturally, we should have two nonsplit extensions 
$$0\longrightarrow \ind_B^G(\chi_1) \longrightarrow \pi_1 \longrightarrow \tau  \longrightarrow 0,$$
$$0\longrightarrow \ind_B^G(\chi_2) \longrightarrow \pi_2 \longrightarrow \tau  \longrightarrow 0,$$
where $\chi_1|_{T_0} = \overline{(c, a - 1, b - 2)}$ (resp. $\chi_2|_{T_0} = \overline{(b,c - 1,a - 2)}$) and 
$$\chi_1\left(\begin{pmatrix}p^r & 0 & 0 \\ 0 & p^s & 0 \\ 0 & 0 & p^t\end{pmatrix}\right) = \mu_3^r\mu_1^s\mu_2^t\qquad (\textnormal{resp.}~ \chi_2\left(\begin{pmatrix}p^r & 0 & 0 \\ 0 & p^s & 0 \\ 0 & 0 & p^t\end{pmatrix}\right) = \mu_2^r\mu_3^s\mu_1^t).$$  
Here $\overline{(a',b',c')}$ denotes the smooth character of $T_0$ obtained by composing the algebraic character $(a',b',c')\in X^*(\bT) \cong \bbZ^3$ with the reduction modulo $p$ (as in Subsubsection \ref{miscchar}).  Proposition \ref{neccond} above then implies that $\tau^{I_1}$ must have a quotient isomorphic to the 3-dimensional module $\fm_{\alpha_0,0,1}$ (resp. $\fm_{\alpha_0,0,2}$) defined as the supersingular part of $\textnormal{H}^1(I_1,\ind_B^G(\chi_1))$ (resp. $\textnormal{H}^1(I_1,\ind_B^G(\chi_2))$).  The element $\T_{\left(\begin{smallmatrix}x & 0 & 0 \\ 0 & y & 0 \\ 0 & 0 & z\end{smallmatrix}\right)}$, for $x,y,z\in \zp^\times$, acts on these supersingular quotients by 
$$\overline{x}^{1 - a}\overline{y}^{1 - c}\overline{z}^{1 - b},~\overline{x}^{1 - c}\overline{y}^{1 - b}\overline{z}^{1 - a},~\textnormal{and}~\overline{x}^{1 - b}\overline{y}^{1 - a}\overline{z}^{1 - c}.$$

On the other hand, by the Serre weight conjectures of Herzig and Gee--Herzig--Savitt (\cite[Conjecture 6.9]{herzig:serreconjduke} and \cite[Section 1 and Conjecture 7.2.7]{geeherzigsavitt}) the $K$-socle of $\tau$ should satisfy 
$$\textnormal{soc}_K(\tau) ~ \cong ~ F(a + p - 2, c + p - 2, b - 1)~ \oplus~ F(c + 2p - 3, b + p - 2, a - 1) ~\oplus~ F(b + p - 2, a - 1, c - 1),$$
where $F(a',b',c')$ is the algebraic representation of $\textnormal{GL}_3(\fpb)$ of highest weight $(a',b',c')\in X^*(\bT) \cong \bbZ^3$, viewed as a representation of $K$ via $K \longtwoheadrightarrow \textnormal{GL}_3(\bbF_p) \longhookrightarrow \textnormal{GL}_3(\fpb)$.  The action of $\T_{\left(\begin{smallmatrix}x & 0 & 0 \\ 0 & y & 0 \\ 0 & 0 & z\end{smallmatrix}\right)}$, for $x,y,z\in \zp^\times$, on the 3-dimensional space $\soc_K(\tau)^{I_1}$ is given by
$$\overline{x}^{1 - a}\overline{y}^{1 - c}\overline{z}^{1 - b},~\overline{x}^{1 - c}\overline{y}^{1 - b}\overline{z}^{1 - a},~\textnormal{and}~\overline{x}^{1 - b}\overline{y}^{1 - a}\overline{z}^{1 - c},$$
which exactly matches the above.  (In fact, we can upgrade this to an isomorphism $\textnormal{soc}_K(\tau)^{I_1}\cong \fm_{\alpha_0,0,1}$ (resp. $\textnormal{soc}_K(\tau)^{I_1}\cong \fm_{\alpha_0,0,2}$) of modules over the ``finite Hecke algebra'' $\cH_K := \End_K(\cind_{I_1}^K(C))$.)

\bigskip

\appendix

\section{Calculation of parabolic induction}\label{app}

We now carry out the proof of Proposition \ref{gr1} by computing the action of $\cH$ on the space $\fn_{F,\beta,r}\otimes_{\cH_{M_\beta}^+}\cH$.  In order to do this, we will use Vign\'eras' alcove walk Bernstein basis (\cite[Chapter 5]{vigneras:hecke1}).  We first recall the relevant facts.

\subsection{Preparation}

\subsubsection{} For $w\in \tW$ we let $E(w) := E_{o_{-\Pi}}(w)$ denote the alcove walk basis associated to the spherical orientation $o_{-\Pi}$.  We let $E^{M_\beta}(w)$ denote the analogously defined alcove walk basis of $\cH_{M_\beta}$.  In several places below, we will need to relate the bases $\{E(w)\}_{w\in \tW}$ and $\{\T_w\}_{w\in \tW}$.  Note first that if $w\in W_0$, then \cite[Example 5.32]{vigneras:hecke1} implies
$$E(\widehat{w}) = \T_{\widehat{w}}.$$
Next, if $w\in W_0$ and $\omega\in \widetilde{\Omega}$, then
\begin{equation}\label{bernsteinomega}
E(\widehat{w}\omega) = \T_{\widehat{w}\omega} = \T_{\widehat{w}}\T_\omega = E(\widehat{w})\T_\omega
\end{equation}
(cf. \emph{op. cit.}, Corollary 5.27).  Finally, if $w\in W_0$ and $\fat{\alpha} = (-\alpha_0,1)$ (where $\alpha_0$ is the highest root of $\Phi$), then applying the definitions in Chapter 5 of \emph{op. cit.} gives
\begin{equation}\label{alcovewalk}
\T_{\widehat{w}}\T_{\salpha} = \T_{\widehat{w}\salpha} =  \begin{cases}E(\widehat{w}\salpha) + \sum_{x\in k_F^\times}E(\widehat{w}\alpha_0^\vee([x])) & \textnormal{if}~w(\alpha_0)\in \Phi^+,\\ E(\widehat{w}\salpha) & \textnormal{if}~w(\alpha_0)\in \Phi^-.\end{cases}
\end{equation}

\subsubsection{}  Now let $w\in W_0$, $\lambda\in T$, and write $w = uv$, where $u\in \langle s_\beta\rangle$ and $v\in {}^{\beta}W_0$ (we use the notation ${}^{\beta}W_0$ for ${}^{\bM_{\beta}}W_0$).  Choose $\lambda_\beta^+\in T$ such that $\lambda_\beta^+$ is in the center of $\tW_{\bM_\beta}$, and such that $\langle \gamma,\nu(\lambda_\beta^+)\rangle < 0$ for $\gamma\in \Phi^+\smallsetminus\{\beta\}$.  Let $n\geq 0$ be such that $(\lambda_\beta^+)^n\lambda$ is $M_\beta$-positive.  Then, for $v_i\in \fn_{F,\beta,r}$, we have
\begin{eqnarray}
 v_i\otimes E(\lambda \widehat{w}) & \stackrel{\textnormal{\cite[Lem. 2.6(2)]{abe:inductions}}}{=} & v_i\cdot E^{M_\beta}(\lambda_\beta^+)^{-n}\otimes E((\lambda_\beta^+)^n)E(\lambda\widehat{w})\notag\\
 & \stackrel{\textnormal{\cite[Thm. 5.25]{vigneras:hecke1}}}{=} & q_{(\lambda_\beta^+)^n,\lambda\widehat{w}}\left(v_i\cdot E^{M_\beta}(\lambda_\beta^+)^{-n}\otimes E((\lambda_\beta^+)^n\lambda\widehat{u}\widehat{v})\right)\notag\\
 & \stackrel{\sub{\textnormal{\cite[Lem. 2.18(2)]{abe:inductions}},}{\textnormal{\cite[Lem. 5.34]{vigneras:hecke1}}}}{=} & q_{(\lambda_\beta^+)^n,\lambda\widehat{w}}\left(v_i\cdot E^{M_\beta}(\lambda_\beta^+)^{-n}\otimes E((\lambda_\beta^+)^n\lambda\widehat{u})E(\widehat{v})\right)\notag\\
 & \stackrel{\textnormal{\cite[Lem. 2.6(2)]{abe:inductions}}}{=} & q_{(\lambda_\beta^+)^n,\lambda\widehat{w}}\left(v_i\cdot \left(E^{M_\beta}(\lambda_\beta^+)^{-n}E^{M_\beta}((\lambda_\beta^+)^n\lambda\widehat{u})\right)\otimes E(\widehat{v})\right)\notag\\
 & = & q_{(\lambda_\beta^+)^n,\lambda\widehat{w}}\left(v_i\cdot E^{M_\beta}(\lambda\widehat{u})\otimes E(\widehat{v})\right)\notag\\
 & \stackrel{\textnormal{\cite[Ex. 5.32]{vigneras:hecke1}}}{=} & q_{(\lambda_\beta^+)^n,\lambda\widehat{w}}\left(v_i\cdot E^{M_\beta}(\lambda\widehat{u})\otimes \T_{\widehat{v}}\right). \label{bernsteincomp}
\end{eqnarray}
Here, $q_{w,w'} := q^{(\boldl(w) + \boldl(w') - \boldl(ww'))/2}$.

\subsubsection{} We need one more combinatorial fact.  Fix $w,v\in W_0$ and $\beta\in \Phi$.  In the sequel, we will need to understand the discrepancy between the conjugation actions of $\widehat{w}\widehat{v}$ and $\widehat{wv}$ on $\bU_{v^{-1}w^{-1}(\beta)}$.  Noting that $\widehat{w}\widehat{v}\widehat{wv}^{-1}\in T$, we have
\begin{eqnarray*}
u_\beta\left(d_{v,v^{-1}w^{-1}(\beta)}d_{w,w^{-1}(\beta)}x\right) & = & \widehat{w} u_{w^{-1}(\beta)}\left(d_{v,v^{-1}w^{-1}(\beta)}x\right)\widehat{w}^{-1}\\
 & = & \widehat{w}\widehat{v} u_{v^{-1}w^{-1}(\beta)}\left(x\right)\widehat{v}^{-1}\widehat{w}^{-1}\\
 & = & (\widehat{w}\widehat{v}\widehat{wv}^{-1})\widehat{wv} u_{v^{-1}w^{-1}(\beta)}\left(x\right)\widehat{wv}^{-1}(\widehat{w}\widehat{v}\widehat{wv}^{-1})^{-1}\\
 & = & (\widehat{w}\widehat{v}\widehat{wv}^{-1}) u_{\beta}\left(d_{wv,v^{-1}w^{-1}(\beta)}x\right)(\widehat{w}\widehat{v}\widehat{wv}^{-1})^{-1}\\
 & = & u_{\beta}\left(\beta(\widehat{w}\widehat{v}\widehat{wv}^{-1})d_{wv,v^{-1}w^{-1}(\beta)}x\right).
\end{eqnarray*}
This gives the cocycle relation
\begin{equation}\label{strcstcocycle}
d_{v,v^{-1}w^{-1}(\beta)}d_{w,w^{-1}(\beta)} = \beta(\widehat{w}\widehat{v}\widehat{wv}^{-1})d_{wv,v^{-1}w^{-1}(\beta)}.
\end{equation}
Furthermore, if $\beta\in \Pi$ and $v^{-1}w^{-1}(\beta)\in \Phi^-$, then 
$$d_{wv,v^{-1}w^{-1}(\beta)} = d_{\beta,-\beta}d_{s_\beta wv,v^{-1}w^{-1}s_\beta(\beta)} = -d_{s_\beta wv,v^{-1}w^{-1}s_\beta(\beta)}$$ 
(\cite[Lemma 9.2.2(ii)]{springer}), and we get
\begin{equation}\label{strcstcocycle1.5}
d_{v,v^{-1}w^{-1}(\beta)}d_{w,w^{-1}(\beta)} = -\beta(\widehat{w}\widehat{v}\widehat{wv}^{-1})d_{s_\beta wv,v^{-1}w^{-1}s_\beta(\beta)}.
\end{equation}
Similarly, using $\fs_{-\alpha_0}$ instead of $\widehat{v}$, we have
\begin{equation}\label{strcstcocycle2}
d_{-\alpha_0,s_{\alpha_0}w^{-1}(\beta)}d_{w,w^{-1}(\beta)} = \beta(\widehat{w}\fs_{-\alpha_0}\widehat{ws_{\alpha_0}}^{-1})d_{ws_{\alpha_0},s_{\alpha_0}w^{-1}(\beta)}.
\end{equation}
Furthermore, if $\beta\in \Pi$ and $s_{\alpha_0}w^{-1}(\beta)\in \Phi^-$, then 
$$d_{ws_{\alpha_0},s_{\alpha_0}w^{-1}(\beta)} = d_{\beta,-\beta}d_{s_\beta ws_{\alpha_0},s_{\alpha_0}w^{-1}s_\beta(\beta)} = -d_{s_\beta ws_{\alpha_0},s_{\alpha_0}w^{-1}s_\beta(\beta)},$$ 
and we get
\begin{equation}\label{strcstcocycle2.5}
d_{-\alpha_0,s_{\alpha_0}w^{-1}(\beta)}d_{w,w^{-1}(\beta)} = -\beta(\widehat{w}\fs_{-\alpha_0}\widehat{ws_{\alpha_0}}^{-1})d_{s_\beta ws_{\alpha_0},s_{\alpha_0}w^{-1}s_\beta(\beta)}.
\end{equation}

\subsection{Proof} We may now begin the proof.  By \cite[Remark 4.1]{olliviervigneras}, the space $\fn_{F,\beta,r}\otimes_{\cH_{M_\beta}^+}\cH$ has a basis given by $\{v_1\otimes \T_{\widehat{w}}, v_2\otimes\T_{\widehat{w}}\}_{w\in {}^{\beta}W_0}$.  We define a $C$-linear isomorphism by
\begin{eqnarray*}
\ff: \fn_{F,\beta,r}\otimes_{\cH_{M_\beta}^+}\cH & \longrightarrow & \bigoplus_{v\in W_0}C\bar{\eta}^v_{v^{-1}(\beta),r}\\
v_1\otimes\T_{\widehat{w}} & \longmapsto & d_{w,w^{-1}(\beta)}\cdot \bar{\eta}^w_{w^{-1}(\beta),r}\\
v_2\otimes\T_{\widehat{w}} & \longmapsto & d_{w,w^{-1}(\beta)}\cdot \bar{\eta}^{s_\beta w}_{w^{-1}s_\beta(\beta),r}\\
\end{eqnarray*}
We will prove that $\ff$ is $\cH$-equivariant.  For brevity, we denote the element $\bar{\eta}^w_{w^{-1}(\beta),r}$ by $\bar{\eta}^w$.

\begin{enumerate}[$\bullet$]
\item  Suppose first that $\alpha\in \Pi$ and $w\in {}^{\beta}W_0$, so that $w^{-1}(\beta)\in \Phi^+$.  If $s_{\alpha}w^{-1}(\beta)\in \Phi^-$, then $w^{-1}(\beta) = \alpha$.  Thus
\begin{eqnarray*}
(v_i\otimes\T_{\widehat{w}})\cdot\T_{\widehat{s_\alpha}} & \stackrel{\eqref{bruhateqs}}{=} & \begin{cases} v_i\otimes\T_{\widehat{ws_\alpha}} & (\textnormal{I}) \\ v_i\otimes\T_{\widehat{ws_\alpha}} & (\textnormal{II}) \\  v_i\otimes\T_{\widehat{ws_\alpha}}\T_{\widehat{s_\alpha}}^2 & (\textnormal{III}) \end{cases}\\
 & = & \begin{cases}v_i\otimes\T_{\widehat{ws_\alpha}} & \\ (v_i\cdot\T_{\widehat{s_\beta}}^{M_\beta})\otimes\T_{\widehat{s_\beta ws_\alpha}} & \\ v_i\otimes\T_{\widehat{w}}c_{\alpha} &\end{cases}\\
 & = & \begin{cases}v_i\otimes\T_{\widehat{ws_\alpha}} & \\ (v_i\cdot\T_{\widehat{s_\beta}}^{M_\beta})\otimes\T_{\widehat{w}} & \\ -\delta_{(\chi\overline{\beta}^{-p^r})^{s_\beta^{i - 1}}\circ w(\alpha^\vee),1}v_i\otimes\T_{\widehat{w}} & \end{cases} 
\end{eqnarray*}
where $w(\alpha)\in \Phi^+\smallsetminus\{\beta\}$ in case $(\textnormal{I})$; $w(\alpha) = \beta$ in case $(\textnormal{II})$; and $w(\alpha)\in \Phi^-$ in case $(\textnormal{III})$.  This gives
\begin{eqnarray*}
\ff\left((v_1\otimes\T_{\widehat{w}})\cdot\T_{\widehat{s_\alpha}}\right) & = & \ff\left(\begin{cases}v_1\otimes\T_{\widehat{ws_\alpha}} & (\textnormal{I}) \\ 0 & (\textnormal{II}) \\ -\delta_{(\chi\overline{\beta}^{-p^r})\circ w(\alpha^\vee),1}v_1\otimes\T_{\widehat{w}} & (\textnormal{III}) \end{cases}\right)\\
 & = & \begin{cases}d_{ws_\alpha,s_\alpha w^{-1}(\beta)}\cdot\bar{\eta}^{ws_\alpha} &  \\ 0 & \\ -\delta_{(\chi\overline{\beta}^{-p^r})\circ w(\alpha^\vee),1}d_{w,w^{-1}(\beta)}\cdot\bar{\eta}^w & \end{cases}\\
 & \stackrel{\eqref{defofstrcst}}{=} & d_{w,w^{-1}(\beta)}\begin{cases}d_{\alpha,s_\alpha w^{-1}(\beta)}\cdot\bar{\eta}^{ws_\alpha} & \\ 0 & \\ -\delta_{(\chi\overline{\beta}^{-p^r})\circ w(\alpha^\vee),1}\cdot\bar{\eta}^w & \end{cases}\\
 & \stackrel{\textnormal{Lem.}~\ref{TsalphaGr}}{=} & (d_{w,w^{-1}(\beta)}\cdot\bar{\eta}^w)\cdot \T_{\widehat{s_\alpha}}\\
 & = & \ff\left(v_1\otimes\T_{\widehat{w}}\right)\cdot \T_{\widehat{s_\alpha}}\\
 & & \\
 & & \\
\ff\left((v_2\otimes\T_{\widehat{w}})\cdot\T_{\widehat{s_\alpha}}\right) & = & \ff\left(\begin{cases}v_2\otimes\T_{\widehat{ws_\alpha}} & (\textnormal{I}) \\  -\chi\circ\beta^\vee(-1)\cdot\delta_{F,\qp}\cdot v_1\otimes\T_{\widehat{w}} & (\textnormal{II}) \\  \qquad - \delta_{(\chi\overline{\beta}^{-p^r})\circ\beta^\vee,1}\cdot v_2\otimes\T_{\widehat{w}} & \\ -\delta_{(\chi^{s_\beta}\overline{\beta}^{p^r})\circ w(\alpha^\vee),1}v_2\otimes\T_{\widehat{w}} & (\textnormal{III})\end{cases}\right)\\
 & = & \begin{cases}d_{ws_{\alpha},s_\alpha w^{-1}(\beta)}\cdot\bar{\eta}^{s_\beta ws_\alpha} & \\  -\chi\circ\beta^\vee(-1)\delta_{F,\qp}d_{w,w^{-1}(\beta)}\cdot \bar{\eta}^w & \\ \qquad - \delta_{(\chi\overline{\beta}^{-p^r})\circ\beta^\vee,1}d_{w,w^{-1}(\beta)}\cdot \bar{\eta}^{s_\beta w} & \\ -\delta_{(\chi^{s_\beta}\overline{\beta}^{p^r})\circ w(\alpha^\vee),1}d_{w,w^{-1}(\beta)}\bar{\eta}^{s_\beta w} & \end{cases}\\
 & \stackrel{\eqref{defofstrcst}}{=} & d_{w,w^{-1}(\beta)}\begin{cases}d_{\alpha,s_\alpha w^{-1}(\beta)}\cdot\bar{\eta}^{s_\beta ws_\alpha} & \\  -\chi\circ\beta^\vee(-1)\delta_{F,\qp}\cdot \bar{\eta}^w  & \\ \qquad - \delta_{(\chi\overline{\beta}^{-p^r})\circ\beta^\vee,1}\cdot \bar{\eta}^{s_\beta w} & \\ -\delta_{(\chi^{s_\beta}\overline{\beta}^{p^r})\circ w(\alpha^\vee),1}\bar{\eta}^{s_\beta w} & \end{cases}\\
 & \stackrel{\textnormal{Lem.}~\ref{TsalphaGr}}{=} & (d_{w,w^{-1}(\beta)}\cdot\bar{\eta}^{s_\beta w})\cdot \T_{\widehat{s_\alpha}}\\
 & = & \ff\left(v_2\otimes \T_{\widehat{w}}\right)\cdot\T_{\widehat{s_\alpha}}.
\end{eqnarray*}

\item  We now take $\fat{\alpha} = (-\alpha_0,1)$ and $w\in {}^{\beta}W_0$, and examine the element $\widehat{w}\salpha = w(\alpha_0^\vee)(\varpi^{-1})\widehat{w}\fs_{-\alpha_0}$.  By a length computation (using \cite[Ch. VI, \S 1, Proposition 25(iv)]{bourbaki:lie}), we have
\begin{equation}\label{length}
q_{(\lambda_\beta^+)^n,\widehat{w}\salpha} = 0~\Longleftrightarrow~ w(\alpha_0)\in \Phi^+\smallsetminus\{\beta\}.
\end{equation}
Consider now the equation
$$s_{\alpha_0}w^{-1}(\beta) = w^{-1}(\beta) - \langle w^{-1}(\beta),\alpha_0^\vee\rangle\alpha_0.$$
If $w\in {}^{\beta}W_0$ and $w(\alpha_0)\in \Phi^-$, we have $w^{-1}(\beta)\in \Phi^+\smallsetminus\{\alpha_0\}$, and therefore $\langle w^{-1}(\beta),\alpha_0^\vee\rangle \in \{0,1\}$ by \cite[Ch. VI, \S 1, Proposition 25(iv)]{bourbaki:lie}.  Hence, by the equation above, 
\begin{equation}\label{affineroot}
\begin{gathered}
s_{\alpha_0}w^{-1}(\beta)\in \Phi^+~ \Longrightarrow~ \langle\beta,w(\alpha_0^\vee)\rangle = 0,\\
s_{\alpha_0}w^{-1}(\beta)\in \Phi^-~ \Longrightarrow~ \langle\beta,w(\alpha_0^\vee)\rangle = 1.
\end{gathered}
\end{equation}

Thus
\begin{eqnarray*}
(v_i\otimes\T_{\widehat{w}})\cdot \T_{\salpha} & \stackrel{\eqref{alcovewalk}, \eqref{bernsteincomp}, \eqref{length}}{=} & \begin{cases} \sum_{x\in k_F^\times} v_i\cdot \T_{w(\alpha_0^\vee)([x])}^{M_\beta}\otimes\T_{\widehat{w}} & (\textnormal{I})\\ 
v_i\cdot E^{M_\beta}(\widehat{w}\salpha\widehat{w}^{-1})\otimes \T_{\widehat{w}} & (\textnormal{II})\\ 
\qquad\qquad  + \sum_{x\in k_F^\times} v_i\cdot \T_{\beta^\vee([x])}^{M_\beta}\otimes\T_{\widehat{w}} & \\ 
v_i\cdot E^{M_\beta}(\widehat{w}\salpha\widehat{ws_{\alpha_0}}^{-1})\otimes \T_{\widehat{ws_{\alpha_0}}} & (\textnormal{III}) \\ 
v_i\cdot E^{M_\beta}(\widehat{w}\salpha\widehat{ws_{\alpha_0}}^{-1}\widehat{s_\beta})\otimes \T_{\widehat{s_{\beta}ws_{\alpha_0}}} & (\textnormal{IV}) \end{cases}\\
 & \stackrel{\sub{\textnormal{\cite[Ex. 5.33]{vigneras:hecke1},}}{\eqref{defofstrcst}, \eqref{conjlift}}}{=} & \begin{cases} -\delta_{(\chi\overline{\beta}^{-p^r})^{s_\beta^{i - 1}}\circ w(\alpha_0^\vee),1} v_i\otimes \T_{\widehat{w}} & \\ 
 v_i\cdot \T^{M_\beta,*}_{\widehat{s_{\beta^*}}}\T^{M_\beta}_{\beta^\vee(d_{w,-\alpha_0})}\otimes \T_{\widehat{w}} & \\ 
 \qquad\qquad  -\delta_{(\chi\overline{\beta}^{-p^r})\circ\beta^\vee,1} v_i\otimes \T_{\widehat{w}}& \\ 
 v_i\cdot \T^{M_\beta}_{\widehat{w}\salpha\widehat{ws_{\alpha_0}}^{-1}}\otimes \T_{\widehat{ws_{\alpha_0}}} & \\
  v_i\cdot \T^{M_\beta}_{\widehat{w}\salpha\widehat{ws_{\alpha_0}}^{-1}\widehat{s_\beta}}\otimes \T_{\widehat{s_{\beta}ws_{\alpha_0}}} & \end{cases}
\end{eqnarray*}
where $w(\alpha_0)\in \Phi^+\smallsetminus\{\beta\}$ in case $(\textnormal{I})$; $w(\alpha_0) = \beta$ in case $(\textnormal{II})$; $w(\alpha_0)\in \Phi^-$ and $s_{\alpha_0}w^{-1}(\beta)\in \Phi^+$ in case $(\textnormal{III})$; and $w(\alpha_0)\in \Phi^-$ and $s_{\alpha_0}w^{-1}(\beta)\in \Phi^-$ in case $(\textnormal{IV})$.  This gives
\begin{eqnarray*}
 \ff\left((v_1\otimes \T_{\widehat{w}})\cdot\T_{\salpha}\right) & \stackrel{\eqref{affineroot}}{=} & \ff\left(\begin{cases} -\delta_{(\chi\overline{\beta}^{-p^r})\circ w(\alpha_0^\vee),1} v_1\otimes \T_{\widehat{w}} &(\textnormal{I})\\
  -\chi\circ\beta^\vee(-d_{w,-\alpha_0} \varpi)\cdot \delta_{F,\qp}v_2 \otimes \T_{\widehat{w}}& (\textnormal{II})\\ 
  \qquad\qquad  -\delta_{(\chi\overline{\beta}^{-p^r})\circ\beta^\vee,1} v_1\otimes \T_{\widehat{w}}  & \\ 
  (\chi\overline{\beta}^{-p^r})\left(\widehat{ws_{\alpha_0}}\salpha^{-1}\widehat{w}^{-1}\right)\cdot v_1\otimes \T_{\widehat{ws_{\alpha_0}}} &(\textnormal{III}) \\
    -\chi\left(\widehat{ws_{\alpha_0}}\salpha^{-1}\widehat{w}^{-1}\right) & (\textnormal{IV}) \\
     \qquad \cdot \overline{\beta\left(\widehat{w}\fs_{-\alpha_0}\widehat{ws_{\alpha_0}}^{-1}\right)}^{p^r}\cdot v_2 \otimes \T_{\widehat{s_{\beta}ws_{\alpha_0}}} & \end{cases}\right)\\
 & = & \begin{cases} -\delta_{(\chi\overline{\beta}^{-p^r})\circ w(\alpha_0^\vee),1} d_{w,w^{-1}(\beta)}\bar{\eta}^w &\\
  -\chi\circ\beta^\vee(-d_{w,-\alpha_0} \varpi)\delta_{F,\qp}d_{w,w^{-1}(\beta)}\cdot \bar{\eta}^{s_\beta w} & \\ 
  \qquad\qquad  -\delta_{(\chi\overline{\beta}^{-p^r})\circ\beta^\vee,1}d_{w,w^{-1}(\beta)}\cdot \bar{\eta}^w  & \\
   (\chi\overline{\beta}^{-p^r})\left(\widehat{ws_{\alpha_0}}\salpha^{-1}\widehat{w}^{-1}\right)d_{ws_{\alpha_0},s_{\alpha_0}w^{-1}(\beta)}\cdot \bar{\eta}^{ws_{\alpha_0}} & \\ 
    -\chi\left(\widehat{ws_{\alpha_0}}\salpha^{-1}\widehat{w}^{-1}\right) &  \\ 
    \qquad \cdot \overline{\beta\left(\widehat{w}\fs_{-\alpha_0}\widehat{ws_{\alpha_0}}^{-1}\right)}^{p^r}d_{s_\beta ws_{\alpha_0},s_{\alpha_0}w^{-1}s_\beta(\beta)}\cdot \bar{\eta}^{ws_{\alpha_0}} & \end{cases}\\
 & \stackrel{\eqref{strcstcocycle2},\eqref{strcstcocycle2.5}}{=} & d_{w,w^{-1}(\beta)}\begin{cases} -\delta_{(\chi\overline{\beta}^{-p^r})\circ w(\alpha_0^\vee),1}\bar{\eta}^w & (\textnormal{I}) \\ 
 -\chi\circ\beta^\vee(-d_{w,-\alpha_0} \varpi)\delta_{F,\qp}\cdot \bar{\eta}^{s_\beta w}& (\textnormal{II}) \\ 
 \qquad\qquad  -\delta_{(\chi\overline{\beta}^{-p^r})\circ\beta^\vee,1}\cdot \bar{\eta}^w  & \\
  \chi\left(\widehat{ws_{\alpha_0}}\salpha^{-1}\widehat{w}^{-1}\right)d_{-\alpha_0,s_{\alpha_0}w^{-1}(\beta)}\cdot \bar{\eta}^{ws_{\alpha_0}} & (\textnormal{III}) + (\textnormal{IV}) \end{cases}\\
 & \stackrel{\textnormal{Lem.}~\ref{TsalphaGr}}{=} & (d_{w,w^{-1}(\beta)}\cdot \bar{\eta}^{w})\cdot \T_{\salpha}\\
 & = & \ff\left(v_1\otimes \T_{\widehat{w}}\right)\cdot\T_{\salpha}\\
 & & \\
 & & \\
 \ff\left((v_2\otimes \T_{\widehat{w}})\cdot\T_{\salpha}\right) & \stackrel{\eqref{affineroot}}{=} & \ff\left(\begin{cases} -\delta_{(\chi^{s_\beta}\overline{\beta}^{p^r})\circ w(\alpha_0^\vee),1} v_2\otimes \T_{\widehat{w}} & (\textnormal{I})\\ 
 0 &(\textnormal{II}) \\  
 (\chi^{s_\beta}\overline{\beta}^{p^r})\left(\widehat{ws_{\alpha_0}}\salpha^{-1}\widehat{w}^{-1}\right)\cdot v_2\otimes \T_{\widehat{ws_{\alpha_0}}} &(\textnormal{III}) \\ 
 -\chi\left(\widehat{s_\beta ws_{\alpha_0}}\salpha^{-1}\widehat{s_\beta w}^{-1}\right) & (\textnormal{IV}) \\
  \qquad \cdot \overline{\beta\left(\widehat{w}\fs_{-\alpha_0}\widehat{ws_{\alpha_0}}^{-1}\right)}^{-p^r}\cdot v_1\otimes \T_{\widehat{s_{\beta}ws_{\alpha_0}}} & \end{cases}\right)\\
& = & \begin{cases} -\delta_{(\chi^{s_\beta}\overline{\beta}^{p^r})\circ w(\alpha_0^\vee),1} d_{w,w^{-1}(\beta)}\cdot\bar{\eta}^{s_\beta w} & \\
 0 & \\ 
  (\chi^{s_\beta}\overline{\beta}^{p^r})\left(\widehat{ws_{\alpha_0}}\salpha^{-1}\widehat{w}^{-1}\right)d_{ws_{\alpha_0},s_{\alpha_0}w^{-1}(\beta)}\cdot \bar{\eta}^{s_\beta ws_{\alpha_0}} &  \\
   -\chi\left(\widehat{s_\beta ws_{\alpha_0}}\salpha^{-1}\widehat{s_\beta w}^{-1}\right) & \\ 
   \qquad \cdot \overline{\beta\left(\widehat{w}\fs_{-\alpha_0}\widehat{ws_{\alpha_0}}^{-1}\right)}^{-p^r}d_{s_\beta ws_{\alpha_0},s_{\alpha_0}w^{-1}s_\beta(\beta)}\cdot \bar{\eta}^{s_\beta ws_{\alpha_0}} & \end{cases}\\
& \stackrel{\eqref{bruhateqs}, \eqref{strcstcocycle2},\eqref{strcstcocycle2.5}}{=} & d_{w,w^{-1}(\beta)}\begin{cases} -\delta_{(\chi^{s_\beta}\overline{\beta}^{p^r})\circ w(\alpha_0^\vee),1} \cdot\bar{\eta}^{s_\beta w} & (\textnormal{I}) \\
 0 & (\textnormal{II})\\  
 \chi\left(\widehat{s_\beta ws_{\alpha_0}}\salpha^{-1}\widehat{s_\beta w}^{-1}\right) & (\textnormal{III}) + (\textnormal{IV})\\
  \qquad \cdot d_{-\alpha_0,s_{\alpha_0}w^{-1}s_\beta(\beta)}\cdot \bar{\eta}^{s_\beta ws_{\alpha_0}} & \end{cases}\\
& \stackrel{\textnormal{Lem.}~\ref{TsalphaGr}}{=} & (d_{w,w^{-1}(\beta)}\cdot \bar{\eta}^{s_\beta w})\cdot \T_{\salpha}\\
& = & \ff\left(v_w\otimes\T_{\widehat{w}}\right)\cdot \T_{\salpha}.
\end{eqnarray*}

\item  Suppose $t\in T_0$ or $t\in Z$.  Checking that the map $\ff$ is equivariant for the operators $\T_t$ is a straightforward exercise which follows directly from the definitions, and is left to the reader.

\item Finally, let $\omega\in \widetilde{\Omega}$.  Adjusting by an element of $T_0$, we may assume that $\omega\widehat{\bo}^{-1} = \lambda(\varpi)$ for some $\lambda\in X_*(\bT)$.  We may furthermore assume that $\omega\not\in Z$.  Therefore, by projecting $\omega$ to the adjoint group $\bG_{\textnormal{ad}}(F)$ and using the description of $\Omega_{\textnormal{ad}}$ found in \cite[Proposition 1.18]{iwahorimatsumoto}, we have that
$$\omega = \lambda_{\alpha}(\varpi)^{-1}\widehat{w_{(\alpha),\circ}w_\circ},$$
for some $\alpha\in \Pi$, where $w_{(\alpha),\circ}$ is the longest element of $W_{\Pi\smallsetminus\{\alpha\},0}$, and $\lambda_\alpha\in X_*(\bT)$ is a cocharacter satisfying $\langle\gamma,\lambda_\alpha\rangle = \delta_{\gamma,\alpha}$ for $\gamma\in \Pi$ and $\langle\alpha_0,\lambda_\alpha\rangle = 1$.  In particular, we have $\bo = w_{(\alpha),\circ}w_\circ$.

Take $w\in {}^{\beta}W_0$, so that $w^{-1}(\beta) \in \Phi^+$.  By expanding $w^{-1}(\beta)$ as a linear combination of simple roots and examining the action of $\bo$ on $\Pi\sqcup\{-\alpha_0\}$, we obtain
\begin{equation}\label{omegapairing}
\begin{gathered}
\bo^{-1}w^{-1}(\beta)\in \Phi^+ \Longleftrightarrow  \big\langle\beta,\nu(\widehat{w}\omega\widehat{w\bo}^{-1})\big\rangle =  \left\langle w^{-1}(\beta),\lambda_\alpha\right\rangle = 0,\\
\bo^{-1}w^{-1}(\beta)\in \Phi^- \Longleftrightarrow  \big\langle\beta,\nu(\widehat{w}\omega\widehat{w\bo}^{-1})\big\rangle =  \left\langle w^{-1}(\beta),\lambda_\alpha\right\rangle = 1.
\end{gathered}
\end{equation}

Now, we have
\begin{eqnarray*}
(v_i\otimes \T_{\widehat{w}})\cdot \T_{\omega} & = & v_i\otimes \T_{\widehat{w}\omega}\\
 & \stackrel{\eqref{bernsteinomega}}{=} & v_i\otimes E(\widehat{w}\omega)\\
 & \stackrel{\eqref{bernsteincomp}}{=} & \begin{cases} v_i\cdot E^{M_\beta}(\widehat{w}\omega\widehat{w\bo}^{-1})\otimes \T_{\widehat{w\bo}} & (\textnormal{I})\\ 
 v_i\cdot E^{M_\beta}(\widehat{w}\omega\widehat{w\bo}^{-1}\widehat{s_\beta})\otimes \T_{\widehat{s_\beta w\bo}} & (\textnormal{II}) \end{cases}\\
 & = & \begin{cases} v_i\cdot \T_{\widehat{w}\omega\widehat{w\bo}^{-1}}^{M_\beta}\otimes \T_{\widehat{w\bo}} & \\
  v_i\cdot \T_{\widehat{w}\omega\widehat{w\bo}^{-1}\widehat{s_\beta}}^{M_\beta}\otimes \T_{\widehat{s_\beta w\bo}} &  \end{cases}
\end{eqnarray*}
where $\bo^{-1}w^{-1}(\beta)\in \Phi^+$ in case $(\textnormal{I})$ and $\bo^{-1}w^{-1}(\beta)\in \Phi^-$ in case $(\textnormal{II})$.  The equality which is labeled by \eqref{bernsteincomp} uses the fact that the operator $\T_{\omega}$ is invertible (and therefore $q_{(\lambda_\beta^+)^n,\widehat{w}\omega} = 1$).  This gives
\begin{eqnarray*}
\ff\left((v_1\otimes \T_{\widehat{w}})\cdot\T_\omega\right) & \stackrel{\eqref{omegapairing}}{=} &  \ff\left(\begin{cases} \chi\left(\widehat{w\bo}\omega^{-1}\widehat{w}^{-1}\right)\overline{\beta\left(\widehat{w\bo}\omega^{-1}\widehat{w}^{-1}\right)}^{-p^r}\cdot v_1\otimes \T_{\widehat{w\bo}} & (\textnormal{I}) \\
  -\chi\left(\widehat{w\bo}\omega^{-1}\widehat{w}^{-1}\right)\overline{\beta\left(\widehat{w\bo}\bo^{-1}\widehat{w}^{-1}\right)}^{-p^r}\cdot v_{2}\otimes \T_{\widehat{s_\beta w\bo}} & (\textnormal{II}) \end{cases}\right)\\
 & = & \begin{cases} \chi\left(\widehat{w\bo}\omega^{-1}\widehat{w}^{-1}\right)\overline{\beta\left(\widehat{w\bo}\omega^{-1}\widehat{w}^{-1}\right)}^{-p^r}d_{w\bo,\bo^{-1}w^{-1}(\beta)}\cdot \bar{\eta}^{w\bo} & \\ -\chi\left(\widehat{w\bo}\omega^{-1}\widehat{w}^{-1}\right)\overline{\beta\left(\widehat{w\bo}\bo^{-1}\widehat{w}^{-1}\right)}^{-p^r}d_{s_\beta w\bo,\bo^{-1}w^{-1}s_\beta(\beta)}\cdot \bar{\eta}^{w\bo} &  \end{cases}\\
 & \stackrel{\eqref{strcstcocycle},\eqref{strcstcocycle1.5}}{=} & d_{\bo,\bo^{-1}w^{-1}(\beta)}d_{w,w^{-1}(\beta)}\chi\left(\widehat{w\bo}\omega^{-1}\widehat{w}^{-1}\right)\cdot \bar{\eta}^{w\bo}\\
 & \stackrel{\textnormal{Lem.}~\ref{TomegaGr}}{=} & (d_{w,w^{-1}(\beta)}\cdot \bar{\eta}^w)\cdot \T_{\omega}\\
 & = & \ff\left(v_1\otimes\T_{\widehat{w}}\right)\cdot\T_{\omega}\\
 & & \\
 & & \\
 \ff\left((v_2\otimes\T_{\widehat{w}})\cdot\T_{\omega}\right) & \stackrel{\eqref{omegapairing}}{=} & \ff\left(\begin{cases} \chi^{s_\beta}\left(\widehat{w\bo}\omega^{-1}\widehat{w}^{-1}\right)\overline{\beta\left(\widehat{w\bo}\omega^{-1}\widehat{w}^{-1}\right)}^{p^r}\cdot v_2\otimes \T_{\widehat{w\bo}} & (\textnormal{I}) \\
  -\chi^{s_\beta}\left(\widehat{w\bo}\omega^{-1}\widehat{w}^{-1}\beta^\vee(-1)\right)\overline{\beta\left(\widehat{w\bo}\bo^{-1}\widehat{w}^{-1}\right)}^{p^r}\cdot v_1\otimes \T_{\widehat{s_\beta w\bo}} & (\textnormal{II})\end{cases}\right)\\
 & \stackrel{\eqref{bruhateqs}}{=} &  \begin{cases} \chi\left(\widehat{s_\beta w\bo}\omega^{-1}\widehat{s_\beta w}^{-1}\right)\overline{\beta\left(\widehat{w\bo}\omega^{-1}\widehat{w}^{-1}\right)}^{p^r}d_{w\bo, \bo^{-1}w^{-1}(\beta)}\cdot \bar{\eta}^{s_\beta w\bo} & \\
  -\chi\left(\widehat{s_\beta w\bo}\omega^{-1}\widehat{s_\beta w}^{-1}\right)\overline{\beta\left(\widehat{w\bo}\bo^{-1}\widehat{w}^{-1}\right)}^{p^r}d_{s_\beta w\bo, \bo^{-1}w^{-1}s_\beta(\beta)}\cdot \bar{\eta}^{s_\beta w\bo} &  \end{cases}\\
 & \stackrel{\eqref{strcstcocycle},\eqref{strcstcocycle1.5}}{=} & d_{\bo,\bo^{-1}w^{-1}(\beta)}d_{w,w^{-1}(\beta)}\chi\left(\widehat{s_\beta w\bo}\omega^{-1}\widehat{s_\beta w}^{-1}\right)\cdot \bar{\eta}^{s_\beta w\bo}\\
 & \stackrel{\textnormal{Lem.}~\ref{TomegaGr}}{=} & (d_{w,w^{-1}(\beta)}\cdot \bar{\eta}^{s_\beta w})\cdot \T_{\omega}\\
 & = & \ff\left(v_2\otimes \T_{\widehat{w}}\right)\cdot \T_{\omega}.
\end{eqnarray*}

\end{enumerate}

\bibliographystyle{amsalpha}
\bibliography{refs}

\end{document}